\documentclass[11pt]{article}
\usepackage[T1]{fontenc}
\usepackage{bbm}
\usepackage{bm}
\usepackage[margin=1in]{geometry}
\usepackage{amsfonts}
\usepackage{amsmath}
\usepackage{amsthm}
\usepackage{amssymb}
\usepackage{enumitem}
\usepackage{mathtools}
\usepackage{tikz}
\usepackage[square,numbers,sort]{natbib}
\usepackage[colorlinks,linkcolor=blue,citecolor=blue]{hyperref}

\allowdisplaybreaks[1]
\numberwithin{equation}{section}

\theoremstyle{plain}
\newtheorem{theorem}{Theorem}[section]
\newtheorem{lemma}[theorem]{Lemma}
\newtheorem{proposition}[theorem]{Proposition}
\newtheorem{claim}[theorem]{Claim}
\newtheorem{corollary}[theorem]{Corollary}
\newtheorem{conjecture}[theorem]{Conjecture}

\theoremstyle{definition}
\newtheorem{definition}[theorem]{Definition}

\theoremstyle{remark}
\newtheorem{remark}[theorem]{Remark}

\newlist{thenum}{enumerate}{1}
\setlist[thenum, 1]{label=(\alph*), ref=\thetheorem(\alph*)}

\let\P\relax
\DeclareMathOperator{\P}{\mathbb{P}}
\DeclareMathOperator{\E}{\mathbb{E}}

\DeclareMathOperator{\Tr}{Tr}
\DeclareMathOperator{\diag}{diag}
\DeclareMathOperator{\Var}{\mathbb{V}ar}
\newcommand*{\0}{\bm{0}}
\newcommand*{\1}{\bm{1}}
\newcommand*{\inner}[2]{\left\langle #1, #2 \right\rangle}
\providecommand{\abs}[1]{\lvert#1\rvert}
\providecommand{\norm}[1]{\lVert#1\rVert}

\newcommand*{\wt}{\widetilde}
\newcommand*{\ind}[1]{\operatorname{\mathbbm{1}}\{#1\}}
\DeclareMathOperator*{\bb1}{\mathbbm{1}}
\newcommand*{\tv}[2]{\operatorname{TV}(#1, #2)}
\newcommand*{\kl}[2]{\operatorname{KL}(#1\parallel#2)}
\newcommand*{\Df}[2]{\operatorname{D}_f(#1\parallel#2)}

\DeclareMathOperator{\TV}{TV}

\newcommand*{\RR}{\mathbb{R}}
\renewcommand*{\SS}{\mathbb{S}}

\newcommand*{\cF}{\mathcal{F}}
\newcommand*{\cG}{\mathcal{G}}
\newcommand*{\cN}{\mathcal{N}}
\newcommand*{\cP}{\mathcal{P}}
\newcommand*{\cQ}{\mathcal{Q}}
\newcommand*{\cW}{\mathcal{W}}

\newcommand*{\dG}{\dot{\mathcal{G}}}

\renewcommand*{\b}{\bm{b}}
\newcommand*{\e}{\bm{e}}
\newcommand*{\h}{\bm{h}}
\newcommand*{\m}{\bm{m}}
\renewcommand*{\v}{\bm{v}}

\newcommand*{\x}{\bm{x}}
\newcommand*{\y}{\bm{y}}
\newcommand*{\z}{\bm{z}}

\newcommand*{\A}{\bm{A}}
\newcommand*{\B}{\bm{B}}
\newcommand*{\D}{\bm{D}}
\let\Hu\H
\renewcommand*{\H}{\bm{H}}
\newcommand*{\I}{\bm{I}}
\newcommand*{\K}{\bm{K}}
\newcommand*{\M}{\bm{M}}
\renewcommand*{\O}{\bm{O}}
\newcommand*{\R}{\bm{R}}
\renewcommand*{\S}{\bm{S}}
\newcommand*{\W}{\bm{W}}
\newcommand*{\V}{\bm{V}}
\newcommand*{\X}{\bm{X}}
\newcommand*{\Y}{\bm{Y}}
\newcommand*{\Z}{\bm{Z}}

\newcommand*{\bP}{\bm{P}}
\newcommand*{\bQ}{\bm{Q}}

\newcommand*{\bmu}{\bm{\mu}}

\newcommand*{\bPi}{\bm{\varPi}}
\newcommand*{\bSigma}{\bm{\varSigma}}

\newcommand*{\rB}{\mathrm{B}}

\newcommand*{\bra}{\overline{a}}
\newcommand*{\brs}{\overline{s}}

\newcommand*{\ER}{Erd\Hu{o}s-R\'enyi}

\begin{document}
\title{Phase transition in noisy high-dimensional\\
random geometric graphs}

\author{Suqi Liu\thanks{Princeton University, Princeton, NJ 08544, USA;
Email: \texttt{\{suqil, mracz\}@princeton.edu}.
Research supported in part by NSF grant DMS-1811724.
}
\and
Mikl\'os Z.\ R\'acz\footnotemark[1]}

\maketitle
\begin{abstract}
We study the problem of detecting latent geometric structure
in random graphs.
To this end, we consider the soft high-dimensional random geometric graph
$\mathcal{G}(n,p,d,q)$,
where each of the $n$ vertices corresponds to an independent random
point distributed uniformly on the sphere~$\mathbb{S}^{d-1}$,
and the probability that two vertices are connected by an edge is a decreasing
function of the Euclidean distance between the points.
The probability of connection is parametrized by $q \in [0,1]$,
with smaller $q$ corresponding to weaker dependence on the geometry;
this can also be interpreted as the level of noise in the geometric graph.
In particular,
the model smoothly interpolates between
the spherical hard random geometric graph $\mathcal{G}(n,p,d)$
(corresponding to $q = 1$) and the Erd\Hu{o}s-R\'enyi model $\mathcal{G}(n,p)$
(corresponding to $q = 0$).
We focus on the dense regime (i.e., $p$ is a constant).

We show that if $nq \to 0$ or $d \gg n^{3} q^{2}$,
then geometry is lost:
$\mathcal{G}(n,p,d,q)$ is asymptotically indistinguishable from
$\mathcal{G}(n,p)$.
On the other hand, if $d \ll n^{3} q^{6}$,
then the signed triangle statistic provides an asymptotically powerful test for
detecting geometry.
These results generalize those of Bubeck, Ding, Eldan, and R\'acz (2016) for
$\mathcal{G}(n,p,d)$,
and give quantitative bounds on how the noise level affects the dimension
threshold for losing geometry.
We also prove analogous results under a related but different distributional
assumption,
and we further explore generalizations of signed triangles in order to
understand the intermediate regime left open by our results.
\end{abstract}

\section{Introduction} \label{se:intro}

Random graphs emerge as canonical models for many real-life applications,
including social networks, wireless communications,
and in the biological sciences.
Among them, the simplest yet structurally rich model is the \ER{} random graph,
which has been studied extensively
(e.g., \cite{erdos59on,bollobas1976cliques}).
However, in many scenarios the independence of edges is an oversimplified
assumption and
is often insufficient to capture the subtle relations in complex networks.

A natural extension is to assume an underlying geometric structure.
The graph is then generated according to some dependency over this structure.
Due to their wide applicability,
random graphs of this kind have various incarnations
in different fields: random geometric graphs, latent space models,
spatial networks, random connection models, to name a few.
We refer the reader to~\cite{penrose2003random} for a comprehensive theoretical
treatment of the subject.

In real-world networks, the geometric space is often latent, not visible.
Most of the time,
only the graph structure is observable rather than the latent
variables that generated the graph.
For example, in a communication network,
we can only observe the connections but not the reasons behind links.
This brings up the natural question of understanding the extent to which
a latent space model is an accurate description.
As a first step, it is crucial to understand when it is possible to
even detect the presence of geometry,
which is necessary for justifying
the underlying geometric structure assuming a generating process.
Mathematically this was first studied by Devroye, Gy\"{o}rgy, Lugosi,
and Udina~\cite{devroye2011high} for
a particular random geometric graph equipped with a spherical geometry.
They showed that this random geometric graph
becomes indistinguishable from an \ER{} graph when the dimension of the
sphere goes to infinity.
In other words, geometry is lost in high dimensions.
Subsequently, Bubeck, Ding, Eldan, and R\'acz~\cite{bubeck2016testing}
pinpointed the phase transition
for testing high-dimensional geometry in dense random graphs.
Our paper builds upon and generalizes this result.

A caveat of the aforementioned results is that the model is restricted to
a ``hard geometry'' setting, where the existence of an edge is a deterministic
function of the distance between the latent variables corresponding to the two
vertices.
This assumption overlooks the fact that in reality connections often bear
a stochastic dependence on the latent variables.
Consequently, the phase transition in the hard geometry setting happens at
dimensions as high as the cube of the number of vertices,
seemingly much larger than what many high-dimensional statistics theories
would consider~\cite{wainwright2019high}.

Our focus in this paper is to understand the above question in the setting of
\emph{soft} random geometric graphs,
in which the softness can be viewed as noise in the geometric graph.
We are particularly interested in the interplay between
dimensionality and noise in affecting the phenomenon of losing geometry in
random graphs.
To this end,
we study a particular type of soft random geometric graph where there is a
parameter $q \in [0,1]$ that naturally reflects the level of noise,
corresponding to the strength of geometry.
This model is an interpolation between the hard spherical random geometric graph
(corresponding to $q=1$) and the \ER{} model (corresponding to $q=0$).
Our main results provide bounds,
as a function of both dimension and geometry strength,
of where the phase transition lies.
In particular, these quantitatively demonstrate the qualitative phenomenon that
the dimension threshold for losing geometry is smaller for soft random
geometric graphs and decreases as a function of geometry strength.
We next specify the precise setting of our work,
before describing our main results.

\subsection{Random geometric graphs} \label{se:mod}
We first describe the spherical random geometric graph model $\cG(n, p, d)$
\cite{devroye2011high,bubeck2016testing}.
For a set of $n$ vertices $V = [n] \coloneqq \{1,2,\ldots,n\}$,
associate each vertex $i$ with a point represented
by a $d$-dimensional random vector $\x_i$.
We assume that $\x_1, \ldots, \x_n$ are independently uniformly distributed on
the sphere $\mathbb{S}^{d-1} \coloneqq \{\x \in \RR^d: \norm{\x} = 1\}$,
where $\norm{\cdot}$ stands for the Euclidean norm.
For a fixed value of $p \in [0, 1]$, referred to as the edge probability,
the edge set of the graph
is defined as follows:
There is an undirected edge
between distinct vertices $i$ and $j$ if and only if
\begin{equation}\label{eq:dot_product}
\inner{\x_i}{\x_j} \ge t_{p,d},
\end{equation}
where $\inner{\cdot}{\cdot}$ denotes the inner product of two vectors.
Equivalently,
$i$ and $j$ are connected by an edge if and only if their Euclidean
distance satisfies $\norm{\x_i - \x_j} \le \sqrt{2(1 - t_{p,d})}$.
The threshold $t_{p,d}$,
which may depend on $p$ and $d$,
is determined by the equation
\begin{equation*}
\P(\inner{\x_i}{\x_j} \ge t_{p,d}) = p,
\end{equation*}
so that the probability of an edge existing between any pair of distinct
vertices is $p$.
Conditioning on the vectors $\x_1, \ldots, \x_n$,
the edges in the graph are deterministic.
The only source of randomness in $\cG(n, p, d)$ comes from the random points.
For this reason,
the random graph defined above is referred to as a \emph{hard}
random geometric graph.

More generally, the model above may be extended by
adding additional randomness to the edge generation process,
given the latent positions.
That is,
for any \emph{connection function}
$\phi: \RR \to [0, 1]$,
we may connect $i$ and $j$ with probability
$\phi(\inner{\x_i}{\x_j})$.
Formally,
let $i \sim j$ denote the event that there exists an undirected edge between
$i$ and $j$.
Then,
\begin{equation}\label{eq:edge}
\P(i \sim j \mid \x_1, \ldots, \x_n)
= \phi(\inner{\x_i}{\x_j}).
\end{equation}
This is equivalent to connecting $i$ and $j$ with probability
$\phi(1 - \norm{\x_i - \x_j}^2/2)$.
When the connection function is an indicator function
$\phi(x) = \ind{x \ge t_{p,d}}$,
we obtain the hard random geometric graph defined previously.
For general connection functions,
which are typically nondecreasing,\footnote{In the literature,
connection functions often take as their argument the distance
$\norm{\x_i - \x_j}$ and hence are a nonincreasing function.
Here it is more convenient to take the inner product $\inner{\x_i}{\x_j}$
as the argument of $\phi$ and hence this is a nondecreasing function.}
such random graphs are referred to as \emph{soft} random geometric graphs.

Denote by $\K = [k_{i,j}]$ the \emph{connection matrix}
defined by $k_{i,j} \coloneqq \phi(\inner{\x_i}{\x_j})$.
For a simple graph $G = (V, E)$, where $V$ is the set of vertices and
$E$ is the set of edges,
let $\A = [a_{i,j}]$ be its adjacency matrix,
where $a_{i,j} = 1$ if $i \sim j$ and $a_{i,j} = 0$ otherwise
for all $i,j \in V$.
Denote by $\X \in \RR^{n \times d}$ the matrix
whose rows are the random vectors $\x_1, \ldots, \x_n$.
For a soft random geometric graph with connection matrix $\K$,
conditioning on $\X$,
$\ind{i \sim j}$ is an independent Bernoulli random variable with
parameter $k_{i,j}$.
The distribution of the soft random geometric graph is then specified by
\begin{equation}\label{eq:P_srgg}
\P(G)
= \E_{\X}\biggl[\prod_{i < j}
{k_{i,j}}^{a_{i,j}}{(1 - k_{i,j})}^{1 - a_{i,j}}\biggr].
\end{equation}

We focus on a particular family of soft random geometric graphs
which are parametrized by the strength of the underlying geometry.
The connection function in these models is a linear interpolation between
a constant $p \in [0, 1]$ and a step function $s_t: \RR \to \{0, 1\}$ defined as
$s_t(x) := \ind{x \ge t}$.
Concretely, we consider the following connection function with parameter
$q \in [0, 1]$:
\begin{equation} \label{eq:link}
\phi_q(x) \coloneqq (1 - q) p + q s_t(x),
\end{equation}
where $q$ controls the strength of the geometry.
The threshold $t$ is similarly determined by forcing the edge probability
to be $p$:
\begin{equation*}
\P(i \sim j) = \E[\phi_q(\inner{\x_i}{\x_j})] = p,
\end{equation*}
which gives $t = t_{p,d}$ equal to the threshold in the definition
of $\cG(n, p, d)$.
We denote by $\cG(n, p, d, q)$ the soft random geometric graph equipped with
the connection function $\phi_q$.
When $q = 1$, the graph becomes the hard random geometric graph $\cG(n, p, d)$.
When $q = 0$,
each edge of the graph is generated independently
with probability $p$,
corresponding to the well-known \ER{} graph $\cG(n, p)$
that does not possess geometry.
As an illustration,
$\phi_q(x)$ is plotted against the connection functions of $\cG(n, p)$ and
$\cG(n, p, d)$ in Figure~\ref{fg:link}.
\begin{figure}[ht]
\begin{center}
\begin{tikzpicture}[scale=2]
\draw[->] (-2,0) -- (2,0) node[below right] {$x$};
\draw[->] (0,0) node[below] {$0$} -- (0,1.5) node[left] {$\phi(x)$};
\draw[blue,thick] (-2,0.4) node[below]
{$\phi_q(x)\coloneqq(1 - q) p + q s_t(x)$} -- (1,0.4);
\draw[blue,fill=white] (1,0.4) circle (0.03);
\draw[blue,thick] (1,0.8) -- (2,0.8);
\draw[blue,fill] (1,0.8) circle (0.03);
\draw[red,thick] (-2,0.6) node[above] {$p$} -- (2,0.6);
\draw[green,thick] (-2,0) -- (1,0);
\draw[green,fill=white] (1,0) circle (0.03);
\draw[green,thick] (1,1) -- (2,1) node[above]
{$s_{t}(x)$};
\draw[green,fill] (1,1) circle (0.03);
\draw[dashed] (1,0) node[below] {$t$} -- (1,1);
\draw[dashed] (0,1) node[left] {$1$} -- (1,1);
\end{tikzpicture}
\end{center}
\caption{A comparison of connection functions.} \label{fg:link}
\end{figure}
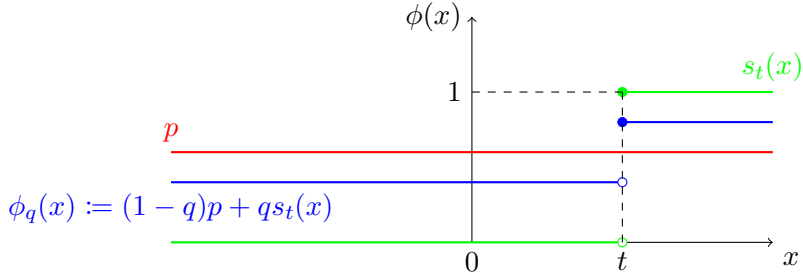

Another view of $\cG(n,p,d,q)$ is as a hard random geometric graph
with independently resampled edges.
Starting with a sample graph from $\cG(n,p,d)$,
for each pair of vertices in the graph,
we flip a biased coin with head probability $q$.
If the coin shows head,
we keep the edge/non-edge between them;
otherwise, we resample the connection,
creating an edge independently with probability $p$.

\subsection{Main results} \label{se:phase}
We are interested in whether it is possible to detect the underlying geometric
structure, which we formulate as the following hypothesis testing problem.
The null hypothesis is that the observed graph $G$ is
a sample from the \ER{} model with probability $p$:
\begin{equation*}
H_0: G \sim \cG(n, p).
\end{equation*}
The alternative hypothesis is that the graph is a soft random geometric graph
with dimension $d$ and geometry strength $q$:
\begin{equation*}
H_1: G \sim \cG(n, p, d, q).
\end{equation*}

The hypothesis testing problem can be understood through the guarantees for
the total variation distance between the two distributions.
Our findings are summarized in the following theorem.
\begin{theorem}[Detecting geometry]\label{th:det}
Let $p \in (0, 1)$ be fixed.
\begin{thenum}
\item\label{th:det_lower} (Impossibility)
If $n q \to 0$ or $n^3 q^2 / d \to 0$,
then
\begin{equation*}
\TV(\cG(n, p),\cG(n,p,d,q)) \to 0.
\end{equation*}
\item\label{th:det_upper} (Possibility)
If $n^3 q^6 / d \to \infty$, then
\begin{equation*}
\TV(\cG(n, p),\cG(n,p,d,q)) \to 1.
\end{equation*}
\end{thenum}
\end{theorem}

Theorem~\ref{th:det_lower} specifies a lower bound for detection:
If $n q \to 0$ or $d \gg n^{3} q^{2}$, then no test can detect geometry;
$\cG(n, p, d, q)$ is asymptotically indistinguishable from $\cG(n, p)$.
On the other hand,
Theorem~\ref{th:det_upper} provides an upper bound:
If $d \ll n^{3} q^{6}$,
then there exists an asymptotically powerful test for detecting geometry.
Specifically, we will show that the signed triangle statistic of
Bubeck et al.~\cite{bubeck2016testing} (which in particular is computationally
efficient) works in this regime to distinguish the two models.

Recall that $\cG(n, p, d, q)$ becomes $\cG(n, p, d)$ in the special case
when $q = 1$.
In this case Theorem~\ref{th:det} recovers the results of
Bubeck et al.~\cite{bubeck2016testing},
showing that $d \asymp n^{3}$ is the dimension threshold for losing geometry.

In general,
both the impossibility and possibility results in Theorem~\ref{th:det} depict
a polynomial dependency on~$q$.
However, the polynomials have different powers of $q$,
which means there exists a gap between the lower and upper bounds.
We believe that the phase transition for losing geometry in $\cG(n, p, d, q)$
happens at a certain power of $q$;
currently we do not have a good conjecture for the critical power.

In the impossibility statement,
the bound behaves differently in two regimes:
$d \gg n$ and $d \ll n$.
The result is a kink at $d \asymp n$ in the phase diagram
(see Figure~\ref{fg:diag}).
This may be due to the Wishart matrix being singular when $d < n$.

Theorem~\ref{th:det} can be displayed graphically by a phase diagram of
when geometry can be detected and when it cannot
in the space of dimension $d$ and geometry strength $q$.
We further introduce a more convenient parametrization
that allows us to visualize the phase diagram.
\begin{corollary}[Phase diagram]\label{co:diag}
Suppose that $d = n^\alpha$ and $q = n^{-\beta}$ for some $\alpha, \beta > 0$.
\begin{thenum}
\item\label{co:diag_upper} If $\beta > 1$ or $\alpha + 2\beta > 3$,
then as $n \to \infty$,
\begin{equation*}
\tv{\cG(n, p)}{\cG(n,p,d,q)} \to 0.
\end{equation*}
\item\label{co:diag_lower} If $\alpha + 6\beta < 3$, then as $n \to \infty$,
\begin{equation*}
\tv{\cG(n, p)}{\cG(n,p,d,q)} \to 1.
\end{equation*}
\end{thenum}
\end{corollary}
The resulting phase diagram is plotted in the two-dimensional space of $\alpha$
and $\beta$ in Figure~\ref{fg:diag}.
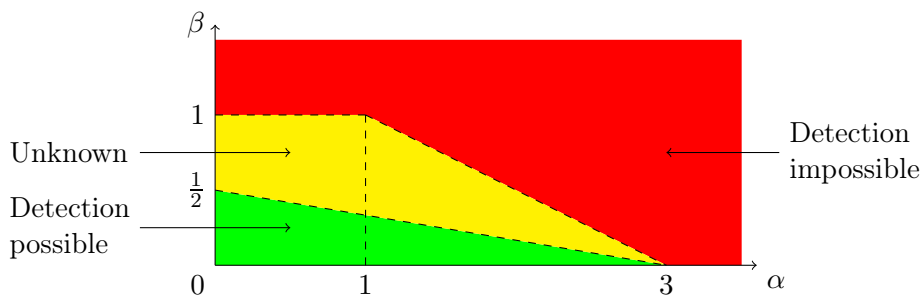
\begin{figure}[ht]
\begin{center}
\begin{tikzpicture}[scale=2]
\draw[fill=green,draw=none] (0, 0) -- (0,0.5) -- (3,0);
\draw[fill=red,draw=none] (0,1) -- (0,1.5) -- (3.5,1.5) -- (3.5,0) -- (3,0)
-- (1,1);
\draw[fill=yellow,draw=none] (0,0.5) -- (0,1) -- (1,1) -- (3,0);
\draw[->] node[below left] {$0$} (0,0) -- (3.6,0) node[below right] {$\alpha$};
\draw[->] (0,0) -- (0,1.6) node[left] {$\beta$};
\draw[dashed] (0,0.5) node[left] {$\frac{1}{2}$} -- (3,0) node[below] {$3$};
\draw[dashed] (1,1) -- (3,0);
\draw[dashed] (0,1) node[left] {$1$} -- (1,1);
\draw[dashed] (1,0) node[below] {$1$} -- (1,1);
\node[left,text width=1.6cm](a) at (-0.5,0.25) {Detection possible};
\node[left,text width=1.6cm](c) at (-0.5,0.75) {Unknown};
\node[right,text width=1.8cm](b) at (3.75,0.75) {Detection impossible};
\draw[->] (a) -- (0.5,0.25);
\draw[->] (c) -- (0.5,0.75);
\draw[->] (b) -- (3,0.75);
\end{tikzpicture}
\end{center}
\caption{Phase diagram for detecting geometry in the soft random geometric
graph $\cG(n,p,d,q)$.
Here $d = n^\alpha$ and $q = n^{-\beta}$ for some $\alpha, \beta > 0$.}
\label{fg:diag}
\end{figure}

\subsection{Related work}
The study of high-dimensional random geometric graphs originates in the work
of~\cite{devroye2011high},
who showed via the multivariate central limit theorem
that geometry is lost in high dimensions.
Subsequently, Bubeck et al.~\cite{bubeck2016testing} determined that
the phase transition of losing geometry
happens asymptotically at dimension $d \asymp n^3$
in the dense setting (for a fixed $p \in (0, 1)$).
This work also pointed out connections to classical random matrix ensembles,
showing that the Wishart to GOE transition also happens
at $d \asymp n^{3}$
(see also~\cite{jiang2013approximation,bubeck2016entropic,racz2019smooth,RB17,
chetelat2019middle}).
Eldan and Mikulincer~\cite{eldan2020information} further extended the results
to an anisotropic setting of the underlying distribution of points.
In the sparse setting, when $p$ vanishes as a function of $n$,
it is conjectured in \cite{bubeck2016testing} that geometry should be lost
at much lower dimensions.
Progress towards this conjecture, which in particular breaks the $n^3$ barrier,
was made by Brennan, Bresler, and Nagaraj~\cite{brennan2020phase}.
This is an active line of research;
after we finished this work,
we learned about a new preprint by Brennan, Bresler, and
Huang~\cite{brennan2021combinatorial} on understanding the Wishart to GOE
transition when only a subset of entries are revealed.

Soft random geometric graphs arise as natural models in many areas,
including
wireless communication \cite{gupta2000capacity}, social networks
\cite{hoff2002latent}, and biological networks \cite{tee2018phase}.
Penrose~\cite{penrose2016connectivity} studied the connectivity of soft random
geometric graphs from a modern probability-theoretic perspective,
determining the asymptotic probability of connectivity in fixed
dimensions and for a broad class of connection functions.
Dettmann and Georgiou~\cite{dettmann2016random} discussed the same questions
from a statistical physics viewpoint in two and three dimensions,
and provided a comprehensive list of connection functions
widely used in practice.
Connectivity in one-dimensional soft random geometric graphs was considered
in~\cite{wilsher2020connectivity},
where the authors showed that the reason for connectivity is vastly
different from the hard case.
Parthasarathy et al.~\cite{parthasarathy2017quest} studied a model of perturbed
networks, which is similar to the setting under consideration in our work.
A phase transition in soft random geometric graphs with a critical value of
chemical potential was demonstrated in \cite{ostilli2015statistical},
where a related model was also considered.
Our paper can be viewed as a first step towards understanding
the questions described in the previous paragraph
for soft random geometric graphs.

Following up detecting geometry, a natural next question is to recover it;
indeed, a line of research focuses on recovering the underlying latent positions
of soft random geometric graphs.
In \cite{sussman2013consistent},
it is shown that latent positions for random dot product graphs can be estimated
consistently using the eigendecomposition of the adjacency matrix,
when the dimension $d$ is fixed.
Several subsequent works \cite{araya2018relative,valdivia2019latent,
eldan2020community} applied similar approaches to kernels and general connection
functions on spheres satisfying certain eigengap conditions.

\subsection{Open problems}
The most immediate problem that our work leaves open is to understand
the intermediate region not covered by Theorem~\ref{th:det}.
Specifically, the main question is to determine the exact boundary between the
two phases where the limiting total variation distance transitions from $1$
to~$0$.
The existence of an intermediate phase where detection is
information-theoretically possible
while no efficient algorithm exists
is also worth studying.

More broadly, a natural direction of future research is to consider
these questions for other connection functions or underlying latent spaces,
in order to understand how the dimension threshold for
losing geometry depends on them.

\subsection{Outline of the paper}
The rest of the paper is organized as follows.
In Section~\ref{se:pre}, we introduce some notations used throughout the
paper and several standard definitions.
The impossibility of detection is presented afterwards in
Section~\ref{se:impsb},
where the two regimes of Theorem~\ref{th:det_lower} are proved.
Section~\ref{se:posb} consists of the proof of Theorem~\ref{th:det_upper}
using the signed triangle statistic.
Generalizations of signed triangles are also discussed.
Finally, in Section~\ref{se:rdpg}
we show that similar results also hold under a different distributional
assumption, when the underlying latent vectors are i.i.d.\ standard normal,
and we consider the corresponding random dot product graph.

\section{Notations and preliminaries} \label{se:pre}
We use boldface capitals to denote matrices, and their corresponding lower cases
with subscript indices separated by a comma to denote the entries.
For example, a matrix $\A = [a_{i,j}]$ has entry $a_{i,j}$ in its $i$th row and
$j$th column.
We use $\circ$ to denote the Hadamard product,
which is the entrywise product of two matrices.
Boldface lowercase letters represent vectors,
and their corresponding lower cases
with a subscript index denote the entries of the vector.
All vectors are treated as columns.
\begin{definition}[$f$-divergence] \label{de:fdiv}
Let $\cP$ and $\cQ$ be two probability measures on the same measurable space
$(\Omega, \cF)$.
Suppose that $\cP$ is absolutely continuous with respect to $\cQ$.
For a convex function $f$ such that $f(1) = 0$,
the $f$-divergence of $\cP$ and $\cQ$ is defined as
\begin{equation*}
\Df{\cP}{\cQ} \coloneqq \E_\cQ\biggl[f\biggl(\frac{d\cP}{d\cQ}\biggr)\biggr]
= \int_\Omega f\biggl(\frac{d\cP}{d\cQ}\biggr)\,d\cQ,
\end{equation*}
where $\frac{d\cP}{d\cQ}$ is the Radon--Nikodym derivative of $\cP$ with
respect to $\cQ$.
\end{definition}
In Definition~\ref{de:fdiv}, by choosing $f(t) = t\log t$,
we have the \emph{Kullback--Leibler (KL) divergence},
which we simply refer to as the \emph{divergence}.
Throughout the paper, $\log$ stands for the natural logarithm.
\begin{definition}[Kullback--Leibler (KL) divergence]
Let $\cP$ and $\cQ$ be two probability measures on the same measurable space
$(\Omega, \cF)$.
Suppose that $\cP$ is absolutely continuous with respect to $\cQ$.
The (KL) divergence of $\cP$ and $\cQ$ is defined as
\begin{equation*}
\kl{\cP}{\cQ}
\coloneqq \E_\cQ\biggl[\frac{d\cP}{d\cQ}\log\frac{d\cP}{d\cQ}\biggr]
= \int_\Omega \log\frac{d\cP}{d\cQ}\,d\cP.
\end{equation*}
\end{definition}
\begin{definition}[Total variation distance]
Let $\cP$ and $\cQ$ be two probability measures on a measurable space
$(\Omega, \cF)$.
The total variation distance between $\cP$ and $\cQ$ is defined as
\begin{equation*}
\tv{\cP}{\cQ} \coloneqq \sup_{A \in \cF} \abs{\cP(A) - \cQ(A)}.
\end{equation*}
\end{definition}
The total variation distance can also be viewed as an
$f$-divergence with $f(x) = \frac{1}{2}\abs{x - 1}$.
The total variation distance is simply referred to as the \emph{distance}
where no confusion is caused.
From the definition, it is clear that the distance between $\cP$ and $\cQ$ is
symmetric.
That is,
\begin{equation*}
\tv{\cP}{\cQ} = \tv{\cQ}{\cP}.
\end{equation*}

We use several inequalities concerning $f$-divergences in the proofs;
we state these later before they are applied.
The divergence and the distance are connected through Pinsker's inequality.
\begin{proposition}[Pinsker's inequality]\label{pr:pinsker}
For probability measures $\cP$ and $\cQ$, we have that
\begin{equation*}
\tv{\cP}{\cQ} \le \sqrt{\frac{1}{2} \kl{\cP}{\cQ}}.
\end{equation*}
\end{proposition}

\section{Impossibility of detecting geometry} \label{se:impsb}
In this section we prove the impossibility results stated
in Theorem~\ref{th:det_lower}.
We start our discussion with some weaker bounds
obtained through a simple mixture argument,
before proceeding to the proof of the main impossibility results.

We may view the soft random geometric graph $\cG(n, p, d, q)$ as an
edge-wise mixture between $\cG(n, p)$ and $\cG(n,p,d)$, in the following way.
First, we draw two sample graphs $G_{1} \sim \cG(n, p)$
and $G_{2} \sim \cG(n, p, d)$.
We next construct a graph $G$ using $G_{1}$, $G_{2}$,
and additional coin flips.
Specifically, for every pair of distinct vertices $i$ and $j$,
we flip an independent biased coin which comes up heads with probability $q$.
If the coin flip is heads,
connect $i$ and $j$ with an edge in $G$ if and only if
they are connected with an edge in~$G_{2}$;
otherwise, connect $i$ and $j$ with an edge in $G$ if and only if
they are connected with an edge in~$G_{1}$.
This construction guarantees that $G \sim \cG(n,p,d,q)$.

We can obtain two simple bounds directly from this construction.
Bubeck et al.~\cite[Theorem~1(c)]{bubeck2016testing} showed that
$\cG(n, p)$ and $\cG(n,p,d)$ become indistinguishable when $d \gg n^3$.
Thus, as an edge-wise mixture of these two models,
$\cG(n,p,d,q)$ also cannot be distinguished from $\cG(n,p)$ in this regime.

Meanwhile, when $q \ll 1/n^{2}$, for any pair $(i,j)$
the probability that the connection between $i$ and $j$ is sampled from $G_{2}$
is $o(1/n^2)$.
This implies that the expected number of edges
that are sampled from $G_{2}$ is $o(1)$.
Therefore, by Markov's inequality, the probability that
there exists an edge which is sampled from $G_{2}$ is $o(1)$.
Hence, with probability $1-o(1)$
we have that $G = G_1$ in the construction above.
Therefore, $\cG(n, p)$ and $\cG(n,p,d,q)$ are indistinguishable
when $q \ll 1/n^{2}$.
These two arguments are summarized in the following claim.
\begin{claim} \label{cl:box}
If $n^3/d \to 0$ or $n^2q \to 0$, then
\begin{equation*}
\sup_{p \in [0,1]} \tv{\cG(n,p)}{\cG(n,p,d,q)} \to 0.
\end{equation*}
\end{claim}
Our main result in Theorem~\ref{th:det_lower} improves on Claim~\ref{cl:box}
by relaxing both of the conditions and thus proving the impossibility result for
a larger parameter regime.
First, notice that the condition $n^{3}/d \to 0$
does not take $q$ into consideration at all;
we improve this to the condition $n^{3} q^{2} / d \to 0$.
We also improve the condition $n^2 q \to 0$ to $n q \to 0$.

As before, if we set $d = n^\alpha$ and $q = 1/n^{\beta}$
for some $\alpha, \beta > 0$,
then Claim~\ref{cl:box} translates into the following corollary picturing
a region of a phase diagram.
\begin{corollary}
Suppose that $d = n^\alpha$ and $q = 1/n^{\beta}$ for some $\alpha, \beta > 0$.
If $\alpha > 3$ or $\beta > 2$, then
\begin{equation*}
\sup_{p \in [0,1]} \tv{\cG(n,p)}{\cG(n,p,d,q)} \to 0.
\end{equation*}
\end{corollary}

\subsection{Impossibility of detection under large noise} \label{se:weak}
In this subsection we show that $\cG(n, p)$ and $\cG(n, p, d, q)$ are
asymptotically indistinguishable when $nq \to 0$,
thus proving Theorem~\ref{th:det_lower} under this regime.
This shows that when the noise is large enough (i.e., $q$ is small enough),
detecting geometry becomes impossible, regardless of the dimensionality.

For a graph $G$ with adjacency matrix $\A = [a_{i,j}]$,
the density $\P_{\cG(n, p, d, q)}(G)$ is given in \eqref{eq:P_srgg} with
$k_{i,j} = \phi_q(\inner{\x_i}{\x_j})$.
We can write the density of $\cG(n, p)$ similarly:
\begin{equation}\label{eq:gnp}
\P_{\cG(n, p)}(G) = \prod_{i < j} p^{a_{i,j}}{(1-p)}^{1-a_{i,j}}.
\end{equation}
The divergence of $\cG(n, p)$ and $\cG(n, p, d, q)$ can then be written as
\begin{equation*}
\begin{split}
\kl{\cG(n, p)}{\cG(n, p, d, q)}
&= \E_{\cG(n, p)}
\biggl[-\log \frac{\P_{\cG(n, p, d, q)}(G)}{\P_{\cG(n, p)}(G)}\biggr]\\
&= \E_{\A} \biggl[-\log \E_{\X}\biggl[\prod_{i < j}
{\biggl(\frac{k_{i,j}}{p}\biggr)}^{a_{i,j}}
{\biggl(\frac{1 - k_{i,j}}{1 - p}\biggr)}^{1 - a_{i,j}}\biggr]\biggr],
\end{split}
\end{equation*}
where the $a_{i,j}$'s are independent Bernoulli random variables
with parameter $p$
since the expectation is taken under $G \sim \cG(n,p)$.
Since $-\log$ is convex, by Jensen's inequality we have that
\begin{equation} \label{eq:KL_Jensen}
\begin{split}
\kl{\cG(n, p)}{\cG(n, p, d, q)}
&\le \E_{\A,\X} \biggl[-\log \prod_{i < j}
{\biggl(\frac{k_{i,j}}{p}\biggr)}^{a_{i,j}}
{\biggl(\frac{1 - k_{i,j}}{1 - p}\biggr)}^{1 - a_{i,j}}\biggr]\\
&= \E_{\A,\X} \biggl[- \sum_{i < j} \biggl(a_{i,j} \log
\frac{k_{i,j}}{p} + (1 - a_{i,j}) \log \frac{1 - k_{i,j}}{1 - p}\biggr)\biggr]\\
&= -\sum_{i < j} \biggl(\E_{\A}[a_{i,j}] \E_{\X} \biggl[\log
\frac{k_{i,j}}{p}\biggr] + \E_{\A}[1 - a_{i,j}] \E_{\X}
\biggl[\log \frac{1 - k_{i,j}}{1 - p}\biggr]\biggr),
\end{split}
\end{equation}
where the last line is due to linearity of expectation and independence.

Since $a_{i,j} \sim \mathrm{Bern}(p)$,
we have $\E[a_{i,j}] = p$ and $\E[1 - a_{i,j}] = 1 - p$.
By the definition of the connection function in \eqref{eq:edge}
and~\eqref{eq:link}, we have that
\begin{equation}
k_{i,j} = \phi_q(\inner{\x_i}{\x_j}) = \begin{cases}
(1 - q) p + q &\text{if } \inner{\x_i}{\x_j} \ge t_{p,d},\\
(1 - q) p &\text{otherwise}.
\end{cases}
\end{equation}
Recall that $t_{p,d}$ is chosen such that
$\P(\inner{\x_i}{\x_j} \ge t_{p,d}) = p$.
Hence, the marginal distribution of $k_{i,j}$ satisfies
$k_{i,j} = (1 - q) p + q$ with probability $p$
and $k_{i,j} = (1 - q) p$ with probability $1 - p$.
Therefore, we have that
\begin{equation*}
\E_{\X} \biggl[\log \frac{k_{i,j}}{p}\biggr]
= p \log \biggl(1 + \frac{1 - p}{p} q\biggr) + (1 - p) \log (1 - q)
\end{equation*}
and
\begin{equation*}
\E_{\X} \biggl[\log \frac{1 - k_{i,j}}{1 - p}\biggr]
= p \log (1 - q) + (1 - p) \log \biggl(1 + \frac{p}{1 - p} q\biggr).
\end{equation*}
By the elementary inequality $\log (1 + x) \ge x - x^2$ for $x \ge -1/2$,
we obtain that for $0 \le q \le 1/2$,
\begin{equation}
\E_{\X} \biggl[\log \frac{k_{i,j}}{p}\biggr] \ge -\frac{1-p}{p}q^2
\quad \mathrm{and} \quad
\E_{\X} \biggl[\log \frac{1 - k_{i,j}}{1 - p}\biggr] \ge -\frac{p}{1-p}q^2.
\end{equation}
Inserting the above estimates into \eqref{eq:KL_Jensen},
we conclude that for $0 \le q \le 1/2$,
\begin{equation}\label{eq:kl_up_a1}
\kl{\cG(n, p)}{\cG(n, p, d, q)}
\le \binom{n}{2}q^2
\le \frac{1}{2}n^2q^2.
\end{equation}

To be consistent with the main discussion,
we turn this upper bound on the divergence into an upper bound on the distance.
An application of Pinsker's inequality (Proposition~\ref{pr:pinsker})
combined with \eqref{eq:kl_up_a1} proves the following theorem.
\begin{theorem}\label{th:upper_a1}
For $0 \le q \le 1/2$ we have that
\begin{equation}
\sup_{p \in [0,1]} \tv{\cG(n, p)}{\cG(n, p, d, q)}
\le \frac{1}{2} n q.
\end{equation}
\end{theorem}
The $nq \to 0$ regime of Theorem~\ref{th:det_lower} directly follows from
Theorem~\ref{th:upper_a1}.

\subsection{Impossibility of detecting weak high-dimensional geometry}
\label{se:impo}
In this subsection we show that $\cG(n, p)$ and $\cG(n, p, d, q)$ are
asymptotically indistinguishable when $n^{3} q^{2} / d \to 0$,
proving Theorem~\ref{th:det_lower} under this regime.
This result thus highlights the interplay between noise and dimensionality
in determining when it is possible to detect geometry.

In order to capture this interplay between noise and dimensionality,
we use several inequalities concerning $f$-divergences,
and we start by recalling these.
The distance (divergence) between two random variables is understood as
the distance (divergence) between their corresponding probability measures.
Since we focus our attention on random graphs without self-loops,
the diagonal entries of real symmetric matrices are usually set to zero unless
specified.
For the distance (divergence) between two real symmetric random matrices,
only the lower triangular part is considered.
We also place notations for distributions inside the operators $\E, \Var$
to denote a sample from the corresponding distribution.

\begin{proposition}[Conditioning increases divergence] \label{pr:cond_div}
Let $\cP_X$ and $\cQ_X$ be two probability measures.
Let $Y$ be a random variable on the same space and denote by $\cP_{X \mid Y}$
and $\cQ_{X \mid Y}$ the conditional laws.
Then, the $f$-divergence satisfies
\begin{equation*}
\Df{\cP_X}{\cQ_X} \le \E_Y \Df{\cP_{X \mid Y}}{\cQ_{X \mid Y}}.
\end{equation*}
\end{proposition}
Proposition~\ref{pr:cond_div} is usually referred to as
``conditioning increases divergence'' in standard texts
(e.g., Theorem~2.2(5) and Remark~4.2 in \cite{polyanskiy2019it}).
Following the convention widely adopted in the information theory community,
we write $\kl{\cP_{X \mid Y}}{\cQ_{X \mid Y} \mid \cP_Y}
\coloneqq \E_Y \kl{\cP_{X \mid Y}}{\cQ_{X \mid Y}}$
and call it the \emph{conditional divergence}.

\begin{proposition}[Data processing inequality] \label{pr:dp}
For two probability distributions $\cP_X$ and $\cQ_X$,
consider the joint distributions $\cP_{X, Y} = \cP_{Y \mid X} \cP_X$ and
$\cQ_{X, Y} = \cP_{Y \mid X} \cQ_X$ with the same conditional law
$\cP_{Y \mid X}$.
Then, the $f$-divergence of the marginal distributions
$\cP_Y \coloneqq \E_{\cP_X}[\cP_{Y \mid X}]$ and
$\cQ_Y \coloneqq \E_{\cQ_X}[\cP_{Y \mid X}]$ satisfies
\begin{equation*}
\Df{\cP_Y}{\cQ_Y} \le \Df{\cP_X}{\cQ_X}.
\end{equation*}
\end{proposition}
A simple proof of Proposition~\ref{pr:dp} using Jensen's inequality can be found
in most texts (see, e.g., \cite[Theorem~6.2]{polyanskiy2019it}).
For a measurable function $g \colon E \to F$,
by choosing $\cP_{Y \mid X} (y \mid x) = \ind{y=g(x)}$,
we have the following corollary.
\begin{corollary}[Data processing inequality] \label{co:dp}
Let $X, Y \in E$ be two random variables
and let $g: E \to F$ be a measurable function.
Then, the $f$-divergence of the pushforward measures satisfies
\begin{equation*}
\Df{g(X)}{g(Y)} \le \Df{X}{Y}.
\end{equation*}
\end{corollary}

With these preliminaries in place, we now turn to our question of interest.
Let $\Z \in \RR^{n \times d}$ be a random matrix with
independent standard normal entries.
Then,
$\W \coloneqq \Z \Z^\top \in \RR^{n \times n}$ has
a Wishart distribution $\cW_n(\I, d)$.
Let $\Y$ contain the off-diagonals of $\W$ scaled
by the square root of the dimension,
that is, $y_{i,j} \coloneqq w_{i,j}/\sqrt{d}$.
By the central limit theorem, $y_{i,j}$ converges to
a standard normal random variable as $d \to \infty$.
Let $\V \coloneqq \frac{1}{d}\diag(\W)$ consist of
the scaled diagonals;
by the law of large numbers,
each nonzero entry $v_{i,i}$ converges to $1$.
For compactness of presentation,
denote by~$\v$ the vectorized diagonal of $\V$,
that is, $v_i \coloneqq w_{i,i}/d$.
Let $\M = [m_{i,j}]$ be a zero-diagonal symmetric random matrix with
off-diagonal entries $m_{i,j}$ following independent standard normal
distributions for $1 \le i < j \le n$.
Let $\M'$ and $\M''$ be two independent copies of~$\M$.

A standard method to create uniform random vectors on the sphere
(which goes back to~\cite{marsaglia1972choosing} and before) is as follows:
If $\z$ is a standard normal vector,
then $\hat{\z} \coloneqq \z / \norm{\z}$ is uniformly distributed
in $\SS^{d-1}$.
Thus, we can create the random vectors $\{ \x_i \}_{i=1}^{n}$
using $\{ \z_i \}_{i=1}^{n}$.
With the random matrices defined above,
the inner product of $\x_i$ and $\x_j$ can be expressed as
\begin{equation*}
\inner{\x_i}{\x_j}
= \inner{\hat{\z}_i}{\hat{\z}_j}
= \frac{\inner{\z_i}{\z_j}}{\norm{\z_i}\norm{\z_j}}
= \frac{w_{i,j}}{\sqrt{w_{i,i} w_{j,j}}}
= \frac{1}{\sqrt{d}}\frac{y_{i,j}}{\sqrt{v_i v_j}}
= \frac{1}{\sqrt{d}}(\V^{-1/2} \Y \V^{-1/2})_{i,j}.
\end{equation*}

Define the step function $s_{t_p}(x) := \ind{x \ge t_{p}}$,
where $t_p$ is determined by $\E[s_{t_p}(\cN(0, 1))] = p$.
That is, if we denote the cumulative distribution function of
the standard normal distribution by~$\Phi$,
then $t_p = \Phi^{-1}(1 - p)$.
In what follows,
when we apply a univariate function to a matrix,
it is applied entrywise,
resulting in a matrix of the same shape.
With this convention, let
\begin{equation*}
\bP \coloneqq s_{t_p}(\M')
\end{equation*}
and
\begin{equation*}
\bQ \coloneqq (1 - q) s_{t_p}(\M)
+ q s_{t_{p,d}}(\X\X^\top) = (1 - q) s_{t_p}(\M)
+ q s_{t_{p,d}\sqrt{d}}(\V^{-1/2} \Y \V^{-1/2}).
\end{equation*}

By the definition of $\cG(n, p)$ and the independence of entries of $\M'$,
\begin{equation} \label{eq:gnp_P}
\P_{\cG(n, p)}(G) = \prod_{i < j} p^{a_{i,j}} (1 - p)^{1 - a_{i,j}}
= \E_{\bP}\biggl[\prod_{i < j} p_{i,j}^{a_{i,j}}
(1 - p_{i,j})^{1 - a_{i,j}}\biggr],
\end{equation}
where we use the standard convention that $0^0 = 1$.
By \eqref{eq:link} and the definition of $m_{i,j}$, we can write
\begin{equation*}
k_{i,j}
= (1 - q) \E[s_{t_p}(m_{i,j})] + q s_{t_{p,d}}(\inner{\x_i}{\x_j})
= \E[q_{i,j} \mid \x_i, \x_j].
\end{equation*}
Then, by \eqref{eq:P_srgg} and the independence of the $m_{i,j}$,
\begin{equation}
\P_{\cG(n, p, d, q)}(G) = \E_{\bQ} \biggl[\prod_{i < j} q_{i,j}^{a_{i,j}}
(1 - q_{i,j})^{1 - a_{i,j}}\biggr].
\end{equation}
Since $\cG(n,p)$ and $\cG(n,p,d,q)$ have the same conditional law given $\bP$
and $\bQ$, respectively,
by Proposition~\ref{pr:dp} we have that
\begin{equation*}
\tv{\cG(n, p)}{\cG(n, p, d, q)} \le \tv{\bP}{\bQ}.
\end{equation*}

Define the zero-diagonal symmetric random matrix $\B = [b_{i,j}]$
with $b_{i,j}, 1 \le i < j \le n$, following
an independent Bernoulli distribution with parameter $q$;
the matrix $\B$ is also independent of everything else defined previously.
We can then rewrite $\bQ$ as
\begin{equation*}
\bQ = \E_{\B} [(\1\1^\top - \B) \circ s_{t_p}(\M)
+ \B \circ s_{t_{p,d}\sqrt{d}}(\V^{-1/2}\Y\V^{-1/2})].
\end{equation*}
For step functions with parameters $t$ and $t'$,
we have the simple relation:
$s_{t'}(x) = s_{t}(x + t - t')$.
Hence, we can further express $\bQ$ as
\begin{equation*}
\begin{split}
\bQ &= \E_{\B} [(\1\1^\top - \B) \circ s_{t_p}(\M)
+ \B \circ s_{t_{p}}(\V^{-1/2}\Y\V^{-1/2}
+ (t_{p} - t_{p,d}\sqrt{d})\1\1^\top)]\\
&= \E_{\B} [s_{t_p}((\1\1^\top - \B) \circ \M
+ \B \circ (\V^{-1/2}\Y\V^{-1/2}
+ (t_{p} - t_{p,d}\sqrt{d})\1\1^\top))].
\end{split}
\end{equation*}
Let $\delta_{p,d} \coloneqq t_p - t_{p,d}\sqrt{d}$
and $ \H \coloneqq (\1\1^\top - \B) \circ \M
+ \B \circ (\V^{-1/2}\Y\V^{-1/2} + \delta_{p,d}\1\1^\top)$.
With this notation we have that $\bQ = \E_{\B}[ s_{t_{p}}(\H)]$.

Applying Proposition~\ref{pr:cond_div} gives
\begin{equation*}
\tv{\bP}{\bQ}
= \tv{s_{t_{p}}(\M')}{\E_{\B}[ s_{t_{p}}(\H)]}
\le \E_{\B}\tv{s_{t_p}(\M')}{s_{t_p}(\H)}.
\end{equation*}
Further, by Corollary~\ref{co:dp} we have that
\begin{equation*}
\E_{\B}\tv{s_{t_p}(\M')}{s_{t_p}(\H)} \le \E_{\B}\tv{\M'}{\H}.
\end{equation*}
Let $\H' \coloneqq (\1\1^\top - \B) \circ \M''
+ \B \circ \V^{-1/2}\M''\V^{-1/2}$.
By the triangle inequality of the distance
(see, e.g., \cite[(4.6)]{levin2017markov}),
\begin{equation*}
\E_{\B}\tv{\M'}{\H} \le \E_{\B} \tv{\M'}{\H'} + \E_{\B} \tv{\H'}{\H}.
\end{equation*}
By Proposition~\ref{pr:cond_div} again,
\begin{equation*}
\E_{\B} \tv{\H'}{\H} \le \E_{\B,\V} \tv{\H'}{\H}.
\end{equation*}

For a fixed value of $\B$,
both $\H$ and $\H'$ consist of entries from two matrices.
For the $(i,j)$th entry, if $b_{i,j} = 0$,
the entries are from $\M''$ and $\M$ respectively;
if $b_{i,j} = 1$, they come from $\V^{-1/2}\Y\V^{-1/2} + \delta_{p,d}\1\1^\top$
and $\V^{-1/2}\M''\V^{-1/2}$ respectively.
In the latter case, when $\V$ is fixed,
we can multiply both entries by $\sqrt{v_{i,i}v_{j,j}}$, and
the distance between the new matrices stays the same by definition.
That is,
if we let $\H'' \coloneqq (\1\1^\top - \B) \circ \M
+ \B \circ (\Y + \delta_{p,d} \V^{1/2}\1\1^\top\V^{1/2})$, then
\begin{equation*}
\E_{\B,\V} \tv{\H'}{\H} = \E_{\B,\V} \tv{\M''}{\H''}.
\end{equation*}

Putting the bounds from above together, we obtain that
\begin{equation} \label{eq:tri_tv}
\tv{\cG(n, p)}{\cG(n, p, d, q)} \le \underbrace{\E_{\B} \tv{\M'}{\H'}}_{E_1}
+ \underbrace{\E_{\B,\V}\tv{\M''}{\H''}}_{E_2}.
\end{equation}
The first term on the right depicts the distance caused by normalization,
while the second one characterizes the level of independence between edges.
We deal with the two expectations $E_1$ and~$E_2$ in~\eqref{eq:tri_tv}
separately in the following two parts.
Subsequently, we bring our estimates together to conclude at
the end of the section.
\subsubsection{Upper bound for the first expectation}
By Pinsker's inequality and Proposition~\ref{pr:cond_div},
\begin{equation*}
E_1
\le \E_{\B}\sqrt{\frac{1}{2} \kl{\M'}{\H'}}
\le \E_{\B}\sqrt{\frac{1}{2} \E_{\V}\kl{\M'}{\H'}}
\le \sqrt{\frac{1}{2} \E_{\B, \V}\kl{\M'}{\H'}},
\end{equation*}
where the last inequality is by Jensen's inequality.

Given $\B$ and $\V$, the entries of $\H'$ are independent. Since the entries of
$\M'$ are also independent, we have that
\begin{equation*}
\kl{\M'}{\H'} = \sum_{i<j} b_{i,j}
\kl{m_{i,j}'}{(v_i v_j)^{-1/2} m_{i,j}''},
\end{equation*}
where we use the facts that $m_{i,j}'$ and $m_{i,j}''$ are identically
distributed and that the divergence of identical distributions is zero.
Therefore,
\begin{equation*}
\begin{split}
\E_{\B, \V}\kl{\M'}{\H'}
&= \sum_{i<j} \E[b_{i,j}]
  \E_{\V}\kl{m_{i,j}'}{(v_i v_j)^{-1/2} m_{i,j}''}\\
&= q \sum_{i<j} \E_{\V}\kl{m_{i,j}'}{(v_i v_j)^{-1/2} m_{i,j}''}.
\end{split}
\end{equation*}
Since $m_{i,j}''$ is a standard normal random variable,
$(v_i v_j)^{-1/2}m_{i,j}''$ is distributed as $\cN(0, (v_i v_j)^{-1})$.

The divergence of two normal distributions has an explicit formula given by
the following proposition, which appears in most standard texts
(see, e.g., \cite[exercise 1.30]{bishop2006pattern}).
\begin{proposition} \label{pr:kl_uni_norm}
For two normal distributions with means $\mu_1, \mu_2$
and variances $\sigma_1^2, \sigma_2^2$,
\begin{equation*}
\kl{\cN(\mu_1,\sigma_1^2)}{\cN(\mu_2,\sigma_2^2)}
= \log \frac{\sigma_2}{\sigma_1}
+ \frac{\sigma_1^2 + (\mu_1 - \mu_2)^2}{2\sigma_2^2} - \frac{1}{2}.
\end{equation*}
\end{proposition}
Applying Proposition~\ref{pr:kl_uni_norm}, we have that
\begin{equation*}
\kl{m_{i,j}'}{(v_i v_j)^{-1/2} m_{i,j}''}
= \frac{1}{2}(- \log (v_i v_j) + v_i v_j - 1)
= \frac{1}{2}(- \log v_i - \log v_j + v_i v_j - 1).
\end{equation*}
Since the $z_{i,j}$'s are independent standard normal random variables,
$v_i d = \sum_{j=1}^d z_{i,j}^2$ has a $\chi^2(d)$ distribution.
We utilize a lower bound on the expected logarithm of
a chi-square random variable shown by the following proposition.
\begin{proposition} \label{pr:log_chi2}
Suppose that $X$ is a $\chi^2(k)$ random variable.
Then,
\begin{equation*}
\E[\log X] \ge \log k - \frac{2}{k}.
\end{equation*}
\end{proposition}

\begin{proof}
We have the following explicit formula for the expected logarithm of $X$
(see, e.g., \cite[(B.30)]{bishop2006pattern}):
\begin{equation*}
\E[\log X] = \psi\biggl(\frac{k}{2}\biggr) - \log \frac{1}{2},
\end{equation*}
where $\psi$ is the digamma function defined by
$\psi(x) \coloneqq \Gamma'(x) / \Gamma(x)$.
The digamma function has the well-known upper and lower bounds
(see \cite[(2.2)]{alzer1997some} and references therein):
\begin{equation} \label{eq:digamma_bounds}
\log x - \frac{1}{x} \le \psi(x) \le \log x - \frac{1}{2x}.
\end{equation}

By the lower bound in \eqref{eq:digamma_bounds},
\begin{equation*}
\E[\log X] \ge \log \frac{k}{2} - \frac{2}{k} - \log \frac{1}{2}
= \log k - \frac{2}{k}.\qedhere
\end{equation*}
\end{proof}

Using $\E[v_i] = \E[v_id]/d = 1$ and
the estimate in Proposition~\ref{pr:log_chi2},
we have that
\begin{equation*}
\begin{split}
\E_{\V}\kl{m_{i,j}'}{(v_i v_j)^{-1/2} m_{i,j}''}
&= \frac{1}{2}(-\E[\log(v_id)]+\log d-\E[\log(v_jd)]+\log d+\E[v_i]\E[v_j]-1)\\
&= \log d - \E[\log(v_id)] \le \frac{2}{d}.
\end{split}
\end{equation*}
Therefore, we conclude that
\begin{equation} \label{eq:upper_first}
E_1
\le \sqrt{\frac{1}{2}q\sum_{i<j}\E_{\V}\kl{m_{i,j}'}{(v_iv_j)^{-1/2} m_{i,j}''}}
\le \sqrt{\binom{n}{2}\frac{q}{d}} \le \sqrt{\frac{n^2q}{2d}}.
\end{equation}

\subsubsection{Upper bound for the second expectation}
We now turn to estimating $E_{2}$ from~\eqref{eq:tri_tv}.
We first bound the divergence of $\M''$ and $\H''$,
assuming that $\B$ and $\V$ are fixed,
and then provide an estimate for the distance between them
through Pinsker's inequality (Proposition~\ref{pr:pinsker}).
The benefit of resorting to the divergence is the chain rule property.
Our strategy resembles that of Bubeck and Ganguly~\cite{bubeck2016entropic}.

We state the chain rule for the divergence as the following proposition.
\begin{proposition}[Chain rule] \label{pr:chain}
For joint distributions $\cP_{X, Y} = \cP_{X \mid Y} \cP_Y$
and $\cQ_{X, Y} = \cQ_{X \mid Y} \cQ_Y$,
the chain rule for the divergence reads
\begin{equation}
\kl{\cP_{X, Y}}{\cQ_{X, Y}}
= \kl{\cP_Y}{\cQ_Y} + \kl{\cP_{X \mid Y}}{\cQ_{X \mid Y} \mid \cP_Y}.
\end{equation}
\end{proposition}

For an $n \times n$ matrix $\A = [a_{i,j}]$,
denote its $k$th order leading principal submatrix by $\A_k$
and let $\bm{a}_k = (a_{k,1}, \ldots, a_{k,k-1})$
be the vector of the first $k-1$ entries in the $k$th row.
We also use $\Z_k$ to denote the matrix composed of the first $k$ rows of
$\Z \in \RR^{n \times d}$
and $\z_k$ to denote the $k$th row of $\Z$.

Until the end of this section, we assume that $\B$ and $\V$ are fixed.
Applying the chain rule to the divergence of $\H_{k+1}''$ and
$\M_{k+1}''$,
we obtain
\begin{equation*}
\kl{\H_{k+1}''}{\M_{k+1}''}
= \kl{\H_k''}{\M_k''} + \E_{\H_k''}\kl{\h_{k+1}'' \mid \H_k''}
{\m_{k+1}'' \mid \M_k'' = \H_k''}.
\end{equation*}
Further, since $\m_{k+1}''$ is independent of $\M_k''$,
\begin{equation*}
\E_{\H_k''}\kl{\h_{k+1}'' \mid \H_k''}
{\m_{k+1}'' \mid \M_k'' = \H_k''}
= \E_{\H_k''} \kl{\h_{k+1}'' \mid \H_k''}{\m_{k+1}''}.
\end{equation*}
By Proposition~\ref{pr:cond_div},
\begin{equation*}
\E_{\H_k''} \kl{\h_{k+1}'' \mid \H_k''}{\m_{k+1}''}
\le \E_{\H_k'', \Z_k} \kl{\h_{k+1}'' \mid \H_{k}'', \Z_k}{\m_{k+1}''}
= \E_{\Z_k} \kl{\h_{k+1}'' \mid \Z_k}{\m_{k+1}''}.
\end{equation*}
The equality holds since $\h_{k+1}''$ only depends on $\Z_k$
and is independent of other randomness in $\H_k''$.
Since $\z_{k+1}$ is a standard normal random vector, conditioning on $\Z_k$,
$\Z_k \z_{k+1} / \sqrt{d}$ is distributed as $\cN(\0, \Z_k \Z_k^\top / d)$.
By definition, $\m_{k+1}''$ has a $\cN(\0, \I_k)$ distribution.
Let $\D_k \coloneqq \diag(\b_{k+1})$ be the diagonal matrix whose entries are
the elements of $\b_{k+1}$.
Since
\begin{equation*}
\begin{split}
\h_{k+1}'' &= (\1 - \b_{k+1}) \circ \m_{k+1} + \b_{k+1} \circ
(\y_{k+1} + \delta_{p,d} \sqrt{v_{k+1}} \V_k^{1/2}\1)\\
&= (\1 - \b_{k+1}) \circ \m_{k+1} + \b_{k+1} \circ
\biggl(\frac{\Z_k \z_{k+1}}{\sqrt{d}}
+ \delta_{p,d}\sqrt{v_{k+1}} \V_k^{1/2} \1\biggr),
\end{split}
\end{equation*}
the distribution of $\h_{k+1}''$,
given $\Z_k$, $\D_k$, and $\V_k$,
is $\cN(\bmu_k, \bSigma_k)$ with
\begin{equation*}
\bmu_k = \delta_{p,d} \sqrt{v_{k+1}} \D_k \V_k^{1/2} \1
\end{equation*}
and
\begin{equation*}
\bSigma_k
= \D_k \biggl(\frac{\Z_k {\Z_k}^\top}{d}\biggr) {\D_k}^\top
+ (\I_k - \D_k){(\I_k - \D_k)}^\top
= \D_k \biggl(\frac{\Z_k {\Z_k}^\top}{d}\biggr) {\D_k} + \I_k - \D_k.
\end{equation*}

As a general form of Proposition~\ref{pr:kl_uni_norm},
we have an explicit formula for the divergence of two
$d$-dimensional normal distributions
(see, e.g., \cite[Exercise~15.13(b)]{wainwright2019high}),
stated as follows.
\begin{proposition} \label{pr:kl_norm}
For two $d$-dimensional multivariate normal distributions with means
$\bmu_1, \bmu_2$ and covariance matrices $\bSigma_1, \bSigma_2$, we have that
\begin{equation*}
\kl{\cN(\bmu_1, \bSigma_1)}{\cN(\bmu_2, \bSigma_2)}
= \frac{1}{2} \biggl({(\bmu_1 - \bmu_2)}^\top \bSigma_2^{-1} (\bmu_1 - \bmu_2)
+ \log \frac{\det{(\bSigma_2)}}{\det{(\bSigma_1)}}
+ \Tr(\bSigma_2^{-1} \bSigma_1) - d\biggr).
\end{equation*}
\end{proposition}

Since $\m_{k+1}'' \sim \cN(\0, \I_k)$, applying Proposition~\ref{pr:kl_norm},
we have that
\begin{equation*}
\begin{split}
\E_{\V,\Z_k} \kl{\h_{k+1}'' \mid \Z_k}{\m_{k+1}''}
&= \frac{1}{2}\E_{\V,\Z_k}[\delta_{p, d}^2 v_{k+1} \1^\top \D_k^2 \V_k \1
- \log\det{(\bSigma_k)} + \Tr(\bSigma_k) - k]\\
&= \frac{1}{2}\E_{\Z_k,\V_k,v_{k+1}}
[\delta_{p, d}^2 v_{k+1} \Tr(\D_k \V_k)
- \log\det{(\bSigma_k)} + \Tr(\bSigma_k) - k].
\end{split}
\end{equation*}
Since $v_k d$ has a $\chi^2(d)$ distribution,
$\E[v_k] = \E[v_kd]/d = 1$.
Further, since $\V$ and $\D$ are independent,
by linearity of expectation we have that
\begin{equation*}
\E_{\V,\Z_k} \kl{\h_{k+1}'' \mid \Z_k}{\m_{k+1}''}
= \frac{1}{2} (\delta_{p, d}^2 \Tr(\D_k)
+ \E_{\Z_k}[- \log \det(\bSigma_k)]
+ \E_{\Z_k} \Tr(\bSigma_k) - k).
\end{equation*}
Additionally,
\begin{equation*}
\E_{\Z_k} \Tr(\bSigma_k)
= \Tr \biggl(\D_k \E_{\Z_k}\biggl[\frac{\Z_k \Z_k^\top}{d}\biggr]
\D_k \biggr) + k - \Tr(\D_k)
= \Tr(\D_k) + k - \Tr(\D_k) = k.
\end{equation*}
Therefore, we obtain that
\begin{equation} \label{eq:two_terms}
\E_{\V,\Z_k} \kl{\h_{k+1}'' \mid \Z_k}{\m_{k+1}''}
= \frac{1}{2} (\delta_{p, d}^2 \Tr(\D_k) + \E_{\Z_k}[-\log\det{(\bSigma_k)}]).
\end{equation}
Next, we derive upper bounds for the two terms in the above display.

An upper bound on $\abs{\delta_{p,d}} = \abs{t_{p,d} \sqrt{d} - t_p}$
is shown in \cite{devroye2011high}, which is stated as the following lemma.
\begin{lemma}[{\cite[Lemma~1]{devroye2011high}}]
\label{le:tpd_de}
Assume $0 < p \le 1/2$ and $d \ge \max \{(2/p)^2, 27\}$.
Then
\begin{equation*}
\abs{t_{p,d} \sqrt{d} - t_p} \le U_{p,d},
\end{equation*}
where
\begin{equation*}
U_{p,d} = \kappa_p \sqrt{\log d / d} + \kappa_p' / \sqrt{d}
\end{equation*}
with $\kappa_p = 2\sqrt{2} \Phi^{-1}(1 - p)$ and
$\kappa_p' = 2\sqrt{2\pi} \exp((\Phi^{-1}(1 - p/2))^2/2)$.
\end{lemma}

However, the $\log d$ factor above is an artifact
due to the use of concentration inequalities;
if we were to apply the lemma directly,
it would show up in our final bounds.
To address this, we present an improved result,
which not only removes the $\log d$ factor in the upper bound
but also tightens the upper bound
such that it decays as $1/d$ instead of $1/\sqrt{d}$.
\begin{lemma} \label{le:delta_pd}
Assume $0 < p \le 1/2$.
There exists a constant $C_p$ such that
\begin{equation*}
\abs{t_{p,d}\sqrt{d} - t_p} \le \frac{C_p}{d}.
\end{equation*}
\end{lemma}

\begin{proof}
By rotational invariance,
fixing $\x_1 = \e_1$ and
letting $\x_2 = \hat{\z} \coloneqq \z/\norm{\z}$ with $\z \sim \cN(\0, \I_d)$,
we have that
\begin{equation*}
\P(\inner{\x_1}{\x_2} \ge t_{p,d})
= \P\biggl(\frac{z_1}{\norm{\z}} \ge t_{p,d}\biggr) = p.
\end{equation*}
Then, by symmetry of the distribution, we have
\begin{equation*}
\P\biggl(\frac{z_1}{\norm{\z}} \le -t_{p,d}\biggr) = p.
\end{equation*}
Therefore,
\begin{equation*}
\P\biggl(\frac{z_1^2}{\sum_{i=1}^d z_i^2} \ge t_{p,d}^2\biggr) = 2p.
\end{equation*}

Since $z_1^2 \sim \chi^2(1)$ and $\sum_{i=2}^d z_i^2 \sim \chi^2(d-1)$ are
independent,
$z_1^2/\sum_{i=1}^d z_i^2$ has a
$\mathrm{Beta}(\frac{1}{2},\frac{d-1}{2})$
distribution.
For ease of presentation,
we switch from dimension $d$ to considering dimension $d+3$.
By the probability density function of the beta distribution we have that
\begin{equation*}
2p
=
\P\biggl(\frac{z_1^2}{\sum_{i=1}^{d+3} z_i^2} \ge t_{p,d+3}^2\biggr)
=\frac{\Gamma\bigl(\frac{d+3}{2}\bigr)}
     {\Gamma\bigl(\frac{1}{2}\bigr)\Gamma\bigl(\frac{d+2}{2}\bigr)}
\int_{t_{p,d+3}^2}^{1} x^{-1/2}(1-x)^{d/2}\,dx.
\end{equation*}
The change of variables $x = z/d$, and some rearranging, yields
\begin{equation}\label{eq:beta_identity}
\frac{\Gamma\bigl(\frac{d+2}{2}\bigr)\sqrt{d}}
     {\sqrt{2}\Gamma\bigl(\frac{d+3}{2}\bigr)}2p
= \frac{1}{\sqrt{2}\Gamma\bigl(\frac{1}{2}\bigr)}
\int_{t_{p,d+3}^2 d}^{d} z^{-1/2}\biggl(1-\frac{z}{d}\biggr)^{d/2}\,dz.
\end{equation}
Wendel's double inequality
(see \cite[equation (7)]{wendel1948note}) states that for $0 < s < 1$,
\begin{equation*}
\biggl(\frac{z}{z+s}\biggr)^{1-s} \le \frac{\Gamma(z+s)}{z^s\Gamma(z)} \le 1.
\end{equation*}
Then, by setting $s=1/2$ and $z = d/2$, we have that
\begin{equation} \label{eq:wendel}
\frac{1}{\sqrt{d/2}}
\le \frac{\Gamma\bigl(\frac{d}{2}\bigr)}{\Gamma\bigl(\frac{d+1}{2}\bigr)}
\le \frac{\sqrt{2(d+1)}}{d}.
\end{equation}
Thus,
\begin{equation*}
\frac{\Gamma\bigl(\frac{d+2}{2}\bigr)\sqrt{d}}
     {\sqrt{2}\Gamma\bigl(\frac{d+3}{2}\bigr)}
\ge  \sqrt{\frac{d}{d+2}} \ge 1 - \frac{2}{d+2},
\end{equation*}
where the last inequality is due to the fact that
$(1-x)^{1/2} \ge 1 - x$ for $0 \le x \le 1$.

Since
\begin{equation} \label{eq:log_upper}
\log \biggl(1-\frac{z}{d}\biggr)^{d/2}
= \frac{d}{2} \log \biggl(1-\frac{z}{d}\biggr)
\le \frac{d}{2}\biggl(-\frac{z}{d}\biggr)
= -\frac{z}{2},
\end{equation}
we have
\begin{equation*}
\begin{split}
\frac{1}{\sqrt{2}\Gamma\bigl(\frac{1}{2}\bigr)}
\int_{t_{p,d+3}^2 d}^{d} z^{-1/2}\biggl(1-\frac{z}{d}\biggr)^{d/2}\,dz
&\le \frac{1}{\sqrt{2}\Gamma\bigl(\frac{1}{2}\bigr)}
\int_{t_{p,d+3}^2 d}^{d} z^{-1/2} e^{-z/2}\,dz\\
&\le \frac{1}{\sqrt{2}\Gamma\bigl(\frac{1}{2}\bigr)}
\int_{t_{p,d+3}^2 d}^{+\infty} z^{-1/2} e^{-z/2}\,dz
= 1 - F\biggl(t_{p,d+3}^2 d; \frac{1}{2}, \frac{1}{2}\biggr),
\end{split}
\end{equation*}
where $F(x;a,b)$ is the cumulative distribution function of
the gamma distribution $\mathrm{Gamma}(a, b)$.
Therefore, putting these inequalities back into~\eqref{eq:beta_identity}, we
obtain that
\begin{equation*}
2\biggl(1 - \frac{2}{d+2}\biggr)p
\le 1 - F\biggl(d t_{p,d+3}^2; \frac{1}{2}, \frac{1}{2}\biggr).
\end{equation*}
Since $\mathrm{Gamma}(\frac{1}{2}, \frac{1}{2})$ is also the distribution of
a squared standard normal random variable,
\begin{equation*}
\biggl(1 - \frac{2}{d+2}\biggr)p \le 1 - \Phi(t_{p,d+3}\sqrt{d}).
\end{equation*}
By the monotonicity of the cumulative distribution function,
we have that
\begin{equation*}
t_{p,d+3}\sqrt{d} \le \Phi^{-1}\biggl(1 - p + \frac{2p}{d+2}\biggr).
\end{equation*}

Since $\Phi^{-1}(x)$ is convex for $1/2 \le x < 1$,
we have that for $0 < y < 1 - x$,
\begin{equation*}
\Phi^{-1}(x+y) \le \Phi^{-1}(x) + y(\Phi^{-1})'(x+y).
\end{equation*}
Let $\varphi$ be the probability density function of
the standard normal distribution.
Then,
\begin{equation*}
(\Phi^{-1})'(x+y) = \frac{1}{\varphi(\Phi^{-1}(x+y))}
= \sqrt{2\pi}\exp\biggl(\frac{1}{2}(\Phi^{-1}(x+y))^2\biggr).
\end{equation*}
Additionally, for $d \ge 2$,
\begin{equation*}
\Phi^{-1}\biggl(1 - p + \frac{2p}{d+2}\biggr)
\le \Phi^{-1}\biggl(1 - \frac{p}{2}\biggr) = t_{p/2}.
\end{equation*}
Therefore, for $d \ge 2$ we have that
\begin{equation*}
t_{p,d+3} \sqrt{d}
\le t_p + \frac{2p}{d+2}\sqrt{2\pi}\exp\biggl(\frac{1}{2}t_{p/2}^2\biggr).
\end{equation*}
Then,
\begin{equation*}
t_{p,d+3} \sqrt{d+3}
\le \sqrt{1 + \frac{3}{d}} \biggl(t_p
    +\frac{2\sqrt{2\pi}}{d+2}\exp\biggl(\frac{1}{2}t_{p/2}^2\biggr)\biggr)
\le \biggl(1 + \frac{3}{2d}\biggr)\biggl(t_p
    +\frac{2\sqrt{2\pi}}{d+2}\exp\biggl(\frac{1}{2}t_{p/2}^2\biggr)\biggr).
\end{equation*}
By assuming $d \ge 6$,
\begin{equation*}
t_{p,d}\sqrt{d}
\le \biggl(1 + \frac{3}{2(d-3)}\biggr)\biggl(t_p
    +\frac{2\sqrt{2\pi}}{d-1} \exp\biggl(\frac{1}{2}t_{p/2}^2\biggr)\biggr)
\le \biggl(1 + \frac{3}{d}\biggr)\biggl(t_p
    +\frac{4\sqrt{2\pi}}{d} \exp\biggl(\frac{1}{2}t_{p/2}^2\biggr)\biggr)
\le t_p + \frac{C_p}{d},
\end{equation*}
where $C_p = 3(t_p+2\sqrt{2\pi}\exp(\frac{1}{2}t_{p/2}^2))$.

Similarly, we also have
\begin{equation*}
\P\biggl(\frac{z_1^2}{\sum_{i=1}^{d+3} z_i^2} \le t_{p,d+3}^2\biggr) = 1-2p,
\end{equation*}
which gives
\begin{equation*}
\frac{\Gamma\bigl(\frac{d+2}{2}\bigr)\sqrt{d}}
     {\sqrt{2}\Gamma\bigl(\frac{d+3}{2}\bigr)}(1-2p)
= \frac{1}{\sqrt{2}\Gamma\bigl(\frac{1}{2}\bigr)}
\int_0^{d t_{p,d+3}^2} x^{-1/2}\biggl(1-\frac{x}{d}\biggr)^{d/2}\,dx.
\end{equation*}
Employing \eqref{eq:wendel} and \eqref{eq:log_upper} again, we have
\begin{equation*}
\biggl(1 - \frac{2}{d+2}\biggr)(1-2p)
\le F\biggl(d t_{p,d+3}^2; \frac{1}{2}, \frac{1}{2}\biggr).
\end{equation*}
Then,
\begin{equation*}
\Phi(t_{p,d+3}\sqrt{d})
\ge 1 - \frac{1}{2}\biggl(1 - \biggl(1 - \frac{2}{d+2}\biggr)(1-2p)\biggr)
= 1 - p - \frac{1-2p}{d+2}.
\end{equation*}
By convexity of $\Phi^{-1}(x)$ in $(1/2, 1)$,
for $1/2 < x + y <  1$,
\begin{equation*}
\Phi^{-1}(x + y) \ge
\Phi^{-1}(x) + y\sqrt{2\pi}\exp
\biggl(\frac{1}{2}(\Phi^{-1}(x))^2\biggr).
\end{equation*}
Therefore, we have that
\begin{equation*}
t_{p,d+3}\sqrt{d+3} \ge t_{p,d+3}\sqrt{d}
\ge t_p - \frac{(1-2p)}{d+2}\sqrt{2\pi}\exp\biggl(\frac{1}{2}t_p^2\biggr).
\end{equation*}
Hence, by assuming $d \ge 6$,
\begin{equation*}
t_{p,d}\sqrt{d} \ge t_p - \frac{C_p}{d},
\end{equation*}
where $C_p = 2(1-2p)\sqrt{2\pi}\exp(\frac{1}{2}t_p^2)$.
\end{proof}
\begin{remark}
From the proof of Lemma~\ref{le:delta_pd}, we see that the lemma actually
specifies a convergence rate for the quantile function of a scaled beta
distribution to that of a gamma distribution.
More general claims and a Berry--Ess\'een type result can be derived with the
same techniques.
\end{remark}

\begin{corollary} \label{co:delta_pd}
For $p \in (0, 1)$, there exists a constant $C_p$ such that
\begin{equation*}
\delta_{p,d}^2 \le \frac{C_p}{d^2}.
\end{equation*}
\end{corollary}

\begin{proof}
Applying Lemma~\ref{le:delta_pd} for $0 < p \le 1/2$,
we have that there exists a constant $C_p'$ such that
\begin{equation*}
\delta_{p,d}^2 \le \frac{C_p'}{d^2}.
\end{equation*}
As before, by fixing $\x_1 = \e_1$ and
letting $\x_2 = \hat{\z} \coloneqq \z/\norm{\z}$ with $\z \sim \cN(\0, \I_d)$,
we have
\begin{equation*}
\P(\inner{\x_1}{\x_2} \ge t_{p,d})
= \P\biggl(\frac{z_1}{\norm{\z}} \ge t_{p,d}\biggr) = p.
\end{equation*}
Since $z_1/\norm{\z}$ has a symmetric distribution,
we have $t_{p,d} \ge 0$ for $0 < p \le 1/2$
and $t_{p,d} \le 0$ for $1/2 \le p <1$.
When $1/2 \le p < 1$,
\begin{equation*}
\P\biggl(\frac{z_1}{\norm{\z}} \ge -t_{p,d}\biggr)
= \P\biggl(-\frac{z_1}{\norm{\z}} \le t_{p,d}\biggr)
= \P\biggl(\frac{z_1}{\norm{\z}} \le t_{p,d}\biggr)
= 1 - p.
\end{equation*}
Applying Lemma~\ref{le:delta_pd} again, we obtain that there exists
a constant $C_p''$
such that
\begin{equation*}
(-t_{p,d}\sqrt{d} - \Phi^{-1}(p))^2
= (t_{p,d}\sqrt{d} + \Phi^{-1}(p))^2
= (t_{p,d}\sqrt{d} - \Phi^{-1}(1 - p))^2
\le \frac{C_p''}{d^2}.
\end{equation*}
By taking $C_p = C_p' + C_p''$, the claim directly follows.
\end{proof}

We now return to bounding the two terms in~\eqref{eq:two_terms},
starting with the first one.
Let $\ell \coloneqq \Tr(\D_k)$ be the number of nonzero entries in $\b_{k+1}$.
Then, $\ell$ is a function of $\B$, is independent of everything else,
and has a binomial distribution.
Corollary~\ref{co:delta_pd} gives
\begin{equation} \label{eq:EKL_1}
\delta_{p,d}^2 \Tr(\D_k) \le \frac{C_p\ell}{d^2}.
\end{equation}

Next, we turn to the upper bound for the second term in \eqref{eq:two_terms}.
Let $\bPi$ be a permutation matrix such that
$\b_{k+1}\bPi = (1, \dotsc, 1, 0, \dotsc, 0)$ becomes a vector
with its first $\ell$
entries equal to $1$ and the remaining $k-\ell$ entries equal to $0$.
Then,
\begin{equation*}
\bPi^\top \bSigma_k \bPi = \begin{pmatrix}
\S & \O\\
\O & \I_{k-\ell}
\end{pmatrix},
\end{equation*}
where $\O$ is the all-zero matrix and $\S$ is distributed the same as
$\Z_\ell \Z_\ell^\top/d$.
Recall that $\Z_\ell \in \RR^{\ell \times d}$ is the matrix of the first $\ell$
rows of $\Z$,
which has independent standard normal entries.
Since $\bPi$ is a permutation matrix,
its determinant is either $1$ or $-1$.
Therefore, we have that
\begin{equation} \label{eq:exp_logdet}
\E_{\Z_k}[- \log \det(\bSigma_k)]
= \E_{\Z_k}[- \log \det(\bPi^\top\bSigma_k\bPi)]
= \E_{\Z_k}[- \log \det(\S)]
= \E_{\Z_\ell}\biggl[-\log \det
\biggl(\frac{\Z_\ell{\Z_\ell}^\top}{d}\biggr) \biggr].
\end{equation}
So our main focus is bounding \eqref{eq:exp_logdet} from above.

The study of covariance matrices has attracted broad interests
in probability and statistics communities.
As an example, Cai, Liang, and Zhou~\cite{cai2015law} showed a central limit
theorem (CLT) for the log-determinant of such matrices.
An upper bound of the expected negative log-determinant is given
in \cite{bubeck2016entropic} under a general log-concave measure assumption,
serving as the major step towards an entropic CLT.
We first state their result as the following lemma.
\begin{lemma}[{\cite[Lemma~2]{bubeck2016entropic}}] \label{le:logdet_bm}
Let $\Z$ be an $n \times d$ random matrix with i.i.d.\ entries
from a log-concave probability measure $\mu$ with zero mean and unit variance.
There exists an absolute constant $C > 0$ such that for $d \ge C n^2$,
\begin{equation*}
\E \biggl[-\log\det\biggl(\frac{\Z \Z^\top}{d}\biggr)\biggr]
\le C \biggl(\sqrt{\frac{n}{d}} + \frac{n^2}{d}\biggr).
\end{equation*}
\end{lemma}

A direct application of Lemma~\ref{le:logdet_bm} results in
an upper bound that is loose for our analysis.
It is possible to leverage the normal distribution assumption
to obtain an improved estimate,
which we implement in Lemma~\ref{le:logdet} below.
Applying the lemma results in a better upper bound on the distance.

\begin{lemma} \label{le:logdet}
Consider an $n \times d$ matrix $\Z$ with independent standard normal entries.
For $d \ge 2n$,
\begin{equation*}
\E\biggl[-\log\det\biggl(\frac{\Z \Z^\top}{d}\biggr)\biggr]
\le \frac{4n}{d} + \frac{n^2}{d}.
\end{equation*}
\end{lemma}

\begin{remark}
Compared to Lemma~\ref{le:logdet_bm}, the improvement thanks to
Lemma~\ref{le:logdet} is twofold.
First,
the upper bound removes the $\sqrt{n/d}$ term, which would be
the leading term in our analysis,
replacing it with an $n/d$ term.
Second, the inequality holds for $d \ge 2n$ rather than $d \ge Cn^2$.
We shall see how this improvement is reflected in the upper bound
on the distance in the final remarks.
\end{remark}

\begin{proof}[Proof of Lemma~\ref{le:logdet}]
For a random matrix $\W$ following a Wishart distribution $\cW_n(\bSigma, d)$,
the expectation of its log-determinant has an explicit formula
(see, e.g., \cite[(B.81)]{bishop2006pattern}):
\begin{equation} \label{eq:Wish}
\E[\log\det(\W)] = \sum_{i=1}^{n} \psi\biggl(\frac{d - i + 1}{2}\biggr)
+ n \log 2 + \log \det (\bSigma),
\end{equation}
where $\psi$ is the digamma function. 
Applying \eqref{eq:Wish} to $\Z\Z^\top$, we obtain
\begin{equation*}
\E\biggl[-\log\det\biggl(\frac{\Z\Z^\top}{d}\biggr)\biggr]
= - \E[\log\det(\Z\Z^\top)] + n \log d
= \sum_{i=1}^{n} - \psi\biggl(\frac{d - i + 1}{2}\biggr) + n \log \frac{d}{2}.
\end{equation*}
By the lower bound on the digamma function in \eqref{eq:digamma_bounds},
\begin{equation} \label{eq:logdet_lower}
\begin{split}
\E\biggl[- \log \det \biggl(\frac{\Z \Z^\top}{d}\biggr)\biggr]
&\le \sum_{i=1}^{n} \biggl(\frac{2}{d - i + 1}
- \log \frac{d - i + 1}{2}\biggr) + n \log \frac{d}{2}\\
&= \underbrace{\sum_{i=1}^{n}\frac{2}{d - i + 1}}_{S_1}
- \underbrace{\sum_{i=1}^{n} \log (d - i + 1)}_{S_2} + n \log d.
\end{split}
\end{equation}
The rest of the proof is devoted to bounding the two sums $S_1$ and $S_2$ from
 above separately.

By the elementary inequality $x \le - \log (1 - x)$,
$S_1$ can be bounded from above by
\begin{equation*}
S_1
\le 2\sum_{i=1}^{n} - \log\biggl(1-\frac{1}{d-i+1}\biggr)
= -2\log\biggl(\prod_{i=1}^{n}\frac{d-i}{d-i+1}\biggr)
= -2\log\biggl(1-\frac{n}{d}\biggr).
\end{equation*}
Further, by $- \log (1 - x) \le 2 x$ for $0 \le x \le 1/2$,
we have that for $d \ge 2n$,
\begin{equation} \label{eq:sum1}
S_1 \le -2\log\biggl(1-\frac{n}{d}\biggr) \le \frac{4n}{d}.
\end{equation}

For $S_2$,
we show a lower bound by constructing a continuous integral.
Since $\log(d-i+1) \ge \log(d-x)$ for all $x \in [i-1,i)$,
\begin{equation} \label{eq:sum2}
S_2 \ge \int_0^n \log(d-x)\,dx
= (x\log x-x) \Bigr\rvert_{d - n}^{d}
= -n + d\log d - (d - n) \log(d-n).
\end{equation}
Bringing the inequalities \eqref{eq:sum1} and \eqref{eq:sum2}
into \eqref{eq:logdet_lower},
we conclude that for $d \ge 2n$,
\begin{equation*}
\begin{split}
\E\biggl[- \log \det \biggl(\frac{\Z \Z^\top}{d}\biggr)\biggr]
&\le \frac{4n}{d} + n - d \log d + (d - n) \log (d - n) + n \log d\\
&= \frac{4n}{d} + n + (d - n) \log \biggl(1 - \frac{n}{d}\biggr)
\le \frac{4n}{d} + n + (d - n) \biggl(-\frac{n}{d}\biggr)
= \frac{4n}{d} + \frac{n^2}{d}.\qedhere
\end{split}
\end{equation*}
\end{proof}

As a consequence of Lemma~\ref{le:logdet},
there is an absolute constant $C > 0$ (we can take $C=5$),
such that for $d \ge 2n \ge 2k \ge 2\ell$,
\begin{equation} \label{eq:logdet_c}
\E\biggl[- \log \det \biggl(\frac{\Z_\ell \Z_\ell^\top}{d}\biggr)\biggr]
\le \frac{C\ell^2}{d}.
\end{equation}

Plugging the estimates in \eqref{eq:EKL_1} and \eqref{eq:logdet_c}
into \eqref{eq:two_terms},
we get that for $d \ge 2n$,
\begin{equation} \label{eq:EXKL_upper}
\E_{\V,\H_k''} \kl{\h_{k+1}'' \mid \H_k''}{\m_{k+1}''}
\le \E_{\V,\Z_k} \kl{\h_{k+1}'' \mid \Z_k}{\m_{k+1}''}
\le C_p\frac{\ell}{d^2} + C\frac{\ell^2}{d}
\end{equation}
for constants $C, C_p$.

Since $b_{i,j}$ has an independent Bernoulli distribution,
$\ell = \sum_{j=1}^{k} b_{k+1,j}$ follows a binomial distribution
$\mathrm{Bin}(k, q)$.
Hence, we have $\E[\ell] = kq$,
and $\E[\ell^2]$ can be bounded from above by
\begin{equation*}
\E[\ell^2] = \Var[\ell] + \E[\ell]^2 = kq(1-q) + (kq)^2 \le kq + k^2q^2.
\end{equation*}
Therefore, by taking the expectation over $\B$ in \eqref{eq:EXKL_upper},
we obtain that for $d \ge 2n$,
\begin{equation*}
\begin{split}
\E_{\B,\V,\H_k''} \kl{\h_{k+1}'' \mid \H_k''}{\m_{k+1}''}
\le C_p\frac{\E[\ell]}{d^2} + C\frac{\E[\ell^2]}{d}
\le C_p\frac{kq}{d^2} + C\biggl(\frac{kq}{d}+\frac{k^2q^2}{d}\biggr).
\end{split}
\end{equation*}

An iterative application of the chain rule yields
\begin{equation*}
\begin{split}
\E_{\B,\V} \kl{\H''}{\M''}
&= \E_{\B,\V} \biggl[\sum_{k=0}^{n-1} \E_{\H_k''}
\kl{\h_{k+1}'' \mid \H_k''}{\m_{k+1}''}\biggr]\\
&= \sum_{k=0}^{n-1}
\E_{\B,\V,\H_k''} \kl{\h_{k+1}'' \mid \H_k''}{\m_{k+1}''}.
\end{split}
\end{equation*}
Therefore, for $d \ge 2n$,
\begin{equation}\label{eq:final_KL_bound}
\E_{\B,\V} \kl{\H''}{\M''}
\le \sum_{k=0}^{n-1} \biggl(C_p\frac{kq}{d^2}
+ C\biggl(\frac{kq}{d}+\frac{k^2q^2}{d}\biggr)\biggr)
\le C_p \frac{n^2q}{d^2}
+ C \biggl(\frac{n^2q}{d} + \frac{n^3q^2}{d}\biggr),
\end{equation}
for some $C, C_p < \infty$.

By Pinsker's inequality (Proposition~\ref{pr:pinsker})
and Jensen's inequality,
we have that
\begin{equation*}
E_2 \coloneqq \E_{\B,\V} \tv{\M''}{\H''}
\le \E_{\B,\V} \sqrt{\frac{1}{2}\kl{\H''}{\M''}}
\le \sqrt{\frac{1}{2}\E_{\B,\V} \kl{\H''}{\M''}}.
\end{equation*}
Hence, using~\eqref{eq:final_KL_bound}
we conclude that there exist $C, C_p < \infty$ such that for
$d \ge 2n$,
\begin{equation} \label{eq:upper_second}
E_2 \le C_p\sqrt{\frac{n^2q}{d^2}}
+ C \biggl(\sqrt{\frac{n^2q}{d}} + \sqrt{\frac{n^3q^2}{d}}\biggr).
\end{equation}

\subsubsection{Concluding the proof}
Plugging the estimates in \eqref{eq:upper_first} and \eqref{eq:upper_second}
into \eqref{eq:tri_tv},
we have proven that there exist constants $C, C_p$ such that
for $d \ge 2n$,
\begin{equation} \label{eq:tv_upper}
\tv{\cG(n,p)}{\cG(n,p,d,q)} \le C_p\sqrt{\frac{n^2q}{d^2}}
+ C \biggl(\sqrt{\frac{n^2q}{d}} + \sqrt{\frac{n^3q^2}{d}}\biggr).
\end{equation}
We now explain why the $n^3q^2/d \to 0$ regime of Theorem~\ref{th:det_lower}
follows.
First, note that $n^{3} q^{2} / d = (nq)^{2} n/d$.
Thus, if we were to have $d < 2n$,
then $n^3q^2/d \to 0$ implies that $nq \to 0$,
and under this assumption we have already shown in Section~\ref{se:weak} that
the conclusion of Theorem~\ref{th:det_lower} holds.
So we may assume that $d \geq 2n$,
in which case the bound in~\eqref{eq:tv_upper} holds.
Then, $n^3q^2/d \to 0$ implies that the last term in~\eqref{eq:tv_upper}
goes to $0$.
For the second term, note that
$n^{2}q/d = (n^{4}q^{2}/d^{2})^{1/2}
\le (n^{3}q^{2}/d)^{1/2}$,
where we used that $d \geq n$, and so this term also vanishes.
This implies that the first term vanishes as well.

\begin{remark}
Using Lemma~\ref{le:logdet_bm} in the place of Lemma~\ref{le:logdet}
and following the same derivations,
we would similarly obtain that
for constants $C', C_p'$, when $d \ge C' n^2$,
\begin{equation}
\tv{\cG(n, p)}{\cG(n, p, d, q)}
\le C_p' \sqrt{\frac{n^2q}{d^2}}
+ C'\biggl(\sqrt[\leftroot{-2}\uproot{6}4]{\frac{n^3q}{d}}
+ \sqrt{\frac{n^3q^2}{d}}\biggr).
\end{equation}
In this case, the dominating term becomes $\sqrt[4]{n^3q/d}$,
resulting in a worse bound than~\eqref{eq:tv_upper},
and hence the conclusion follows only in a smaller
parameter regime.
\end{remark}

\begin{remark}
Utilizing Lemma~\ref{le:logdet},
an upper bound of the total variation distance between Wishart and GOE
is readily available.
Applying chain rule directly to the divergence between $\Y$ and $\M$,
we conclude that for an absolute constant $C > 0$,
\begin{equation}
\tv{\Y}{\M} \le C\sqrt{\frac{n^3}{d}}.
\end{equation}
This result removes the first term and log factors of Theorem~2
in \cite{bubeck2016entropic} in this special case, and coincides with
the exact formula given by R\'acz and Richey~\cite{racz2019smooth}
up to a multiplicative constant.
\end{remark}

\section{Detecting geometry using signed triangles and other statistics}
\label{se:posb}
In this section, we explore
when detecting geometry in $\cG(n, p, d, q)$ is possible
and how to detect it.
In particular, we demonstrate that the signed triangle statistic,
proposed by Bubeck et al.~\cite{bubeck2016testing},
can be used to detect latent geometric structure
whenever $n^{3} q^{6} / d \to \infty$,
thus proving Theorem~\ref{th:det_upper}.
We also provide evidence suggesting that this boundary cannot be improved
with the family of signed cliques and of signed cycles,
which are natural extensions of signed triangles.

Consider a simple graph $G=(V,E)$, where $V$ is the set of vertices and
$E \subset \binom{V}{2}$ is the set of edges.
For a set $S$, we use $\binom{S}{k}$ to denote the collection of all subsets of
$S$ with cardinality $k$.
Let $\A = [a_{i,j}]$ be the adjacency matrix of $G$ and
write $a_e \coloneqq a_{i,j}$
for any edge $e = \{i,j\} \in \binom{V}{2}$.
Let $H=(S,F)$ be another graph with $S \subset V$ and
$F \subset \binom{S}{2}$.
Define $I_{H \mid G}$ to be the indicator of $H$ being a subgraph of $G$.
When the graph $G$ is clear from the context, we simply write $I_H$ instead of
$I_{H \mid G}$.
Then,
\begin{equation}
I_H = \ind{F \subset E} = \prod_{e \in F} a_e.
\end{equation}
Further, for a constant $p \in [0,1]$, let
\begin{equation}
\lambda_H \coloneqq \prod_{e \in F} (a_e - p)
\end{equation}
be the signed indicator of the subgraph $H$.

We first state a lemma that connects the expected signed indicator in
$\cG(n,p,d,q)$ to that in $\cG(n,p,d)$.
\begin{lemma} \label{le:gnpdq_gnpd}
Let $H = (S, F)$ be a fixed graph.
The signed indicator satisfies
\begin{equation*}
\E_{\cG(n,p,d,q)}[\lambda_H] = q^{\abs{F}}\E_{\cG(n,p,d)}[\lambda_H].
\end{equation*}
\end{lemma}
\begin{proof}
By conditioning on $\X$,
\begin{equation*}
\E_{\cG(n,p,d,q)}[\lambda_H]
= \E_{\X}\biggl[\E\biggl[\prod_{e \in F} (a_e - p) \biggm\vert \X\biggr]\biggr]
= \E_{\X}\biggl[\prod_{e \in F} \E[a_e - p \mid \X]\biggr],
\end{equation*}
where the last equality is by conditional independence of edges.

Given $\x_i$ and $\x_j$,
$a_{i,j}$ is a Bernoulli random variable with parameter
\begin{equation*}
k_{i,j} = (1-q)p + qs_{t_{p,d}}(\inner{\x_i}{\x_j}).
\end{equation*}
Hence,
\begin{equation} \label{eq:E_aij_sij}
\E[a_{i,j}-p \mid \X] = (1-p)k_{i,j} + (-p)(1-k_{i,j}) = k_{i,j} - p
= q(s_{t_{p,d}}(\inner{\x_i}{\x_j}) - p).
\end{equation}
Therefore,
\begin{equation} \label{eq:gnpd_qxsp}
\E_{\cG(n,p,d,q)}[\lambda_H]
= \E_{\X}\biggl[\prod_{\{i,j\} \in F}
  q(s_{t_{p,d}}(\inner{\x_i}{\x_j}) - p)\biggr]
= q^{\abs{F}}\E_{\X}\biggl[\prod_{\{i,j\} \in F}
             (s_{t_{p,d}}(\inner{\x_i}{\x_j}) - p)\biggr].
\end{equation}

On the other hand, for a hard random geometric graph $\cG(n,p,d)$,
\begin{equation*}
\E_{\cG(n,p,d)}[\lambda_H]
= \E\biggl[\prod_{e \in F} (a_e - p)\biggr]
= \E_{\X}\biggl[\prod_{e \in F}
  (s_{t_{p,d}}(\inner{\x_i}{\x_j}) - p)\biggr].\qedhere
\end{equation*}
\end{proof}

We first consider the graph $H = (S,F)$ being a complete graph
on three vertices, namely a \emph{triangle}.
Since in this case the subgraph $H$ is fully determined by its vertex set $S$,
we denote $T_S \coloneqq I_H$ to emphasize the dependency.
Given the adjacency matrix $\A$ of $G$,
$T_S$ can be expressed as
\begin{equation}
T_S = \prod_{\{i, j\} \subset S} a_{i,j}.
\end{equation}
Then, the total number triangles in $G$,
denoted by $T_3(G)$,
can be written as
\begin{equation}
T_3(G) \coloneqq \sum_{S \in \binom{V}{3}} T_S.
\end{equation}
The signed triangle and its count in $G$,
following the proposal by Bubeck et al.~\cite{bubeck2016testing},
are defined as
\begin{equation}
\tau_{\{i,j,k\}} \coloneqq \prod_{e \subset \{i,j,k\}} (a_e - p)
\quad \text{and} \quad
\tau_3(G) \coloneqq \sum_{\{i,j,k\} \subset V} \tau_{\{i,j,k\}}.
\end{equation}
For a sample random graph $G$ with edge probability $p$,
$\tau_3(G)$ is called the \emph{signed triangle statistic}.

To simplify our presentation, for random graphs with edge probability $p$,
we let
\begin{equation*}
\bra_{i,j} \coloneqq a_{i,j} - \E[a_{i,j}] = a_{i,j} - p
\quad
\text{and}
\quad
\brs_{i,j} \coloneqq s_{t_{p,d}}(\inner{\x_i}{\x_j}) - p.
\end{equation*}

For $\cG(n, p)$,
by the analyses in \cite[Section~3.1]{bubeck2016testing},
we have that
\begin{equation}
\E[\tau_3(\cG(n,p))] = 0
\quad \text{and} \quad
\Var[\tau_3(\cG(n,p))\bigr] = \binom{n}{3} p^3 (1 - p)^3.
\end{equation}

We analyze the expectation and variance of the signed triangle statistic
in $\cG(n, p, d, q)$ in the following two subsections.
For expository purposes, we start with the special case when $p = 1/2$,
and then proceed to the general case when $p$ is fixed in $(0, 1)$.

\subsection{The case \texorpdfstring{$p = \frac{1}{2}$}{p = 1/2}}
In this case, due to the symmetry of the distribution,
the threshold satisfies $t_{1/2,d} = 0$,
which no longer depends on the dimension.
As a result, the vectors that have an inner product greater than or equal to
the threshold with a fixed vector lie in a half space instead of a cone,
thus allowing a projection argument.
Utilizing the explicit distribution function in the projected space,
we are able to obtain asymptotically tight bounds for both the expectation
and the variance of the signed triangle statistic.
Some results can also be derived as corollaries from
more general statements in \cite{eldan2020information}.
However, we include our much simplified proofs for completeness
and as preparations for further claims.

Recall that the threshold $t_{p,d}$ is determined by
$\E[s_{t_{p,d}}(\inner{\x_i}{\x_j})]
= \P(\inner{\x_i}{\x_j} \ge t_{p,d}) = p$.
When $p$ is set to $1/2$, by symmetry of the distribution,
we have $t_{p,d} = 0$.
Then, the connection probability becomes
\begin{equation}
k_{i,j} = \frac{1}{2}(1 - q) + q s_0(\inner{\x_i}{\x_j}).
\end{equation}

Let $\x_1, \x_2, \x_3$ be independent random vectors uniformly in $\SS^{d-1}$.
Consider the event
\begin{equation} \label{eq:E_Delta}
E^\Delta \coloneqq \{\inner{\x_1}{\x_2} \ge 0, \inner{\x_2}{\x_3} \ge 0,
\inner{\x_3}{\x_1} \ge 0\}.
\end{equation}
We first give asymptotically tight bounds for $\P(E^\Delta)$ via a geometric
argument.

\subsubsection{Estimating the expectation}
Before starting our main discussion,
we present a proposition which gives an explicit probability density function
for the angle between two uniform random vectors in $\SS^{d-1}$.
Note that the probability density expressed by $\sin$
and gamma functions was also
derived in \cite{hammersley1950distribution,bubeck2016testing}
using different approaches.

For two vectors $\x, \y \in \RR^d$,
let $\theta(\x,\y) \in [0, \pi]$ stand for the angle between them.
Then,
\begin{equation*}
\theta(\x,\y) = \arccos\frac{\inner{\x}{\y}}{\norm{\x}\norm{\y}}.
\end{equation*}
If we further assume that $\x, \y \in \SS^{d-1}$, then
\begin{equation*}
\theta(\x,\y) = \arccos\inner{\x}{\y}.
\end{equation*}

\begin{proposition} \label{pr:angle_density}
The angle between two uniformly random vectors in $\SS^{d-1}$ has the
probability density function
\begin{equation*}
h(\theta) = \frac{1}{\zeta} \sin^{d-2}\theta, \quad \theta \in [0, \pi],
\end{equation*}
where
\begin{equation*}
\zeta \coloneqq \int_0^\pi \sin^{d-2}\theta
= \frac{\sqrt{\pi}\Gamma\bigl(\frac{d-1}{2}\bigr)}
  {\Gamma\bigl(\frac{d}{2}\bigr)}
\end{equation*}
is the normalization factor.
\end{proposition}
\begin{proof}
Let $\z \sim \cN(\0, \I_d)$.
Then, $\hat{\z} \coloneqq \z/\norm{\z}$ is a uniform random point in $\SS^{d-1}$
(see \cite{muller1956some}, also \cite{muller1959note,marsaglia1972choosing}).
By rotation invariance on the sphere, we can fix one vector to be $\e_1$,
the first vector of the standard basis in $\RR^d$.
Then, the cumulative distribution function of the angle satisfies
\begin{equation*}
F(\theta) = \P(\arccos\inner{\e_1}{\hat{\z}} \le \theta)
= \P\biggl(\frac{z_1}{\norm{\z}} \ge \cos\theta\biggr).
\end{equation*}
For $\theta\in[0,\pi/2]$,
\begin{equation*}
\P\biggl(\frac{z_1}{\norm{\z}} \ge \cos\theta\biggr)
= \frac{1}{2}\P\biggl(\frac{z_1^2}{\sum_{i=1}^d z_i^2} \ge \cos^2\theta\biggr).
\end{equation*}
Since the $z_i$'s are standard normal random variables,
$z_1^2 \sim \chi^2(1)$ and $\sum_{i=2}^d z_i^2 \sim \chi^2(d-1)$
are independent.
Therefore, $z_1^2/\sum_{i=1}^d z_i^2$ is distributed as
$\mathrm{Beta}(\frac{1}{2},\frac{d-1}{2})$.
By the definition of the beta distribution,
\begin{equation*}
\P\biggl(\frac{z_1^2}{\sum_{i=1}^d z_i^2} \ge \cos^2\theta\biggr)
= \frac{1}{\rB\bigl(\frac{1}{2},\frac{d-1}{2}\bigr)}\int_{\cos^2\theta}^1
  x^{-1/2}(1-x)^{d/2-3/2}\,dx,
\end{equation*}
where $\rB(x,y)=\Gamma(x)\Gamma(y)/\Gamma(x+y)$ is the beta function.
Taking the derivative with respect to $\theta$, we have that
\begin{equation*}
h(\theta) = \frac{1}{2\rB\bigl(\frac{1}{2},\frac{d-1}{2}\bigr)}
(-(\cos\theta)^{-1}(1-\cos^2\theta)^{d/2-3/2})(-2\cos\theta\sin\theta)
= \frac{\Gamma\bigl(\frac{d}{2}\bigr)}
  {\sqrt{\pi}\Gamma\bigl(\frac{d-1}{2}\bigr)}\sin^{d-2}\theta.
\end{equation*}

For $\theta \in [\pi/2,\pi]$,
\begin{equation*}
\P\biggl(\frac{z_1}{\norm{\z}} \ge \cos\theta\biggr)
= \P\biggl(-\frac{z_1}{\norm{\z}} \le \cos(\pi-\theta)\biggr)
= \P\biggl(\frac{z_1}{\norm{\z}} \le \cos(\pi-\theta)\biggr),
\end{equation*}
where the last equality is by symmetry of the distribution.
Hence,
\begin{equation*}
h(\theta) = \frac{\Gamma\bigl(\frac{d}{2}\bigr)}
  {\sqrt{\pi}\Gamma\bigl(\frac{d-1}{2}\bigr)}\sin^{d-2}(\pi-\theta)
= \frac{\Gamma\bigl(\frac{d}{2}\bigr)}
  {\sqrt{\pi}\Gamma\bigl(\frac{d-1}{2}\bigr)}\sin^{d-2}\theta.\qedhere
\end{equation*}
\end{proof}

\begin{lemma} \label{le:tri_p=1/2}
For $E^\Delta$ defined in \eqref{eq:E_Delta}, we have
\begin{equation*}
\frac{1}{2\pi\sqrt{2\pi}} \cdot \frac{1}{\sqrt{d}}
\le \P(E^\Delta) - \frac{1}{8}
\le \frac{1}{4\sqrt{\pi}} \cdot \frac{1}{\sqrt{d}}.
\end{equation*}
\end{lemma}

\begin{proof}
We fix the plane determined by $\x_1$ and $\x_2$,
and then project $\x_3$ onto this plane.
Since no direction in this plane is unique,
the projected direction of $\x_3$ is uniform on the circle centered at the
origin.
Conditioning on $\x_1$ and $\x_2$,
$\{\inner{\x_1}{\x_3} \ge 0, \inner{\x_2}{\x_3} \ge 0\}$ happens if and only if
the projection of $\x_3$ falls into the intersection of two half spaces with
normal vectors $\x_1$ and $\x_2$.
Therefore,
\begin{equation*}
\P(\inner{\x_1}{\x_3} \ge 0, \inner{\x_2}{\x_3} \ge 0
\mid \x_1,\x_2)
= \frac{\pi - \theta(\x_1,\x_2)}{2\pi}.
\end{equation*}
Hence, the joint probability
\begin{equation*}
\begin{split}
\P(E^\Delta)
&= \E[\P(\inner{\x_1}{\x_3} \ge 0, \inner{\x_2}{\x_3} \ge 0 \mid \x_1, \x_2)
\mid \inner{\x_1}{\x_2} \ge 0]\\
&= \E\biggl[\frac{\pi-\theta(\x_1,\x_2)}{2\pi}
  \bb1\biggl\{\theta(\x_1,\x_2) \le \frac{\pi}{2}\biggr\}\biggr].
\end{split}
\end{equation*}
The density of $\theta(\x_1,\x_2)$ is given by
Proposition~\ref{pr:angle_density}.
Then,
\begin{equation*}
\begin{split}
 \P(E^\Delta)
&= \int_0^{\pi/2} \frac{\pi-\theta}{2\pi} \cdot \frac{1}{\zeta}\sin^{d-2}\theta
\,d\theta
= \int_0^{\pi/2} \frac{\pi/2+\theta}{2\pi} \cdot \frac{1}{\zeta}\cos^{d-2}\theta
\,d\theta\\
&= \frac{1}{2\zeta} \int_0^{\pi/2} \biggl(\frac{1}{2}+\frac{\theta}{\pi}\biggr)
\cos^{d-2}\theta\,d\theta
= \frac{1}{8} + \frac{1}{2\pi\zeta}
  \int_0^{\pi/2} \theta\cos^{d-2}\theta\,d\theta.
\end{split}
\end{equation*}

The elementary bounds
$2\theta/\pi \le \sin \theta \le \theta$,
which hold for $\theta \in [0, \pi/2]$,
give
\begin{equation} \label{eq:theta_sin}
\sin\theta \le \theta \le \frac{\pi}{2} \sin\theta.
\end{equation}
Thus, multiplying by $\cos^{d-2} \theta$ and taking the integral, we have
\begin{equation*}
\int_0^{\pi/2} \sin\theta \cos^{d-2}\theta\,d\theta
\le \int_0^{\pi/2} \theta \cos^{d-2}\theta\,d\theta
\le \frac{\pi}{2} \int_0^{\pi/2} \sin\theta \cos^{d-2}\theta\,d\theta.
\end{equation*}
A simple calculation gives
\begin{equation*}
\int_0^{\pi/2} \sin\theta \cos^{d-2}\theta\,d\theta
= - \int_0^{\pi/2} \cos^{d-2}\theta\,d\cos\theta
= \int_0^1 t^{d-2}\,dt = \frac{t^{d-1}}{d-1} \biggr\rvert_0^1
= \frac{1}{d-1}.
\end{equation*}
The above display together with the definition of $\zeta$ yields
\begin{equation}
\frac{1}{\zeta} \int_0^{\pi/2} \sin\theta \cos^{d-2}\theta\,d\theta
= \frac{\Gamma\bigl(\frac{d}{2}\bigr)}
{\sqrt{\pi}(d-1)\Gamma\bigl(\frac{d-1}{2}\bigr)}
= \frac{\Gamma\bigl(\frac{d}{2}\bigr)}
{2\sqrt{\pi}\Gamma\bigl(\frac{d+1}{2}\bigr)}.
\end{equation}
By \eqref{eq:wendel},
\begin{equation}
\frac{\sqrt{2}}{\sqrt{d}} = \frac{1}{\sqrt{d/2}}
\le \frac{\Gamma\bigl(\frac{d}{2}\bigr)}{\Gamma\bigl(\frac{d+1}{2}\bigr)}
\le \sqrt{\frac{d+1}{d}} \cdot \frac{1}{\sqrt{d/2}}
\le \frac{2}{\sqrt{d}}.
\end{equation}
Putting all the above together, we obtain
\begin{equation*}
\frac{1}{8} + \frac{1}{2\pi\sqrt{2\pi}} \cdot \frac{1}{\sqrt{d}}
\le \P(E^\Delta)
\le \frac{1}{8} + \frac{1}{4\sqrt{\pi}} \cdot \frac{1}{\sqrt{d}}.\qedhere
\end{equation*}
\end{proof}

With Lemma~\ref{le:tri_p=1/2},
we are able to estimate the expectation of the signed triangle statistic
in $\cG(n, 1/2, d, q)$.

\begin{lemma} \label{le:ET_p/2}
There exist absolute constants
$C, C' > 0$ such that
\begin{equation*}
\frac{Cn^3 q^3}{\sqrt{d}}
\le \E[\tau_3(\cG(n, 1/2, d, q))]
\le \frac{C'n^3 q^3}{\sqrt{d}}.
\end{equation*}
\end{lemma}

\begin{proof}
By \eqref{eq:gnpd_qxsp},
the expectation of the signed triangle $\tau_{\{1,2,3\}}$ satisfies
\begin{equation*} \label{eq:E_tau123_gnpdq}
\E_{\cG(n, 1/2, d, q)}[\tau_{\{1,2,3\}}]
= q^3 \E\biggl[\biggl(s_0(\inner{\x_1}{\x_2}) - \frac{1}{2}\biggr)
\biggl(s_0(\inner{\x_2}{\x_3}) - \frac{1}{2}\biggr)
\biggl(s_0(\inner{\x_3}{\x_1}) - \frac{1}{2}\biggr)\biggr].
\end{equation*}

By rotation invariance on the sphere,
we may fix the direction of $\x_1$ to be $\e_1$.
Then,
\begin{equation*}
\begin{split}
\E[s_0(\inner{\x_1}{\x_2})s_0(\inner{\x_1}{\x_3})]
&= \P(\inner{\x_1}{\x_2} \ge 0, \inner{\x_1}{\x_3} \ge 0)
= \P(\inner{\e_1}{\x_2} \ge 0, \inner{\e_1}{\x_3} \ge 0)\\
&= \P(\inner{\e_1}{\x_2} \ge 0) \P(\inner{\e_1}{\x_3} \ge 0)
= \frac{1}{2} \cdot \frac{1}{2} = \frac{1}{4}.
\end{split}
\end{equation*}
Thus, by expanding the product in \eqref{eq:E_tau123_gnpdq}
and using the linearity of expectation,
we obtain
\begin{equation*}
\E_{\cG(n, 1/2, d, q)}[\tau_{\{1,2,3\}}]
= q^3 \biggl(\P(E^\Delta) - \frac{1}{8}\biggr).
\end{equation*}
Inserting Lemma~\ref{le:tri_p=1/2} yields
\begin{equation} \label{eq:Eq_tau123}
\frac{1}{2\pi\sqrt{2\pi}} \cdot \frac{q^3}{\sqrt{d}}
\le \E_{\cG(n, 1/2, d, q)}[\tau_{\{1,2,3\}}]
\le \frac{1}{4\sqrt{\pi}} \cdot \frac{q^3}{\sqrt{d}}.
\end{equation}
Since all signed triangle indicators are identically distributed,
\begin{equation*}
\E_{\cG(n, 1/2, d, q)}[\tau_3(G)]
= \sum_{\{i, j, k\} \subset V}
\E_{\cG(n, 1/2, d, q)}[\tau_{\{i, j, k\}}]
= \binom{n}{3} \E_{\cG(n, 1/2, d, q)}[\tau_{\{1,2,3\}}].
\end{equation*}
The claim directly follows.
\end{proof}

\subsubsection{Estimating the variance}
The variance of $\tau_3(G)$ for $G \sim \cG(n,1/2,d,q)$ satisfies
\begin{equation*}
\Var[\tau_3(\cG(n,1/2,d,q))]
= \E[\tau_3(G)^2] - \E[\tau_3(G)]^2
= \E \biggl[\biggl(\sum_{\{i,j,k\} \in V} \tau_{\{i,j,k\}}\biggr)^2\biggr]
- \biggl(\sum_{\{i,j,k\} \in V} \E[\tau_{\{i,j,k\}}]\biggr)^2.
\end{equation*}
Expanding the squares of sums and by linearity of expectation,
we can decompose the variance into one summation of variances
and three summations of covariances;
they are grouped by the number of shared vertices.
Since the variances or the covariances are identically distributed
within each group,
we can rewrite them as
\begin{equation} \label{eq:Var_tau3}
\begin{split}
\Var[\tau_3(\cG(n,1/2,d,q))]
&= \binom{n}{3} V_{\{1,2,3\},\{1,2,3\}}
+ \binom{n}{4}\binom{4}{2} V_{\{1,2,3\},\{1,2,4\}}\\
&\phantom{{}={}} + \binom{n}{5}\binom{5}{3}\frac{3}{2} V_{\{1,2,3\},\{1,4,5\}}
+ \binom{n}{6}\binom{6}{3}\frac{1}{2} V_{\{1,2,3\},\{4,5,6\}},
\end{split}
\end{equation}
where
\begin{equation*}
V_{\{i,j,k\},\{i',j',k'\}}
\coloneqq \E[\tau_{\{i,j,k\}}\tau_{\{i',j',k'\}}]
  -\E[\tau_{\{i,j,k\}}]\E[\tau_{\{i',j',k'\}}]
\end{equation*}
and the coefficients arise from simple combinatorial computations.
We bound the variance and covariances in the following parts respectively.

Two signed triangles that do not share any vertices are independent,
which implies that
\begin{equation*}
\E[\tau_{\{1,2,3\}}\tau_{\{4,5,6\}}]
= \E[\tau_{\{1,2,3\}}] \E[\tau_{\{4,5,6\}}].
\end{equation*}
Thus, we have that
\begin{equation*}
V_{\{1,2,3\},\{4,5,6\}} = 0.
\end{equation*}

For two signed triangles sharing a single vertex,
by rotation invariance on the sphere,
if we fix the direction of the shared vertex to be $\e_1$,
they are also independent,
which gives
\begin{equation*}
\begin{split}
\E[\tau_{\{1,2,3\}} \tau_{\{1,4,5\}}]
&= \E[\E[\tau_{\{1,2,3\}} \tau_{\{1,4,5\}}\mid \x_1]]
= \E[\E[\tau_{\{1,2,3\}} \mid \x_1]
\E[\tau_{\{1,4,5\}} \mid \x_1]]\\
&= \E[\E[\tau_{\{1,2,3\}}] \E[\tau_{\{1,4,5\}}]]
=\E[\tau_{\{1,2,3\}}] \E[\tau_{\{1,4,5\}}].
\end{split}
\end{equation*}
Therefore,
\begin{equation*}
V_{\{1,2,3\},\{1,4,5\}} = 0.
\end{equation*}

For two signed triangles on exactly the same vertices,
\begin{equation*}
\E[(\tau_{\{1,2,3\}})^2]
= \E[\bra_{1,2}^2 \bra_{2,3}^2 \bra_{3,1}^2]
= \E_{\X}[\E[\bra_{1,2}^2 \mid \X] \E[\bra_{2,3}^2 \mid \X]
\E[\bra_{3,1}^2 \mid \X]].
\end{equation*}
When $p = 1/2$, $\bra_{1,2}$ is either $1/2$ or $-1/2$.
Hence, $\bra_{1,2}^2 = 1/4$ regardless of $\X$
and other randomness of $a_{1,2}$.
Therefore,
\begin{equation*}
\E[(\tau_{\{1,2,3\}})^2] = \frac{1}{64},
\end{equation*}
which combined with \eqref{eq:Eq_tau123} gives that for $d \ge 8/\pi$,
\begin{equation*}
\frac{1}{128}
\le \frac{1}{64} - \frac{1}{16\pi}\cdot\frac{q^6}{d}
\le V_{\{1,2,3\},\{1,2,3\}}
\le \frac{1}{64} - \frac{1}{8\pi^3}\cdot\frac{q^6}{d}
\le \frac{1}{64}.
\end{equation*}
That is, $V_{\{1,2,3\},\{1,2,4\}}$ is bounded between absolute constants.

As a last step, for a pair of triangles sharing exactly two vertices,
the following lemma provides asymptotically tight bounds for the expectation of
their product.
\begin{lemma} \label{le:tau1/2_var}
The expectation of two signed triangles sharing two vertices
in $\cG(n,p,d,q)$ satisfies
\begin{equation*}
\frac{1}{16\pi^2} \cdot \frac{q^4}{d}
\le \E_{\cG(n, 1/2, d, q)}[\tau_{\{1,2,3\}} \tau_{\{1,2,4\}}]
\le \frac{1}{64} \cdot \frac{q^4}{d}.
\end{equation*}
\end{lemma}
\begin{proof}
By the definition of signed triangles and conditional independence of edges
given $\X$,
\begin{equation*}
\begin{split}
\E[\tau_{\{1,2,3\}}\tau_{\{1,2,4\}}]
&= \E[\bra_{1,2}^2 \bra_{2,3} \bra_{3,1} \bra_{2,4} \bra_{4,1}]
= \E_{\X}[\E[\bra_{1,2}^2 \bra_{2,3} \bra_{3,1} \bra_{2,4} \bra_{4,1}\mid \X]]\\
&= \E_{\X}[\E[\bra_{1,2}^2 \mid \X]
\E[\bra_{2,3} \mid \X] \E[\bra_{3,1} \mid \X]
\E[\bra_{2,4} \mid \X] \E[\bra_{4,1} \mid \X]]\\
&= \frac{1}{4}\E_{\X}[\E[\bra_{2,3} \mid \X] \E[\bra_{3,1} \mid \X]
   \E[\bra_{2,4} \mid \X] \E[\bra_{4,1} \mid \X]].
\end{split}
\end{equation*}
Further by \eqref{eq:E_aij_sij} and rotation invariance, we have
\begin{equation} \label{eq:E_tau123*tau124}
\begin{split}
\E[\tau_{\{1,2,3\}}\tau_{\{1,2,4\}}]
&= \frac{q^4}{4}\E_{\X}[\brs_{2,3}\brs_{3,1}\brs_{2,4}\brs_{4,1}]
 = \frac{q^4}{4}\E_{\x_1, \x_2}[
   \E_{\x_3,\x_4}[\brs_{2,3}\brs_{3,1}\brs_{2,4}\brs_{4,1} \mid \x_1,\x_2]]\\
&= \frac{q^4}{4}\E_{\x_1, \x_2}[
   \E_{\x_3}[\brs_{2,3}\brs_{3,1} \mid \x_1,\x_2]
   \E_{\x_4}[\brs_{2,4}\brs_{4,1} \mid \x_1,\x_2]]\\
&= \frac{q^4}{4}\E_{\x_1, \x_2}[
   \E_{\x_3}[\brs_{2,3}\brs_{3,1} \mid \x_1,\x_2]^2].
\end{split}
\end{equation}
The last equality holds
since $\x_3$ and $\x_4$ are identically distributed.
Recall that $\theta(\x_1, \x_2)$ denotes the angle between $\x_1$ and $\x_2$.
The conditional expectation can be written as
\begin{equation*}
\begin{split}
\E_{\x_3}[\brs_{2,3}\brs_{3,1} \mid \x_1,\x_2]
&= \E_{\x_3}\biggl[\biggl(s_0(\inner{\x_2}{\x_3}) - \frac{1}{2}\biggr)
   \biggl(s_0(\inner{\x_3}{\x_1}) - \frac{1}{2}\biggr)\biggr]\\
&= \P(\inner{\x_2}{\x_3} \ge 0, \inner{\x_3}{\x_1} \ge 0
  \mid \x_1, \x_2) - \frac{1}{4}\\
&= \frac{\pi - \theta(\x_1, \x_2)}{2\pi} - \frac{1}{4}
= \frac{\pi/2 - \theta(\x_1, \x_2)}{2\pi}.
\end{split}
\end{equation*}
Therefore, we have
\begin{equation} \label{eq:E_tau123*tau124_eta}
\E[\tau_{\{1,2,3\}}\tau_{\{1,2,4\}}]
= \frac{q^4}{4} \E\biggl[
\biggl(\frac{\pi/2 - \theta(\x_1, \x_2)}{2\pi}\biggr)^2\biggr].
\end{equation}

Using Proposition~\ref{pr:angle_density}, we can write
\begin{equation*}
\begin{split}
\E\biggl[\biggl(\frac{\pi/2-\theta(\x_1,\x_2)}{2\pi}\biggr)^2\biggr]
&= \int_0^\pi \biggl(\frac{\pi/2 - \theta}{2\pi}\biggr)^2 h(\theta)\,d\theta
= \frac{1}{\zeta}\int_0^{\pi} \frac{(\pi/2-\theta)^2}{4\pi^2}\sin^{d-2}\theta
  \,d\theta\\
&= \frac{1}{\zeta} \int_{-\pi/2}^{\pi/2} \frac{\theta^2}{4\pi^2}
   \cos^{d-2}\theta\,d\theta
= \frac{1}{2\pi^2\zeta} \int_{0}^{\pi/2} \theta^2 \cos^{d-2} \theta\,d\theta,
\end{split}
\end{equation*}
where the last line follows from a change of variables
and the function inside the integral being even.
Applying \eqref{eq:theta_sin}, we have
\begin{equation*}
\frac{1}{2\pi^2\zeta}\int_{0}^{\pi/2} \sin^2\theta \cos^{d-2} \theta\,d\theta
\le \E\biggl[\biggl(\frac{\pi/2-\theta(\x_1,\x_2)}{2\pi}\biggr)^2\biggr]
\le \frac{1}{8\zeta} \int_{0}^{\pi/2} \sin^2\theta \cos^{d-2} \theta\,d\theta.
\end{equation*}
By the definition of $\zeta$,
\begin{equation} \label{eq:int_zeta}
\int_0^{\pi/2} \cos^{d-2}\theta
= \int_0^{\pi/2} \sin^{d-2}\theta
= \frac{\zeta}{2}
= \frac{\sqrt{\pi}\Gamma\bigl(\frac{d-1}{2}\bigr)}
  {2\Gamma\bigl(\frac{d}{2}\bigr)}.
\end{equation}
Since
\begin{equation*}
\int_{0}^{\pi/2} \sin^2\theta\cos^{d-2}\theta\,d\theta
= \int_{0}^{\pi/2} (1 - \cos^2\theta) \cos^{d-2} \theta\,d\theta
= \int_{0}^{\pi/2} \cos^{d-2} \theta\,d\theta
- \int_{0}^{\pi/2} \cos^{d} \theta\,d\theta,
\end{equation*}
by \eqref{eq:int_zeta} and $\Gamma(z+1) = z\Gamma(z)$,
\begin{equation*}
\frac{2}{\zeta}\int_{0}^{\pi/2} \sin^2\theta \cos^{d-2} \theta\,d\theta
= 1 -
\frac{\int_{0}^{\pi/2} \cos^{d} \theta\,d\theta}
{\int_{0}^{\pi/2} \cos^{d-2} \theta\,d\theta}
= 1 - \frac{\Gamma\bigl(\frac{d+1}{2}\bigr)}{\Gamma\bigl(\frac{d+2}{2}\bigr)}
\cdot \frac{\Gamma\bigl(\frac{d}{2}\bigr)}{\Gamma\bigl(\frac{d-1}{2}\bigr)}
= 1 - \frac{d-1}{d} = \frac{1}{d}.
\end{equation*}
Putting them together, we obtain
\begin{equation} \label{eq:eta_bounds}
\frac{1}{4\pi^2}\cdot\frac{1}{d}
\le \E\biggl[\biggl(\frac{\pi/2-\theta(\x_1,\x_2)}{2\pi}\biggr)^2\biggr]
\le \frac{1}{16}\cdot\frac{1}{d}.
\end{equation}

The claim directly follows from combining \eqref{eq:E_tau123*tau124_eta} and
\eqref{eq:eta_bounds}.
\end{proof}
Lemma~\ref{le:tau1/2_var} together with \eqref{eq:Eq_tau123} shows that
when $q \le 1/2$, for absolute constants $C, C' > 0$,
\begin{equation*}
\frac{Cq^4}{d}
\le V_{\{1,2,3\},\{1,2,4\}}
\le \E_{\cG(n, 1/2, d, q)}[\tau_{\{1,2,3\}}\tau_{\{1,2,4\}}]
\le \frac{C'q^4}{d}.
\end{equation*}

Putting the estimates together, we conclude with the following lemma.
\begin{lemma} \label{le:VT_p/2}
There exist absolute constants $C, C' > 0$ such that
for $d \ge 8/\pi$ and $q \le 1/2$,
\begin{equation*}
C\biggl(n^3 + \frac{n^4 q^4}{d}\biggr)
\le \Var[\tau_3(\cG(n, 1/2, d, q))]
\le C'\biggl(n^3 + \frac{n^4 q^4}{d}\biggr).
\end{equation*}
\end{lemma}

\subsubsection{Concluding the proof for
\texorpdfstring{$p = \frac{1}{2}$}{p = 1/2}}
\label{sec:signed_triangles_conclusion}
Combining the estimates in Lemma~\ref{le:ET_p/2} and Lemma~\ref{le:VT_p/2},
we establish that for absolute constants $C > 0$ and $C' < \infty$,
\begin{equation*}
\E[\tau_3(\cG(n, 1/2))] = 0, \quad
\E[\tau_3(\cG(n, 1/2, d, q))] \ge \frac{Cn^3q^3}{\sqrt{d}}
\end{equation*}
and
\begin{equation*}
V_m \coloneqq \max \{\Var[\tau_3(\cG(n, 1/2))],
\Var[\tau_3(\cG(n, 1/2, d, q))]\}
\le C'\biggl(n^3 + \frac{n^4 q^4}{d}\biggr).
\end{equation*}
Let $\Delta \coloneqq \E[\tau_3(\cG(n, 1/2, d, q))]$.
Chebyshev's inequality implies that for an absolute constant $C$,
\begin{equation*}
\P\biggl(\tau_3(\cG(n, 1/2, d, q)) \le \frac{1}{2} \Delta\biggr)
\le \frac{4V_m}{\Delta^2} \le
\frac{C}{2}\biggl(\frac{d}{n^3q^6} + \frac{1}{n^2q^2}\biggr)
\end{equation*}
and
\begin{equation*}
\P\biggl(\tau_3(\cG(n, 1/2)) \ge \frac{1}{2} \Delta\biggr)
\le \frac{4V_m}{\Delta^2} \le
\frac{C}{2}\biggl(\frac{d}{n^3q^6} + \frac{1}{n^2q^2}\biggr).
\end{equation*}
Therefore, we conclude that
\begin{equation}
\begin{split}
\tv{\cG(n, 1/2)}{\cG(n, 1/2, d, q)}
&\ge \P\biggl(\tau_3(\cG(n, 1/2, d, q)) \ge \frac{1}{2} \Delta\biggr)
- \P\biggl(\tau_3(\cG(n, 1/2)) \ge \frac{1}{2} \Delta\biggr)\\
&\ge 1 - C\biggl(\frac{d}{n^3q^6} + \frac{1}{n^2q^2}\biggr).
\end{split}
\end{equation}

Theorem~\ref{th:det_upper} in the case when $p = 1/2$ directly follows.
Note that when $n^3q^6/d \to \infty$,
we also have that
$n^2q^2 = (n^3q^3)^{2/3} \ge (n^3q^3 \cdot q^3/d)^{2/3}
\to \infty$,
since $q^3/d \le 1$.

\subsection{Fixed \texorpdfstring{$p \in (0, 1)$}{0 < p < 1}}
We perform similar analyses as seen previously,
but now for $p \in (0, 1)$.
Our strategy is again to estimate the expectation and variance of
signed triangle statistic in $\cG(n, p, d, q)$,
and then apply Chebyshev's inequality.
Various estimates in the previous work \cite{bubeck2016testing}
largely simplify our calculations.

Recall that the threshold $t_{p, d}$ of $\cG(n, p, d, q)$ is determined by
$\E[s_{t_{p,d}}(\inner{\x_i}{\x_j})] = \P(\inner{\x_i}{\x_j} \ge t_{p,d}) = p$.
Different from when $p=1/2$, for $p \in (0, 1)$,
$t_{p, d}$ is now a function of both $p$ and $d$.
The connection probability can be written as
\begin{equation*}
k_{i,j} = (1 - q)p + q s_{t_{p,d}}(\inner{\x_i}{\x_j}).
\end{equation*}
\subsubsection{Estimating the expectation}
Consider the events
\begin{equation} \label{eq:cherry}
E^\Lambda \coloneqq \{\inner{\x_1}{\x_2} \ge t_{p,d},
\inner{\x_1}{\x_3} \ge t_{p,d}\}
\end{equation}
and
\begin{equation} \label{eq:delta}
E^\Delta \coloneqq \{\inner{\x_1}{\x_2} \ge t_{p,d},
\inner{\x_2}{\x_3} \ge t_{p,d},
\inner{\x_3}{\x_1} \ge t_{p,d}\}.
\end{equation}
By rotation invariance on the sphere, we can fix $\x_1 = \e_1$.
Then,
\begin{equation*}
\begin{split}
\P(E^\Lambda)
&= \P(\inner{\x_1}{\x_2} \ge t_{p,d}, \inner{\x_1}{\x_3} \ge t_{p,d})
= \P(\inner{\e_1}{\x_2} \ge t_{p,d}, \inner{\e_1}{\x_3} \ge t_{p,d})\\
&= \P(\inner{\e_1}{\x_2} \ge t_{p,d}) \P(\inner{\e_1}{\x_3} \ge t_{p,d})
= p^2.
\end{split}
\end{equation*}
The following technical lemma from \cite{bubeck2016testing} provides a lower
bound on the probability of $E^\Delta$.
\begin{lemma}[{\cite[Lemma~1]{bubeck2016testing}}]
\label{le:P_D}
For a fixed $p \in (0, 1)$,
there exists a $C_p > 0$ such that for all $d \ge 1/C_p$,
\begin{equation*}
\P(E^\Delta) \ge p^3\biggl(1 + \frac{C_p}{\sqrt{d}}\biggr).
\end{equation*}
\end{lemma}

By \eqref{eq:gnpd_qxsp},
the expectation of a signed triangle can be written as
\begin{equation*}
\begin{split}
\E_{\cG(n,p,d,q)}[\tau_{\{1,2,3\}}]
&= q^3 \E_{\X}[(s_{t_{p,d}}(\inner{\x_1}{\x_2}) - p)
             (s_{t_{p,d}}(\inner{\x_2}{\x_3}) - p)
             (s_{t_{p,d}}(\inner{\x_3}{\x_1}) - p)]\\
&= q^3 (\P(E^\Delta) - 3p\P(E^\Lambda) + 2p^3)
= q^3 (\P(E^\Delta) - p^3).
\end{split}
\end{equation*}
Using Lemma~\ref{le:P_D}, we have that
\begin{equation*}
\E_{\cG(n,p,d,q)}[\tau_{\{1,2,3\}}] \ge \frac{C_p q^3}{\sqrt{d}}.
\end{equation*}
Therefore, we conclude that there exists a $C_p > 0$ depending only on
$p$ such that for $d \ge 1/C_p$,
\begin{equation} \label{eq:Ep_tau3}
\E_{\cG(n,p,d,q)}[\tau_{3}(G)] \ge \frac{C_p n^3 q^3}{\sqrt{d}}.
\end{equation}

\subsubsection{Estimating the variance}
As in the case when $p = 1/2$,
we bound the variance of the signed triangle count in $\cG(n, p, d, q)$
by estimating the variance and covariances in \eqref{eq:Var_tau3} individually.

Similarly,
the signed triangles with vertex sets $\{1,2,3\}$ and $\{4,5,6\}$ are
still independent.
Hence,
\begin{equation*}
V_{\{1,2,3\},\{4,5,6\}} = 0.
\end{equation*}

For two triangles sharing a single vertex,
by rotation invariance,
we have
\begin{equation*}
\begin{split}
\E[\tau_{\{1,2,3\}}\tau_{\{1,4,5\}}]
&= \E[\E[\tau_{\{1,2,3\}}\tau_{\{1,4,5\}}\mid\x_1]]
= \E[\E[\tau_{\{1,2,3\}} \mid \x_1]\E[\tau_{\{1,4,5\}} \mid \x_1]]\\
&= \E[\E[\tau_{\{1,2,3\}}]\E[\tau_{\{1,4,5\}}]]
= \E[\tau_{\{1,2,3\}}] \E[\tau_{\{1,4,5\}}].
\end{split}
\end{equation*}
Thus again,
\begin{equation*}
V_{\{1,2,3\},\{1,4,5\}} = 0.
\end{equation*}

For two triangles with exactly the same vertices,
\begin{equation*}
\E[(\tau_{\{1,2,3\}})^2]
= \E[\bra_{1,2}^2 \bra_{2,3}^2 \bra_{3,1}^2]
= \E_{\X}[\E[\bra_{1,2}^2\mid\X] \E[\bra_{2,3}^2\mid\X]\E[\bra_{3,1}^2\mid\X]]
\le 1.
\end{equation*}
Hence,
\begin{equation*}
V_{\{1,2,3\},\{1,2,3\}} \le \E[(\tau_{\{1,2,3\}})^2] \le 1.
\end{equation*}

As a last step,
for two triangles sharing two vertices,
\begin{equation*}
\begin{split}
\E[\tau_{\{1,2,3\}}\tau_{\{1,2,4\}}]
&= \E_{\X}[\E[\bra_{1,2}^2\mid\X]
\E[\bra_{2,3} \mid \X] \E[\bra_{3,1} \mid \X]
\E[\bra_{2,4} \mid \X] \E[\bra_{4,1} \mid \X]]\\
&\le \E_{\X}[\E[\bra_{2,3} \mid \X] \E[\bra_{3,1} \mid \X]
           \E[\bra_{2,4} \mid \X] \E[\bra_{4,1} \mid \X]]
= q^4 \E_{\X}[\brs_{2,3}\brs_{3,1}\brs_{2,4}\brs_{4,1}].
\end{split}
\end{equation*}
Recall \eqref{eq:E_tau123*tau124}, which also holds for general $p$:
\begin{equation*}
\E_{\X}[\brs_{2,3}\brs_{3,1}\brs_{2,4}\brs_{4,1}]
= \E_{\x_1, \x_2}[\E_{\x_3}[\brs_{2,3}\brs_{3,1} \mid \x_1,\x_2]^2].
\end{equation*}
An estimate in \cite[Lemma~4]{bubeck2016testing} implies that
\begin{equation} \label{eq:E_s23_s31}
\E_{\x_1,\x_2}[\E_{\x_3}[\brs_{2,3}\brs_{3,1} \mid \x_1,\x_2]^2]
\le \frac{\pi^2}{d}.
\end{equation}
Thus, we have that
\begin{equation*}
\E[\tau_{\{1,2,3\}}\tau_{\{1,2,4\}}] \le \frac{\pi^2q^4}{d}.
\end{equation*}
Therefore, we establish that
\begin{equation*}
V_{\{1,2,3\},\{1,2,4\}} \le \E[\tau_{\{1,2,3\}}\tau_{\{1,2,4\}}]
\le \frac{\pi^2q^4}{d}.
\end{equation*}

Inserting the above estimates in \eqref{eq:Var_tau3},
we conclude that for an absolute constant $C > 0$,
\begin{equation*}
\Var[\tau_3(\cG(n,p,d,q))]
\le C\biggl(n^3 + \frac{n^4q^4}{d}\biggr).
\end{equation*}

\subsubsection{Concluding the proof for
\texorpdfstring{$p \in (0, 1)$}{0 < p < 1}}
From the previous analyses,
for a fixed $p \in (0, 1)$,
there exist constants $C_p > 0$ and $C < \infty$ such that
\begin{equation*}
\E_{\cG(n,p)}[\tau_3(G)] = 0, \quad
\E_{\cG(n,p,d,q)}[\tau_3(G)] \ge \frac{C_pn^3q^3}{\sqrt{d}}
\end{equation*}
and
\begin{equation*}
V_m \coloneqq \max\{\Var[\tau_3(\cG(n, p))],\Var[\tau_3(\cG(n, p, d, q))]\}
\le C\biggl(n^3 + \frac{n^4 q^4}{d}\biggr).
\end{equation*}
Repeating the steps in Section~\ref{sec:signed_triangles_conclusion},
we conclude that there exists a $C_p < \infty$, such that for $d \ge C_p$,
\begin{equation*}
\tv{\cG(n, p)}{\cG(n, p, d, q)}
\ge 1 - C_p\biggl(\frac{d}{n^3q^6} + \frac{1}{n^2q^2}\biggr).
\end{equation*}
Theorem~\ref{th:det_upper} for a fixed $p \in (0, 1)$ is thus established.

\subsection{Signed Cliques}
The method introduced in the preceding subsections proves
an upper bound for detecting the geometry in $\cG(n,p,d,q)$,
while providing an asymptotically powerful test
that is computationally efficient.
However, the upper bound for detection does not match the lower bound in
Theorem~\ref{th:det_lower}.
As a final remark,
we explore whether the possibility results for detection can be improved
via generalizations of the signed triangle statistic.
Two families of extensions are studied: signed cliques and signed cycles.
We show by special examples of subgraphs on four vertices,
as well as those on a fixed number of vertices,
that it is unlikely the detection boundary can be improved with them.

A first generalization of the signed triangle is by increasing
the number of vertices in the set,
resulting in the signed induced complete subgraphs of $G$,
which we simply call \emph{signed cliques}.
Similar to the case of the signed triangle,
let $S \subset V$ be a subset of vertices of $G$ with cardinality
$\lvert S\rvert = k$,
where $k \le n$ is fixed.
$T_S$ is again the indicator that
the edges over the vertex set $S$ form a clique;
namely, the induced subgraph is complete.
Given the adjacency matrix $\A$ of $G$,
$T_S$ can be expressed by
\begin{equation}
T_S = \prod_{\{i, j\} \subset S} a_{i,j}.
\end{equation}
Then, the total number of cliques of size $k$ in $G$,
denoted by $T_k(G)$,
can be written as
\begin{equation}
T_k(G) \coloneqq \sum_{S \in \binom{V}{k}} T_S.
\end{equation}
For a constant $p \in [0,1]$,
define the signed indicator and its count in $G$ by
\begin{equation}
\tau_S \coloneqq \prod_{\{i, j\} \subset S} (a_{i,j} - p)
\quad \text{and} \quad
\tau_k(G) \coloneqq \sum_{S \in \binom{V}{k}} \tau_S.
\end{equation}

We first compute the expectation and variance of the signed clique statistic
in $\cG(n, p)$.

For $\cG(n,p)$, since all edges are independent,
\begin{equation*}
\E[\tau_{[k]}] = \prod_{\{i,j\} \subset [k]}\E[a_{i,j} - p] = 0.
\end{equation*}
Then, the expectation of the signed clique statistic satisfies
\begin{equation} \label{eq:E_tauk_gnp}
\E[\tau_k(\cG(n,p))] = \binom{n}{k}\E[\tau_{[k]}] = 0.
\end{equation}

Consider two sets of vertices $S$ and $S'$ of size k.
If $S = S'$, we have
\begin{equation*}
\E[\tau_{S}\tau_{S'}]
= \E[(\tau_{S})^2] = \prod_{\{i,j\} \in V_1} \E[(a_{i,j} - p)^2]
= (p(1-p))^{\binom{k}{2}}
= (p(1-p))^{k(k-1)/2}.
\end{equation*}
For $S \ne S'$, there is at least one signed edge that appears in
$\tau_{S}$ but not in $\tau_{S'}$.
Suppose this edge is $e$.
By the independence of edges in $\cG(n,p)$,
\begin{equation*}
\E[\tau_{S}\tau_{S'}]
= \E[a_e - p]
  \E\biggl[\tau_{S'}\prod_{e' \in \binom{S}{2}\backslash \{e\}}(a_{e'}-p)\biggr]
= 0.
\end{equation*}
Therefore, the variance of the signed clique statistic in $\cG(n,p)$ satisfies
\begin{equation} \label{eq:V_tauk_gnp}
\begin{split}
\Var[\tau_k(\cG(n,p))]
&= \E\biggl[\biggl(\sum_{S \in \binom{n}{k}} \tau_S\biggr)^2\biggr]
= \sum_{S,S' \in \binom{n}{k}} \E[\tau_{S}\tau_{S'}]
= \binom{n}{k} \E[(\tau_{[k]})^2]\\
&= \binom{n}{k} (p(1-p))^{k(k-1)/2}
\ge C_{k,p} n^k
\end{split}
\end{equation}
for some $C_{k,p} > 0$ depending only on $k$ and $p$.

\subsubsection{Signed quadruples}
For $\cG(n,p,d,q)$,
we start with the special case when $p = 1/2$
and consider the signed clique on four vertices,
called the \emph{signed quadruple}.

\begin{theorem} \label{th:E_quad}
There exists an absolute constant $C$ such that
\begin{equation*}
\abs{\E[\tau_4(\cG(n,1/2,d,q))]} \le \frac{Cn^4q^6}{d}.
\end{equation*}
\end{theorem}
Theorem~\ref{th:E_quad} together with \eqref{eq:E_tauk_gnp} shows that
\begin{equation*}
\abs{\E[\tau_4(\cG(n,1/2,d,q))]-\E[\tau_4(\cG(n,1/2))]} \le \frac{Cn^4q^6}{d}.
\end{equation*}

As we shall see, a lower bound on the variance of the signed quadruple statistic
in $\cG(n,1/2,d,q)$ can be obtained from a more
general argument in Lemma~\ref{le:min_var_clique},
which combined with \eqref{eq:V_tauk_gnp} gives
\begin{equation*}
\min \{\Var[\tau_4(\cG(n,1/2))], \Var[\tau_4(\cG(n,1/2,d,q))]\}
\ge C_{k,p} n^4
\end{equation*}
for some $C_{k,p} > 0$.
Therefore, there exists a constant $C_{k,p} < \infty$ such that
\begin{equation*}
\frac{(\E[\tau_4(\cG(n,1/2,d,q))]-\E[\tau_4(\cG(n,1/2))])^2}
{\min \{\Var[\tau_4(\cG(n,1/2))], \Var[\tau_4(\cG(n,1/2,d,q))]\}}
\le \frac{C_{k,p}n^4 q^{12}}{d^2}.
\end{equation*}
The above display implies that detecting geometry using
the previous method with a signed quadruple statistic is only possible
if $n^2q^6/d \to \infty$.
We see that this is stronger than the condition $n^3q^6/d \to \infty$ given
by the signed triangle.
Note that we use the lower bound on the minimum of the variances instead of
the maximum so that testing either hypothesis is not possible.

We prove Theorem~\ref{th:E_quad} in the following.
A key estimation is that the expected signed quadruple
in $\cG(n,1/2,d)$ is at most of order $1/d$,
formally stated as the following lemma.
\begin{lemma} \label{le:E_tau4_gn1/2}
There exists an absolute constant $C$ such that
\begin{equation*}
\abs{\E_{\cG(n,1/2,d)}[\tau_{[4]}]} \le \frac{C}{d}.
\end{equation*}
\end{lemma}
The proof of Lemma~\ref{le:E_tau4_gn1/2} is divided into estimating several
quantities.

As computed before,
for a sample from $\cG(n,1/2,d)$ with adjacency matrix $\A=[a_{i,j}]$,
by conditioning on $\x_1$ and rotation invariance,
\begin{equation*}
\E[a_{1,2}a_{1,3}]= \E[\E[a_{1,2}a_{1,3}\mid\x_1]]
= \E[\E[a_{1,2}\mid\x_1]\E[a_{1,3}\mid\x_1]]= \biggl(\frac{1}{2}\biggr)^2.
\end{equation*}
We also have
$\E[a_{1,2}] = \frac{1}{2}$
by definition.

The expected signed quadruple in $\cG(n,1/2,d)$ can be written as
\begin{equation} \label{eq:E_tau4_expand}
\begin{split}
&\E_{\cG(n,1/2,d)}[\tau_{[4]}]
= \E\biggl[\prod_{i<j}^4 \bra_{i,j}\biggr]
= \E\biggl[\prod_{i<j}^4 \biggl(a_{i,j}-\frac{1}{2}\biggr)\biggr]
- \biggl(\frac{1}{2} - \frac{1}{2}\biggr)^6\\
&\qquad=
\underbrace{\binom{6}{6}\biggl(\E[a_{1,2}a_{1,3}a_{1,4}a_{2,3}a_{2,4}a_{3,4}]
   - \biggl(\frac{1}{2}\biggr)^6\biggr)}_{Q_1}
- \underbrace{\binom{6}{5}\biggl(\frac{1}{2}\biggr)
    \biggl(\E[a_{1,2}a_{1,3}a_{2,3}a_{1,4}a_{2,4}]
    - \biggl(\frac{1}{2}\biggr)^5\biggr)}_{Q_2}\\
&\qquad\phantom{{}={}}
+ \underbrace{\binom{6}{4}\frac{1}{5}\biggl(\frac{1}{2}\biggr)^2
              \biggl(\E[a_{1,3}a_{2,3}a_{1,4}a_{2,4}]
              - \biggl(\frac{1}{2}\biggr)^4\biggr)}_{Q_3}
+ \underbrace{\binom{6}{4}\frac{4}{5}\biggl(\frac{1}{2}\biggr)^2
      \biggl(\E[a_{1,2}a_{2,3}a_{3,1}a_{1,4}]
    - \biggl(\frac{1}{2}\biggr)^4\biggr)}_{Q_4}\\
&\qquad\phantom{{}={}}
- \underbrace{\binom{6}{3}\frac{1}{5}\biggl(\frac{1}{2}\biggr)^3
              \biggl(\E[a_{1,2}a_{2,3}a_{3,1}]
- \biggl(\frac{1}{2}\biggr)^3\biggr)}_{Q_5},
\end{split}
\end{equation}
where the fractions are from simple combinatorial calculations.
In the following we compute and estimate $Q_{1}, \ldots, Q_{5}$.

Following the definitions in the proof of Lemma~\ref{le:tri_p=1/2},
we define the density
\begin{equation*}
h(\theta) \coloneqq \frac{1}{\zeta} \sin^{d-2}\theta,
\end{equation*}
where $\theta \in [0, \pi]$ and the normalization factor
\begin{equation*}
\zeta \coloneqq \int_0^\pi \sin^{d-2}\theta\,d\theta
= \frac{\sqrt{\pi} \Gamma\bigl(\frac{d-1}{2}\bigr)}
{\Gamma\bigl(\frac{d}{2}\bigr)}.
\end{equation*}

Let
\begin{equation} \label{eq:def_gamma}
\gamma \coloneqq \int_0^{\pi/2} \frac{\pi/2 - \theta}{2\pi} h(\theta)\,d\theta.
\end{equation}

By the computation in Lemma~\ref{le:tri_p=1/2},
$\E[a_{1,2}a_{2,3}a_{3,1}] - 1/8 = \gamma$.
Hence, we have
\begin{equation*}
Q_5 = \frac{1}{2}\gamma.
\end{equation*}
Lemma~\ref{le:tri_p=1/2} also shows that
\begin{equation*}
\frac{1}{2\pi\sqrt{2\pi}}\cdot\frac{1}{\sqrt{d}}
\le \gamma
\le \frac{1}{4\sqrt{\pi}}\cdot\frac{1}{\sqrt{d}}.
\end{equation*}

By conditional independence of $a_{1,4}$ and $a_{1,2}a_{2,3}a_{3,1}$
and rotation invariance,
\begin{equation*}
\begin{split}
\E[a_{1,2}a_{2,3}a_{3,1}a_{4,1}]
&= \E_{\x_1}[\E[a_{1,2}a_{2,3}a_{3,1}a_{1,4}\vert \x_1]]
= \E_{\x_1}[\E[a_{1,2}a_{2,3}a_{3,1}\vert \x_1]\E[a_{1,4}\vert \x_1]]\\
&= \frac{1}{2}\E[a_{1,2}a_{2,3}a_{3,1}].
\end{split}
\end{equation*}
Hence, we have
\begin{equation*}
Q_4 = \frac{3}{2}\gamma.
\end{equation*}

Let
\begin{equation} \label{eq:def_eta}
\eta \coloneqq \int_0^{\pi/2}\biggl(\frac{\pi/2-\theta}{2\pi}\biggr)^2h(\theta)
\,d\theta.
\end{equation}
By \eqref{eq:eta_bounds},
\begin{equation*}
\frac{1}{4\pi^2}\cdot\frac{1}{d} \le \eta \le \frac{1}{16}\cdot\frac{1}{d}.
\end{equation*}

\begin{lemma} \label{le:sqr}
Let $\eta$ be defined in \eqref{eq:def_eta}.
Then,
\begin{equation*}
\E[a_{1,3}a_{2,3}a_{1,4}a_{2,4}] = \frac{1}{16} + 2\eta.
\end{equation*}
\end{lemma}
\begin{proof}
By conditional independence of $a_{1,3}a_{2,3}$ and $a_{1,4}a_{2,4}$,
\begin{equation*}
\begin{split}
\E[a_{1,3}a_{2,3}a_{1,4}a_{2,4}]
&= \E[\E[a_{1,3}a_{2,3}a_{1,4}a_{2,4}\mid \x_1, \x_2]]
= \E[\E[a_{1,3}a_{2,3}\mid \x_1, \x_2]\E[a_{1,4}a_{2,4}\mid \x_1, \x_2]]\\
&= \E[\E[a_{1,3}a_{2,3}\mid \x_1, \x_2]^2].
\end{split}
\end{equation*}
The last equality is because $\x_3$ and $\x_4$ are identically distributed.

Similar to the proof of Lemma~\ref{le:tri_p=1/2},
we can fix the space spanned by $\x_1$ and $\x_2$.
The angle between them has the density $h(\theta)$
given in Proposition~\ref{pr:angle_density}.
We have that $a_{1,3}a_{2,3} = 1$ if and only if the projection of $\x_3$ onto
this plane lies in $[\theta-\pi/2,\pi/2]$.
Therefore,
\begin{equation} \label{eq:quadrilateral}
\E[a_{1,3}a_{2,3}a_{1,4}a_{2,4}]
= \int_0^\pi \biggl(\frac{\pi-\theta}{2\pi}\biggr)^2h(\theta)\,d\theta
= \int_0^{\pi/2} \biggl(\frac{\pi-\theta}{2\pi}\biggr)^2h(\theta)\,d\theta
  + \int_{\pi/2}^\pi \biggl(\frac{\pi-\theta}{2\pi}\biggr)^2h(\theta)\,d\theta.
\end{equation}
For the first integral in the above display,
\begin{equation*}
\int_0^{\pi/2} \biggl(\frac{\pi-\theta}{2\pi}\biggr)^2h(\theta)\,d\theta
= \int_0^{\pi/2}
  \biggl(\frac{\pi/2 + \pi/2-\theta}{2\pi}\biggr)^2h(\theta)\,d\theta
= \frac{1}{32} + \frac{1}{2}\gamma + \eta.
\end{equation*}
For the second integral,
by the symmetry of the $\sin$ function,
\begin{equation*}
\int_{\pi/2}^\pi \biggl(\frac{\pi-\theta}{2\pi}\biggr)^2h(\theta)\,d\theta
= \int_0^{\pi/2} \biggl(\frac{\theta}{2\pi}\biggr)^2h(\theta)\,d\theta
= \int_0^{\pi/2} \biggl(\frac{\pi/2 - (\pi/2 - \theta)}{2\pi}\biggr)^2
  h(\theta)\,d\theta
= \frac{1}{32} - \frac{1}{2}\gamma + \eta.
\end{equation*}
The claim directly follows by adding them.
\end{proof}
By Lemma~\ref{le:sqr}, we have that
\begin{equation*}
Q_3 = \frac{3}{2}\eta.
\end{equation*}

Since
\begin{equation*}
\E[a_{1,2}a_{1,3}a_{2,3}a_{1,4}a_{2,4}]
= \int_0^{\pi/2} \biggl(\frac{\pi-\theta}{2\pi}\biggr)^2h(\theta)\,d\theta,
\end{equation*}
which equals the first integral in \eqref{eq:quadrilateral},
we directly have the following lemma.
\begin{lemma}
Let $\gamma$ and $\eta$ be defined in \eqref{eq:def_gamma} and
\eqref{eq:def_eta} respectively.
Then,
\begin{equation*}
\E[a_{1,2}a_{1,3}a_{2,3}a_{1,4}a_{2,4}]
= \frac{1}{32} + \frac{1}{2}\gamma + \eta.
\end{equation*}
\end{lemma}
Therefore, we have
\begin{equation*}
Q_2 = \frac{3}{2}\gamma + 3\eta.
\end{equation*}

Plugging the previous estimates into \eqref{eq:E_tauk_gnp},
we obtain that
\begin{equation} \label{eq:tau}
\E[\tau_{[4]}] = Q_1 - Q_2 + Q_3 + Q_4 - Q_5
= Q_1 - \frac{1}{2}\gamma - \frac{3}{2}\eta.
\end{equation}
An estimation for $Q_1$ is provided in the following lemma.
\begin{lemma} \label{le:Q1}
Let $\gamma$ and $\eta$ be defined before.
Then,
\begin{equation*}
\frac{1}{2}\gamma + \frac{1}{2}\eta + \frac{1}{16\pi^2}\cdot\frac{1}{d}
\le \E\biggl[\prod_{i<j}^4 a_{i,j}\biggr] - \frac{1}{64}
\le \frac{1}{2}\gamma + \frac{1}{2}\eta + \frac{1}{8\pi}\cdot\frac{1}{d}.
\end{equation*}
\end{lemma}

The proof of Lemma~\ref{le:Q1} involves extending the argument in the proof of
Lemma~\ref{le:tri_p=1/2} to a three-dimensional subspace.
Before proving Lemma~\ref{le:Q1}, we show the following claim concerning the
distribution of the angle between a uniform random vector in $\RR^d$ and an
arbitrary two-dimensional plane.
\begin{proposition} \label{pr:ang_den_phi}
Let $\x$ be a uniform random point in $\SS^{d-1}$.
Let $\varphi \in [0,\pi/2]$ be the angle between the vector $\x$
and any fixed $2$-dimensional subspace.
Then, the density $g(\varphi)$ satisfies
\begin{equation*}
g(\varphi) =(d-2)\sin^{d-3}\varphi\cos\varphi.
\end{equation*}
\end{proposition}

\begin{proof}
Let $\z \sim \cN(\0,\I_d)$ be a $d$-dimensional random vector.
Then, $\hat{\z} \coloneqq \z/\norm{\z}$ is a uniform random point
in $\SS^{d-1}$.
By rotation invariance, we can fix the plane to be that spanned by the first
two vectors of the standard basis.
The projection of $\hat{\z}$ onto this plane is
$\tilde{\z} = (z_1, z_2, 0, \dotsc, 0)/\norm{\z}$.
Then, the angle between $\x$ and the plane is equal to the angle between
$\hat{\z}$ and $\wt{\z}$.
Hence, the cumulative distribution function satisfies
\begin{equation*}
F(\varphi)
=\P\biggl(\arccos \frac{\hat{\z}\cdot\wt{\z}}{\norm{\wt{\z}}} \le \varphi\biggr)
=\P\biggl(\frac{\hat{\z}\cdot\wt{\z}}{\norm{\wt{\z}}} \ge \cos\varphi\biggr)
= \P\biggl(\frac{z_1^2 + z_2^2}{\sum_{i=1}^d z_i^2} \ge \cos^2\varphi\biggr).
\end{equation*}
Since the $z_i$'s are standard normal random variables,
we have $z_1^2 + z_2^2 \sim \chi^2(2)$
and $\sum_{i=3}^d z_i^2 \sim \chi^2(d-2)$,
and these are independent.
Therefore, $(z_1^2+z_2^2)/\sum_{i=1}^d z_i^2$ has a
$\mathrm{Beta}(1,\frac{d-2}{2})$ distribution.
Hence, by the definition of the beta distribution,
\begin{equation*}
\P\biggl(\frac{z_1^2 + z_2^2}{\sum_{i=1}^d z_i^2} \ge \cos^2\varphi\biggr)
= \frac{\Gamma\bigl(\frac{d}{2}\bigr)}{\Gamma\bigl(\frac{d-2}{2}\bigr)}
  \int_{\cos^2\varphi}^1 (1-x)^{d/2-2}\,dx
= \frac{d-2}{2}\int_{\cos^2\varphi}^1(1-x)^{d/2-2}\,dx.
\end{equation*}
Taking the derivative with respect to $\varphi$,
we obtain
\begin{equation*}
g(\varphi) = -\frac{d-2}{2}(1-\cos^2\varphi)^{d/2-2}(-2\cos\varphi\sin\varphi)
= (d-2)\sin^{d-3}\varphi\cos\varphi.\qedhere
\end{equation*}
\end{proof}

\begin{proof}[Proof of Lemma~\ref{le:sqr}]
Consider the space spanned by $\x_1, \x_2, \x_3$.
Without loss of generality,
we can fix the coordinates as follows:
\begin{align*}
\x_1 &= (1, 0, 0, 0, \dotsc, 0),\\
\x_2 &= (\cos\theta, \sin\theta, 0, \dotsc, 0),\\
\x_3 &= (\cos\varphi \cos\psi, \cos\varphi \sin\psi, \sin\varphi, 0, \dotsc, 0).
\end{align*}
By symmetry on the sphere,
we can constrain the parameters in the following space:
\begin{align*}
\theta &\in [0, \pi],\\
\psi &\in [-\pi, \pi],\\
\varphi &\in [0, \pi/2].
\end{align*}

Let $f(\theta, \psi, \varphi)$ be the probability density function.
Then, by independence of the vectors,
\begin{equation*}
f(\theta, \psi, \varphi)
= \frac{1}{2\pi} h(\theta)g(\varphi)
= \frac{d-2}{2\pi\zeta}\sin^{d-2}\theta \sin^{d-3}\varphi\cos\varphi.
\end{equation*}

Denoting by $\theta(\x,\y) \in [0, \pi]$ the angle between two
$d$-dimensional vectors, we also have
\begin{align}
\cos\theta(\x_1, \x_2) &= \inner{\x_1}{\x_2} = \cos\theta,\label{eq:ang_12}\\
\cos\theta(\x_1, \x_3) &= \inner{\x_1}{\x_3} = \cos\varphi\cos\psi,
\label{eq:ang_13}\\
\cos\theta(\x_2, \x_3) &= \inner{\x_2}{\x_3} = \cos\varphi\cos\psi\cos\theta
+ \cos\varphi\sin\psi\sin\theta
= \cos\varphi \cos(\psi - \theta)\label{eq:ang_23}.
\end{align}

The event $1 \sim 2$ happens if and only if $\theta \in [0,\pi]$.
Vertex $3$ is connected to both $1$ and $2$ if and only if the projection of
$\x_3$
onto the plane determined by $\x_1$ and $\x_2$ forms an angle no greater than
$\pi/2$ with both $\x_1$ and $\x_2$.
Therefore, we have $\psi \in [\theta-\pi/2, \pi/2]$.
The last vertex is connected to all of them
if and only if the direction of $\x_4$ falls in
the spherical triangle determined by the three half planes with normal vectors
$\x_1, \x_2, \x_3$ respectively.
As a well-known fact
(see, e.g., \cite[proposition 99]{todhunter1863spherical}),
the surface area of the spherical triangle equals the spherical excess
defined by
\begin{equation*}
S
\coloneqq 2\pi - \theta(\x_1, \x_2) - \theta(\x_1, \x_3) - \theta(\x_2, \x_3)
= \biggl(\pi - \theta\biggr) + \biggl(\frac{\pi}{2} - \theta(\x_1, \x_3)\biggr)
  + \biggl(\frac{\pi}{2} - \theta(\x_2, \x_3)\biggr).
\end{equation*}

Since the surface area of the sphere is $4\pi$,
the probability that the four vertices form a clique is
\begin{equation} \label{eq:sp_int}
\begin{split}
\int_0^{\pi/2}\int_0^{\pi/2}\int_{\theta-\pi/2}^{\pi/2}
\frac{S}{4\pi}f(\theta,\psi,\varphi)\,d\psi\,d\theta\,d\varphi
&= \underbrace{\int_0^{\pi/2}\int_0^{\pi/2}\int_{\theta-\pi/2}^{\pi/2}
\frac{\pi-\theta}{4\pi}f(\theta,\psi,\varphi)\,d\psi\,d\theta\,d\varphi}_{I_1}\\
&\quad + \underbrace{\int_0^{\pi/2}\int_0^{\pi/2}\int_{\theta-\pi/2}^{\pi/2}
\frac{\pi/2-\theta(\x_1, \x_3)}{4\pi}f(\theta,\psi,\varphi)
\,d\psi\,d\theta\,d\varphi}_{I_2}\\
&\quad + \underbrace{\int_0^{\pi/2}\int_0^{\pi/2}\int_{\theta-\pi/2}^{\pi/2}
\frac{\pi/2-\theta(\x_2, \x_3)}{4\pi}f(\theta,\psi,\varphi)
\,d\psi\,d\theta\,d\varphi}_{I_3}.
\end{split}
\end{equation}
We deal with the three integrals separately as follows.

By integrating over $\psi$ and $\varphi$,
\begin{equation*}
I_1
= \frac{1}{2} \int_0^{\pi/2}\biggl(\frac{\pi - \theta}{2\pi}\biggr)^2 h(\theta)
\,d\theta
= \frac{1}{2} \int_0^{\pi/2}\biggl(\frac{\pi/2 + (\pi/2-\theta)}{2\pi}\biggr)^2
h(\theta)\,d\theta
= \frac{1}{2}\biggl(\frac{1}{32} + \frac{1}{2}\gamma + \eta\biggr).
\end{equation*}

Plugging \eqref{eq:ang_13} into $I_2$,
we have
\begin{equation} \label{eq:sp_b}
\begin{split}
I_2
&= \frac{1}{4\pi}\int_0^{\pi/2}\int_0^{\pi/2}\int_{\theta-\pi/2}^{\pi/2}
\biggl(\frac{\pi}{2}-\arccos(\cos\varphi\cos\psi)\biggr)f(\theta,\psi,\varphi)
\,d\psi\,d\theta\,d\varphi\\
&= \frac{1}{4\pi}\int_0^{\pi/2}\int_0^{\pi/2}\int_{\theta-\pi/2}^{\pi/2}
\arcsin(\cos\varphi\cos\psi)f(\theta,\psi,\varphi)
\,d\psi\,d\theta\,d\varphi.
\end{split}
\end{equation}
The Taylor expansion of $\arcsin$ gives
\begin{equation*}
\arcsin x = \sum_{n=0}^\infty a_n x^{2n+1},
\end{equation*}
where
\begin{equation*}
a_n \coloneqq \frac{(2n)!}{4^n(n!)^2(2n+1)}.
\end{equation*}
Hence, we get
\begin{equation} \label{eq:sp_taylor}
\arcsin(\cos\varphi\cos\psi)
= \sum_{n=0}^\infty a_n \cos^{2n+1}\varphi\cos^{2n+1}\psi.
\end{equation}
Inserting the expansion \eqref{eq:sp_taylor} into \eqref{eq:sp_b}
and interchanging the summation and integration,
we have
\begin{equation*}
\begin{split}
I_2
&= \frac{1}{4\pi}\int_0^{\pi/2}\int_0^{\pi/2}\int_{\theta-\pi/2}^{\pi/2}
\sum_{n=0}^\infty a_n(\cos^{2n+1}\varphi)(\cos^{2n+1}\psi)f(\theta,\psi,\varphi)
\,d\psi\,d\theta\,d\varphi\\
&= \frac{1}{8\pi^2}\sum_{n=0}^\infty a_n
\int_0^{\pi/2}g(\varphi)\cos^{2n+1}\varphi\,d\varphi
\int_0^{\pi/2}h(\theta)\int_{\theta-\pi/2}^{\pi/2}
\cos^{2n+1}\psi\,d\psi\,d\theta.
\end{split}
\end{equation*}
Since
\begin{equation} \label{eq:psi_split}
\int_{\theta-\pi/2}^{\pi/2} \cos^{2n+1} \psi\,d\psi
= \int_{0}^{\pi/2} \cos^{2n+1} \psi\,d\psi
  + \int_{\theta-\pi/2}^{0} \cos^{2n+1} \psi\,d\psi,
\end{equation}
the integral $I_2$ can also be split into two integrals accordingly:
\begin{equation*}
\begin{split}
I_2
&= \underbrace{\frac{1}{8\pi^2}\sum_{n=0}^\infty a_n
\int_0^{\pi/2}g(\varphi)\cos^{2n+1}\varphi\,d\varphi
\int_0^{\pi/2}h(\theta)\int_{0}^{\pi/2}
\cos^{2n+1}\psi\,d\psi\,d\theta}_{I_{2(a)}}\\
&\phantom{{}={}}+ \underbrace{\frac{1}{8\pi^2}\sum_{n=0}^\infty a_n
\int_0^{\pi/2}g(\varphi)\cos^{2n+1}\varphi\,d\varphi
\int_0^{\pi/2}h(\theta)\int_{\theta-\pi/2}^{0}
\cos^{2n+1}\psi\,d\psi\,d\theta}_{I_{2(b)}}.
\end{split}
\end{equation*}
We deal with $I_{2(a)}$ and $I_{2(b)}$ separately.

Using the definition of $g(\varphi)$ in Proposition~\ref{pr:ang_den_phi}
and by a change of variables $x = \sin^2\varphi$,
we have
\begin{equation*}
\begin{split}
\int_0^{\pi/2}g(\varphi)\cos^{2n+1}\varphi\,d\varphi
&= (d-2)\int_0^{\pi/2}\cos^{2n+1}\varphi\sin^{d-3}\varphi\cos\varphi\,d\varphi\\
&= \frac{d-2}{2} \int_0^1 (1 - x)^{n+1/2} x^{d/2-2} \,dx
= \frac{d-2}{2}\mathrm{B}\biggl(\frac{d}{2} - 1, n+\frac{3}{2}\biggr)\\
&= \frac{(d-2)\Gamma\bigl(\frac{d}{2} - 1\bigr)\Gamma\bigl(n+\frac{3}{2}\bigr)}
  {2\Gamma\bigl(n+\frac{d}{2}+\frac{1}{2}\bigr)}
= \frac{\Gamma\bigl(\frac{d}{2}\bigr)\Gamma\bigl(n+\frac{3}{2}\bigr)}
  {\Gamma\bigl(n+\frac{d}{2}+\frac{1}{2}\bigr)},
\end{split}
\end{equation*}
where the last equality is due to the identity $x\Gamma(x) = \Gamma(x+1)$.

By \eqref{eq:int_zeta},
\begin{equation*}
\int_{0}^{\pi/2} \cos^{2n+1} \psi\,d\psi
= \frac{\sqrt{\pi}\Gamma(n+1)}{2\Gamma\bigl(n+\frac{3}{2}\bigr)}.
\end{equation*}

Combining them and rearranging the terms, we have
\begin{equation*}
\frac{\Gamma\bigl(\frac{d}{2}\bigr)\Gamma\bigl(n+\frac{3}{2}\bigr)}
{\Gamma\bigl(n+\frac{d}{2}+\frac{1}{2}\bigr)}
\cdot \frac{\sqrt{\pi}\Gamma(n+1)}{2\Gamma\bigl(n+\frac{3}{2}\bigr)}
= \frac{\sqrt{\pi}\Gamma\bigl(\frac{d}{2}\bigr)}
  {\Gamma\bigl(\frac{d-1}{2}\bigr)}
  \cdot \frac{\Gamma\bigl(\frac{d-1}{2}\bigr)\Gamma(n+1)}
        {2\Gamma\bigl(n+\frac{d}{2}+\frac{1}{2}\bigr)}
= \frac{\pi}{\zeta} \cdot \frac{1}{2}\rB\biggl(\frac{d-1}{2},n+1\biggr).
\end{equation*}
By the definition of the beta function,
\begin{equation*}
\frac{1}{2}\mathrm{B}\biggl(\frac{d-1}{2},n+1\biggr)
= \frac{1}{2}\int_0^1 x^{d/2-3/2}(1-x)^n\,dx
= \int_0^{\pi/2} \cos^{2n+1}\theta \sin^{d-2}\theta\,d\theta.
\end{equation*}
Interchanging the summation and integration,
\begin{equation*}
\begin{split}
\sum_{n=0}^\infty a_n \int_0^{\pi/2} \cos^{2n+1}\theta \sin^{d-2}\theta\,d\theta
&= \int_0^{\pi/2}\sum_{n=0}^\infty a_n\cos^{2n+1}\theta \sin^{d-2}\theta
  \,d\theta\\
&= \int_0^{\pi/2} \arcsin(\cos\theta) \sin^{d-2}\theta\,d\theta\\
&= \int_0^{\pi/2} \biggl(\frac{\pi}{2}-\theta\biggr)\sin^{d-2}\theta\,d\theta
= 2\pi \zeta \gamma.
\end{split}
\end{equation*}

Further, by symmetry of $\sin\theta$,
\begin{equation*}
\int_0^{\pi/2}h(\theta)\,d\theta = \frac{1}{2}.
\end{equation*}

Putting them together, we obtain
\begin{equation*}
I_{2(a)}
= \frac{1}{8\pi^2} \cdot \frac{1}{2} \cdot \frac{\pi}{\zeta}
  \cdot 2\pi\zeta\gamma
= \frac{1}{8}\gamma.
\end{equation*}

Now we turn to $I_{2(b)}$ and show upper and lower bounds on it.

By the symmetry of $\cos\psi$ and a change of variables $x = \sin\varphi$,
\begin{equation*}
\int_{\theta-\pi/2}^{0} \cos^{2n+1} \psi\,d\psi
=\int_{0}^{\pi/2-\theta} \cos^{2n+1} \psi\,d\psi
= \int_{0}^{\pi/2-\theta} \cos^{2n} \psi\,d\sin\psi
= \int_0^{\cos\theta} (1-x^2)^n\,dx.
\end{equation*}

Since for $x \in [0, \cos\theta]$,
\begin{equation*}
\sin^{2n}\theta \le (1-x^2)^n \le 1,
\end{equation*}
we have
\begin{equation*}
\sin^{2n}\theta \cos\theta
\le \int_{\theta-\pi/2}^{0} \cos^{2n+1} \psi\,d\psi
\le \cos\theta.
\end{equation*}

Integrating over the density of $\theta$ gives
\begin{equation*}
\begin{split}
\int_0^{\pi/2} h(\theta) \sin^{2n}\theta \cos\theta\,d\theta
&= \frac{1}{\zeta}\int_0^{\pi/2} \sin^{2n+d-2}\theta \cos\theta\,d\theta
= \frac{1}{\zeta}\int_0^{\pi/2} \sin^{2n+d-2}\theta\,d\sin\theta\\
&= \frac{1}{(2n+d-1)\zeta}
\end{split}
\end{equation*}
and
\begin{equation*}
\int_0^{\pi/2} h(\theta)\cos\theta\,d\theta
= \frac{1}{\zeta}\int_0^{\pi/2} \sin^{d-2}\theta \cos\theta\,d\theta
= \frac{1}{\zeta}\int_0^{\pi/2} \sin^{d-2}\theta\,d\sin\theta
= \frac{1}{(d-1)\zeta}.
\end{equation*}

We deal with the upper bound first.
Interchanging the summation and integration yields
\begin{equation*}
\begin{split}
\sum_0^\infty a_n \int_0^{\pi/2} g(\varphi)\cos^{2n+1}\varphi\,d\varphi
&=\int_0^{\pi/2}\sum_0^\infty a_n g(\varphi)\cos^{2n+1}\varphi\,d\varphi
= \int_0^{\pi/2} g(\varphi) \arcsin(\cos\varphi)\,d\varphi\\
&= (d-2)\int_0^{\pi/2}\biggl(\frac{\pi}{2}-\varphi\biggr)
   \sin^{d-3}\varphi\cos\varphi\,d\varphi.
\end{split}
\end{equation*}
We can derive an upper bound on the above display by the upper bound in
\eqref{eq:theta_sin}.
Since $\varphi \in [0, \pi/2]$,
\begin{equation*}
\frac{\pi}{2}-\varphi
\le \frac{\pi}{2}\cos\varphi.
\end{equation*}
And by a change of variables $x = \sin^2 \varphi$ and the definition of beta
function,
\begin{equation*}
\begin{split}
\int_0^{\pi/2}\sin^{d-3}\varphi\cos^2\varphi\,d\varphi
&= \frac{1}{2}\int_0^{\pi/2}\sin^{d-4}\varphi\cos\varphi\,d\sin^2\varphi
= \frac{1}{2}\int_0^1 (1-x)^{1/2}x^{d/2-2}\,dx\\
&= \frac{1}{2}\rB\biggl(\frac{d}{2}-1,\frac{3}{2}\biggr)
= \frac{\Gamma\bigl(\frac{d}{2}-1\bigr)\Gamma\bigl(\frac{3}{2}\bigr)}
  {2\Gamma\bigl(\frac{d+1}{2}\bigr)}.
\end{split}
\end{equation*}

Therefore, $I_{2(b)}$ can be upper bounded by
\begin{equation*}
I_{2(b)}
\le \frac{1}{8\pi^2} \cdot
\frac{\Gamma\bigl(\frac{d}{2}\bigr)}
{(d-1)\sqrt{\pi}\Gamma\bigl(\frac{d-1}{2}\bigr)}
\cdot
\frac{\pi (d-2)\Gamma\bigl(\frac{d}{2}-1\bigr)\Gamma\bigl(\frac{3}{2}\bigr)}
{4\Gamma\bigl(\frac{d+1}{2}\bigr)}
= \frac{1}{64\pi} \biggl(\frac{\Gamma\bigl(\frac{d}{2}\bigr)}
                          {\Gamma\bigl(\frac{d+1}{2}\bigr)}\biggr)^2
\le \frac{1}{16\pi}\cdot \frac{1}{d},
\end{equation*}
where the last inequality is by \eqref{eq:wendel}.

For the lower bound on $I_{2(b)}$, we have
\begin{equation*}
\begin{split}
\int_0^{\pi/2} \frac{1}{2n+d-1}g(\varphi)\cos^{2n+1}\varphi\,d\varphi
&= \frac{d-2}{2n+d-1}\int_0^{\pi/2} \cos^{2n+1}\varphi
  \sin^{d-3}\varphi\cos\varphi\,d\varphi\\
&= \frac{d-2}{2n+d-1}
   \cdot\frac{1}{2}\rB\biggl(\frac{d}{2}-1, n+\frac{3}{2}\biggr)\\
&= \frac{d-2}{2n+d-1}\cdot
   \frac{\Gamma\bigl(\frac{d}{2}-1\bigr)\Gamma\bigl(n+\frac{3}{2}\bigr)}
   {2\Gamma\bigl(n+\frac{d}{2}+\frac{1}{2}\bigr)}\\
&\ge \frac{d-2}{2n+d+1}\cdot
     \frac{\Gamma\bigl(\frac{d}{2}-1\bigr)\Gamma\bigl(n+\frac{3}{2}\bigr)}
     {2\Gamma\bigl(n+\frac{d}{2}+\frac{1}{2}\bigr)}\\
&= \frac{\Gamma\bigl(\frac{d}{2}\bigr)\Gamma\bigl(n+\frac{3}{2}\bigr)}
   {2\Gamma\bigl(n+\frac{d}{2}+\frac{3}{2}\bigr)}
= \frac{1}{2}\rB\biggl(\frac{d}{2},n+\frac{3}{2}\biggr).
\end{split}
\end{equation*}
Hence, by the definition of beta function and a change of variables,
\begin{equation*}
\begin{split}
I_{2(b)}
&\ge \frac{1}{16\pi^2\zeta}\sum_{n=0}^\infty a_n
    \rB\biggl(\frac{d}{2},n+\frac{3}{2}\biggr)
= \frac{1}{8\pi^2\zeta}\sum_{n=0}^\infty a_n
  \int_0^{\pi/2}\cos^{2n+1}\psi\sin^{d-1}\psi\cos\psi\,d\psi\\
&= \frac{1}{8\pi^2\zeta}\int_0^{\pi/2}\sum_{n=0}^\infty a_n
   \cos^{2n+1}\psi\sin^{d-1}\psi\cos\psi\,d\psi\\
&= \frac{1}{8\pi^2\zeta}\int_0^{\pi/2}\biggl(\frac{\pi}{2}-\psi\biggr)
   \sin^{d-1}\psi\cos\psi\,d\psi.
\end{split}
\end{equation*}
Further, by the lower bound in \eqref{eq:theta_sin}, we obtain
\begin{equation*}
I_{2(b)}
\ge \frac{1}{8\pi^2\zeta}\int_0^{\pi/2}\sin^{d-1}\psi\cos^2\psi\,d\psi
= \frac{1}{16\pi^2\zeta}\rB\biggl(\frac{d}{2},\frac{3}{2}\biggr).
\end{equation*}
Inserting the definitions of $\zeta$ and the beta function gives
\begin{equation*}
I_{2(b)}
\ge \frac{1}{16\pi^2}
\cdot\frac{\Gamma\bigl(\frac{d}{2}\bigr)}
    {\sqrt{\pi}\Gamma\bigl(\frac{d-1}{2}\bigr)}
\cdot\frac{\Gamma(\frac{d}{2})\Gamma\bigl(\frac{3}{2}\bigr)}
    {\Gamma\bigl(\frac{d+3}{2}\bigr)}
= \frac{1}{32\pi^2}\cdot\frac{d-1}{d+1}
  \biggl(\frac{\Gamma\bigl(\frac{d}{2}\bigr)}
         {\Gamma\bigl(\frac{d+1}{2}\bigr)}\biggr)^2
\ge \frac{1}{64\pi^2}\biggl(\frac{\Gamma\bigl(\frac{d}{2}\bigr)}
         {\Gamma\bigl(\frac{d+1}{2}\bigr)}\biggr)^2.
\end{equation*}
Using \eqref{eq:wendel}, we have
\begin{equation*}
I_{2(b)} \ge \frac{1}{32\pi^2}\cdot\frac{1}{d}.
\end{equation*}

For the integral $I_3$ in \eqref{eq:sp_int},
by a change of variables $\xi = \theta-\psi$,
\begin{equation*}
\begin{split}
I_3 &= \int_0^{\pi/2}\int_0^{\pi/2}\int_{\theta-\pi/2}^{\pi/2}
\arcsin(\cos\varphi\cos(\theta-\psi))h(\theta,\psi,\varphi)
\,d\psi\,d\theta\,d\varphi\\
&=\int_0^{\pi/2}\int_0^{\pi/2}\int_{\pi/2}^{\theta-\pi/2}
-\arcsin(\cos\varphi\cos\xi)h(\theta,\psi,\varphi)
\,d\xi\,d\theta\,d\varphi\\
&=\int_0^{\pi/2}\int_0^{\pi/2}\int_{\theta-\pi/2}^{\pi/2}
\arcsin(\cos\varphi\cos\psi)h(\theta,\psi,\varphi)
\,d\psi\,d\theta\,d\varphi = I_2.
\end{split}
\end{equation*}

Combining the estimates of $I_1, I_2, I_3$ given above proves the claim.
\end{proof}
Plugging Lemma~\ref{le:Q1} into \eqref{eq:tau},
we arrive at the claims in Lemma~\ref{le:E_tau4_gn1/2}.

Using Lemma~\ref{le:gnpdq_gnpd} and Lemma~\ref{le:E_tau4_gn1/2}, we have
\begin{equation*}
\abs{\E_{\cG(n,1/2,d,q)}[\tau_{[4]}]} = q^6\abs{\E_{\cG(n,1/2,d,q)}[\tau_{[4]}]}
\le \frac{Cq^6}{d}.
\end{equation*}
Then,
\begin{equation*}
\abs{\E[\tau_4(\cG(n,1/2,d,q))]}
\le \binom{n}{4}\abs{\E_{\cG(n,1/2,d,q)}[\tau_{[4]}]}
\le \frac{Cn^4q^6}{d}.
\end{equation*}
Theorem~\ref{th:E_quad} is hence proved.

\subsubsection{General signed cliques}
We next turn to general signed cliques in $\cG(n,p,d,q)$.
Similarly, we start with estimations in $\cG(n,p,d)$.
By the definition of the signed clique,
\begin{equation*}
\begin{split}
\abs{\E_{\cG(n,p,d)}[\tau_{[k]}]}
&= \biggl\lvert\E\biggl[\prod_{\{i,j\} \subset [k]} (a_{i,j} - p)\biggr]
\biggr\rvert
= \biggl\lvert\E\biggl[\prod_{\{i,j\} \subset [k]} (a_{i,j} - p)\biggr]
  - \prod_{\{i,j\} \subset [k]} (p - p)\biggr\rvert\\
&= \biggl\lvert\sum_{S \in 2^{\binom{[k]}{2}}} (\P(S) - p^{\abs{S}})
(-p)^{\binom{k}{2} - \abs{S}}\biggr\rvert
\le \sum_{S \in 2^{\binom{[k]}{2}}} \lvert\P(S) - p^{\abs{S}}\rvert
p^{\binom{k}{2} - \abs{S}}.
\end{split}
\end{equation*}

The following corollary, derived from a result in Section~\ref{se:impo},
facilitates our calculations.
\begin{corollary} \label{co:edge_prob}
Let $V = [k]$ be a set of vertices and $E \subset V \times V$ be a set of edges.
Denote by $\lvert E\rvert$ the cardinality of $E$.
Then, we have that for a constant $C_{k,p}$,
\begin{equation*}
\lvert \P_{\cG(k,p,d)}(E) - p^{\lvert E \rvert}\rvert
\le \frac{C_{k,p}}{\sqrt{d}}.
\end{equation*}
\end{corollary}
\begin{proof}
In \eqref{eq:tv_upper}, by setting $q = 1$, we have that
\begin{equation}
\tv{\cG(n, p)}{\cG(n, p, d)} \le C_p\frac{n}{d}
+ C\biggl(\sqrt{\frac{n^2}{d}} + \sqrt{\frac{n^3}{d}}\biggr).
\end{equation}
By the definition of the distance,
\begin{equation*}
\begin{split}
\abs{\P_{\cG(k,p,d)}(E) - p^{\abs{E}}}
= \bigl\lvert \P_{\cG(k,p,d)}(E) - \P_{\cG(k,p)}(E) \bigr\rvert
\le \tv{\cG(k, p)}{\cG(k, p, d)}
\le C_p \sqrt{\frac{k^3}{d}}.
\end{split}
\end{equation*}
The claim directly follows.
\end{proof}

By Corollary~\ref{co:edge_prob},
\begin{equation} \label{eq:E_tauk_gnpd}
\abs{\E_{\cG(n,p,d)}[\tau_{[k]}]} \le \frac{C_{k,p}}{\sqrt{d}}.
\end{equation}
Therefore, for a constant $C_{k,p}$,
\begin{equation} \label{eq:E_tauk_gnpdq}
\abs{\E[\tau_k(\cG(n,p,d,q))]}
= q^{\binom{k}{2}} \abs{\E[\tau_k(\cG(n,p,d))]}
\le q^{\binom{k}{2}} \binom{n}{k} \abs{\E[\tau_{[k]}]}
\le \frac{C_{k,p}n^kq^{k(k-1)/2}}{\sqrt{d}}.
\end{equation}

\begin{lemma} \label{le:min_var_clique}
There exists a constant $C_{k,p} > 0$, depending only on $p$ and $k$, such that
\begin{equation*}
\Var[\tau_k(\cG(n,p,d,q))]
\ge C_{k,p} n^k.
\end{equation*}
\end{lemma}
\begin{proof}
Consider two sets of vertices $S$ and $S'$ of size $k$.
Since $(a_{i,j} - p)^2$ equals $(1-p)^2$ or $p^2$,
we have that
$(a_{i,j} - p)^2 \ge p^2(1-p)^2$.
Now if $S = S'$, then
\begin{equation*}
\E[\tau_{S}\tau_{S'}] = \E[(\tau_{S})^2]
= \E\biggl[\prod_{\{i,j\} \in V_1} (a_{i,j} - p)^2\biggr]
\ge (p^2(1-p)^2)^{\binom{k}{2}} = (p(1-p))^{k(k-1)}.
\end{equation*}
By \eqref{eq:E_tauk_gnpd}, there exists a $C'_{k,p} > 0$,
such that for $d \ge C'_{k,p}$,
\begin{equation*}
\E[\tau_S]^2 \le \frac{1}{2}\E[(\tau_S)^2].
\end{equation*}

For $S \ne S'$, let $V' = S \cap S'$ be the set of overlapping vertices.
Then, we have that
\begin{equation*}
\E[\tau_{S}\tau_{S'}]
= \E[\E[\tau_{S}\tau_{S'} \mid V']]
= \E[\E[\tau_{S} \mid V']\E[\tau_{S'} \mid V']]
= \E[\E[\tau_{S} \mid V']^2]
\ge \E[\tau_{S}]^2,
\end{equation*}
where the inequality is by Jensen's.

Therefore, there exists a $C_{k,p} > 0$, such that when $d \ge C_{k,p}$,
\begin{equation*}
\Var[\tau_k(\cG(n,p,d,q))]
= \E\biggl[\biggl(\sum_{S \in \binom{n}{k}} \tau_S\biggr)^2\biggr]
  - \biggl(\sum_{S \in \binom{n}{k}}\E[\tau_S]\biggr)^2
\ge \binom{n}{k} \Var[\tau_{[k]}]
\ge C_{k,p} n^k.\qedhere
\end{equation*}
\end{proof}

Putting them together, we have that for some constant $C_{k,p} < \infty$,
\begin{equation*}
\frac{(\E[\tau_k(\cG(n,p,d,q))]
- \E[\tau_k(\cG(n,p))])^2}
{\min \{\Var[\tau_k(\cG(n,p))],
\Var[\tau_k(\cG(n,p,d,q))]\}}
\le \frac{C_{k,p}n^k q^{k(k-1)}}{d}.
\end{equation*}

The above display implies that the method used to derive the possibility of
detection does not work when $n^kq^{k(k-1)}/d \to 0$,
which suggests a certain boundary of detection using general signed clique
statistics.
Note that for $k \geq 4$ this does not rule out the whole region
where signed triangles are not able to distinguish.
However, based on the computation of the expected signed quadruple count,
we see that the upper bound on the expectation in \eqref{eq:E_tauk_gnpdq}
is not precise;
in particular, the dependence on $d$ can be improved.
In general, we do not expect the detection boundary to be improved by
signed cliques.
Towards this, we present the following conjecture.

\begin{conjecture} \label{cj:E_tauk_gnpd}
There exists a constant $C_{k,p}$ such that
\begin{equation*}
\abs{\E_{\cG(n,p,d)}[\tau_{[k]}]} \le \frac{C_{k,p}}{d^{k/6}}.
\end{equation*}
\end{conjecture}
We briefly argue why this bound should hold.
In the proof of Theorem~\ref{th:det_upper},
we see that the dominating term comes from the ratio between the variance of
the signed triangle count in $\cG(n, p)$ and the squared expectation of the
signed triangle count in $\cG(n,p,d,q)$.
Suppose this still holds for general signed cliques.
Then, the dominating term if we use a signed clique statistic becomes
$C_{k,p}n^k/d^{2\alpha}$,
where $C_{k,p}/d^{\alpha}$ is a lower bound for
$\abs{\E_{\cG(n,p,d)}[\tau_{[k]}]}$.
Since $n^3/d$ is the precise order for the phase transition in $\cG(n,p,d)$,
we must have $k/(2\alpha) \le 3$,
which gives $\alpha \ge k/6$.
Note that this argument does not give a tight bound on the power of $d$,
as witnessed by the case of $k=4$,
when we know from Lemma~\ref{le:E_tau4_gn1/2} that
$\abs{\E_{\cG(n,1/2,d)}[\tau_{[4]}]}$ decays as~$C/d$.

In any case,
assuming Conjecture~\ref{cj:E_tauk_gnpd} holds,
and by the same arguments presented in this subsection,
we obtain that detection is not possible with this method if
$n^{3} q^{3(k-1)}/d\to 0$,
and this bound gets worse as $k$ grows.

\subsection{Signed cycles}
Let $S$ be a subset of $V$ and denote $k \coloneqq \abs{S}$.
Consider a cycle $C \subset \binom{S}{2}$, which is a set of edges forming a
closed chain.
There are $(k-1)!/2$ possible Hamilton cycles
(each vertex is visited exactly once) on $S$;
they are distributed identically to
$C^0 \coloneqq \{\{1,2\},\{2,3\},\ldots,\{k-1,k\},\{k,1\}\}$.
Denote by $K_C$ the indicator that the pairs in $C$ form a cycle of $G$.
Given the adjacency matrix $\A$ of $G$,
$K_C$ can be expressed by
\begin{equation}
K_C = \prod_{e \in C} a_e.
\end{equation}
Then, the total number of length $k$ cycles in $G$,
denoted by $K_k(G)$,
can be written as
\begin{equation}
K_k(G) = \sum_{C \subset \binom{S}{2}, S \in \binom{V}{k}} K_C.
\end{equation}
We similarly define the signed cycle and its count by
\begin{equation}
\kappa_C = \prod_{e \in C} (a_e - p)
\quad \text{and} \quad
\kappa_k(G) = \sum_{C \subset \binom{S}{2}, S \in \binom{V}{k}} \kappa_C.
\end{equation}

We again start with estimating the expectation and variance of
the signed cycle statistic in $\cG(n,p)$.

For $\cG(n,p)$, again by independence of edges,
a signed cycle has expectation zero:
\begin{equation*}
\E[\kappa_{C^0}] = \prod_{e \in C^0}\E[a_e - p] = 0.
\end{equation*}
Hence, the expectation of the signed length $k$ cycle statistic in $\cG(n,p)$
is also zero:
\begin{equation} \label{eq:E_kappak_gnp}
\E[\kappa_k(\cG(n,p))] = \binom{n}{k}\frac{(k-1)!}{2}\E[\kappa_{C^0}] = 0.
\end{equation}
Consider two cycles $C$ and $C'$ of length $k$.
If $C = C'$, then
\begin{equation*}
\E[\kappa_{C}\kappa_{C'}]
= \E[(\kappa_{C})^2]
= \E\biggl[\prod_{e \in C} (a_e - p)^2\biggr]
= \prod_{e \in C} \E[(a_e - p)^2]
= (p(1-p))^k.
\end{equation*}
For $C \ne C'$, there exists at least one edge $e$ that is in $C$
but not in $C'$.
Hence,
\begin{equation*}
\E[\kappa_{C}\kappa_{C}]
= \E[a_e - p]
  \E\biggl[\kappa_{C'}\prod_{e'\in C\backslash\{e\}}(a_{e'}-p)\biggr]
= 0.
\end{equation*}
Therefore, for some $C_{k,p} > 0$,
\begin{equation} \label{eq:V_kappak_gnp}
\begin{split}
\Var[\kappa_k(\cG(n,p))]
&= \E\biggl[\biggl(\sum_{C \subset \binom{S}{2}, S \in \binom{V}{k}}
\kappa_C\biggr)^2\biggr]
= \frac{n!}{(n-k)!2k} \E[(\kappa_{C^0})^2]
= \frac{n!}{(n-k)!2k} (p(1-p))^k\\
&\ge C_{k,p} n^k.
\end{split}
\end{equation}

In order to estimate the mean and variance of the signed cycle statistic
in $\cG(n,p,d,q)$,
we additionally need the following lemma concerning the probability of
an open path (i.e., an open chain of edges) in $\cG(n,p,d)$.
\begin{lemma} \label{le:path_prob}
In $\cG(n,p,d)$, any open path of length $k$ has probability $p^k$.
\end{lemma}
\begin{proof}
We prove the claim by induction on the length.
For each edge in $\cG(n,p,d)$,
by definition we have $\P(i \sim j) = \E[a_{i,j}] = p$.
Suppose we have a path $P_{k+1} \coloneqq \{\{1,2\},\{2,3\},\ldots,\{k,k+1\}\}$
on vertices $[k+1]$.
Then, the probability of the path is
\begin{equation*}
\P(P_{k+1}) = \E\biggl[\prod_{i=1}^k a_{i,i+1}\biggr].
\end{equation*}
By conditional independence of the edges and rotation symmetry on sphere,
\begin{equation*}
\begin{split}
\E\biggl[\prod_{i=1}^k a_{i,i+1}\biggr]
&= \E\biggl[\E\biggl[\prod_{i=1}^k a_{i,i+1} \biggm\vert \x_k \biggr]\biggr]
= \E\biggl[\E\biggl[\prod_{i=1}^{k-1} a_{i,i+1} \biggm\vert \x_k \biggr]
            \E[a_{k,k+1} \mid \x_k]\biggr]\\
&= p\E\biggl[\prod_{i=1}^{k-1} a_{i,i+1}\biggr]
= p\P(P_k).\qedhere
\end{split}
\end{equation*}
\end{proof}
Expanding the product of a signed cycle,
\begin{equation*}
\E_{\cG(n,p,d)}\bigl[\kappa_C\bigr]
= \sum_{S \in 2^{C}} (-p)^{k-\abs{S}}
\biggl(\E\biggl[\prod_{e \in S}a_e\biggr]-p^{\abs{S}}\biggr).
\end{equation*}
Since $C$ is a cycle, all proper subsets of $C$ are a union of independent
paths.
By Lemma~\ref{le:path_prob},
\begin{equation*}
\E\biggl[\prod_{e \in S}a_e\biggr] = \P(S) = p^{\abs{S}}
\end{equation*}
for all $S \subset C$ except for $S = C$.
Hence, the expectation of a signed length $k$ cycle in $\cG(n,p,d)$ satisfies
\begin{equation} \label{eq:E_kappak_gnpd}
\E_{\cG(n,p,d)}[\kappa_C] = \E\biggl[\prod_{e\in C}a_e\biggr] - p^k.
\end{equation}

\subsubsection{Signed quadrilaterals}
We start with the expected number of signed cycles of length four,
which are called \emph{signed quadrilaterals}.
By \eqref{eq:E_kappak_gnpd}, the expectation of a signed quadrilateral $C^0$
in $\cG(n,1/2,d)$ is
\begin{equation*}
\E[\kappa_{C^0}]
= \E[\bra_{1,2}\bra_{2,3}\bra_{3,4}\bra_{4,1}]
= \E[a_{1,2}a_{2,3}a_{3,4}a_{4,1}] - \biggl(\frac{1}{2}\biggr)^4.
\end{equation*}
Using Lemma~\ref{le:sqr}, we have that
\begin{equation*}
\E[\kappa_{C^0}] = 2\eta.
\end{equation*}
Hence, by \eqref{eq:eta_bounds},
\begin{equation*}
\frac{1}{2\pi^2}\cdot\frac{1}{d}
\le \E_{\cG(n,1/2,d)}[\kappa_{C^0}]
\le \frac{1}{8}\cdot\frac{1}{d}.
\end{equation*}
Therefore, for absolute constants $C > 0$ and $C' < \infty$, we have that
\begin{equation*}
\frac{Cn^4q^4}{d} \le \E[\kappa_4(\cG(n,1/2,d,q))] \le \frac{C'n^4q^4}{d}.
\end{equation*}
Together with \eqref{eq:E_kappak_gnp}, we have that
\begin{equation*}
\abs{\E[\kappa_4(\cG(n,1/2,d,q))]-\E[\kappa_4(\cG(n,1/2))]}
\le \frac{Cn^4q^4}{d}.
\end{equation*}
By \eqref{eq:V_kappak_gnp}, we also have that
\begin{equation*}
\max \{\Var[\kappa_4(\cG(n,1/2))], \Var[\kappa_4(\cG(n,1/2,d,q))]\}
\ge \Var[\kappa_4(\cG(n,1/2))]
\ge C n^4
\end{equation*}
for an absolute constant $C > 0$.

Therefore, there is an absolute constant $C < \infty$ such that
\begin{equation*}
\frac{(\E[\kappa_4(\cG(n,1/2,d,q))]-\E[\kappa_4(\cG(n,1/2))])^2}
{\max \{\Var[\kappa_4(\cG(n,1/2))], \Var[\kappa_4(\cG(n,1/2,d,q))]\}}
\le \frac{Cn^4 q^{8}}{d^2}.
\end{equation*}
This implies that if detection is possible using this method,
then we should have $n^{2}q^{4}/d \to \infty$.
This is worse than the condition $n^{3}q^{6}/d \to \infty$
under which signed triangles can detect.

\subsubsection{General signed cycle}
Next we estimate the signed length $k$ cycle count in $\cG(n,p,d,q)$
with the help of Lemma~\ref{le:path_prob}.

By Corollary~\ref{co:edge_prob}, the probability of a cycle satisfies
\begin{equation*}
\lvert\P(C^0) - p^k\rvert \le \frac{C_{k,p}}{\sqrt{d}}.
\end{equation*}
Hence, by \eqref{eq:E_kappak_gnpd},
\begin{equation}\label{eq:signed_cycle_bd}
\abs{\E_{\cG(n,p,d)}[\kappa_{C^0}]}
\le \frac{C_{k,p}}{\sqrt{d}}.
\end{equation}
Thus,
\begin{equation*}
\abs{\E[\kappa_k(\cG(n,p,d,q))]}
= \abs{q^k \E[\kappa(\cG(n,p,d))]}
\le \frac{q^k n!}{(n-k)!2k}
\abs{\E_{\cG(n,p,d)}[\kappa_{C^0}]}
\le \frac{C_{k,p}n^kq^k}{\sqrt{d}}
\end{equation*}
for a constant $C_{k,p}$.

Using \eqref{eq:V_kappak_gnp},
\begin{equation*}
\max \{\Var[\kappa_k(\cG(n,p))], \Var[\kappa_k(\cG(n,p,d,q))]\}
\ge \Var[\kappa_k(\cG(n,p))]
\ge C_{k,p} n^k.
\end{equation*}

Therefore, for some $C_{k,p} > 0$,
\begin{equation*}
\frac{(\E[\kappa_k(\cG(n,p,d,q))]-\E[\kappa_k(\cG(n,p))])^2}
{\max \{\Var(\kappa_k(\cG(n,p))),\Var(\kappa_k(\cG(n,p,d,q)))\}}
\le \frac{C_{k,p}n^k q^{2k}}{d}.
\end{equation*}
The above display implies that detection is not possible using the previous
method with a signed length $k$ cycle statistic when $n^kq^{2k}/d \to 0$.
Note that for $k \geq 4$ this
does not rule out all regions not detectable by signed
triangles.
However, based on the computations for signed quadrilaterals,
we believe that the dependence on $d$
in the bound in~\eqref{eq:signed_cycle_bd} is not tight.
Analogously to Conjecture~\ref{cj:E_tauk_gnpd},
we have the following conjecture.

\begin{conjecture} \label{cj:E_kappak_gnpd}
There exists a constant $C_{k,p} < \infty$ such that
\begin{equation*}
\abs{\E_{\cG(n,p,d)}[\kappa_{C^0}]} \le \frac{C_{k,p}}{d^{k/6}}.
\end{equation*}
\end{conjecture}

Assuming Conjecture~\ref{cj:E_kappak_gnpd} holds,
detection is not possible with this method when $n^{3}q^{6}/d \to 0$.
This would imply that all signed cycles have the same detecting power.
However, as witnessed by signed quadrilaterals,
the above conjecture is not tight,
suggesting that signed triangles yield the best bound.

\section{Random dot product graphs} \label{se:rdpg}
The starting point of this paper is the random geometric graph
$\cG(n,p,d)$, where the underlying points $\x_1, \ldots, \x_n \in \RR^d$
are distributed uniformly on the unit sphere.
Our main object of study, $\cG(n,p,d,q)$, builds upon $\cG(n,p,d)$.
Note that in $\cG(n,p,d)$ there is an edge between two nodes
if and only if the dot product of the corresponding latent vectors
is greater than some threshold (see~\eqref{eq:dot_product}).
Models with this property are known as
\emph{random dot product graphs}
and have been widely studied~\cite{athreya2018statistical,smith2019geometry}.

A natural variant of $\cG(n,p,d)$ is to take
$\x_1, \ldots, \x_n \in \RR^d$ to be i.i.d.\ standard normal vectors
and to consider the corresponding random dot product graph.
In this section we extend our results to this variant;
the proofs are kept brief, highlighting only the differences.

Let $\x_1, \ldots, \x_n \in \RR^d$ be i.i.d.\ standard normal vectors
and define the threshold $u_{p,d}$ by
\begin{equation*}
\P(\inner{\x_i}{\x_j} \ge u_{p,d}) = p.
\end{equation*}
Consider the connection function
\begin{equation*}
\phi_q(x) = (1 - q)p + qs_{u_{p,d}}(x).
\end{equation*}
We denote the random graph generated using $\x_1, \ldots, \x_n \in \RR^d$ and
this connection function by $\dG(n,p,d,q)$.
The results of Theorem~\ref{th:det} also hold under this setting,
which we state as the following theorem.
\begin{theorem}[Detecting geometry]\label{th:det_dp}
Let $p \in (0, 1)$ be fixed.
\begin{thenum}
\item\label{th:det_dp_lower}(Impossibility)
If $n q \to 0$ or $n^3 q^2 / d \to 0$, then
\begin{equation*}
\tv{\cG(n, p)}{\dot{\cG}(n, p, d, q)} \to 0.
\end{equation*}
\item\label{th:det_dp_upper}(Possibility)
If $n^3 q^6 / d \to \infty$, then
\begin{equation*}
\tv{\cG(n, p)}{\dot{\cG}(n, p, d, q)} \to 1.
\end{equation*}
\end{thenum}
\end{theorem}
\begin{remark}
There are two main differences between $\cG(n,p,d,q)$ and $\dG(n,p,d,q)$:
\begin{enumerate}[label=(\arabic*)]
\item First, the coordinates of $\x_i \in \RR^d$ are independent in
$\dG(n,p,d,q)$; this is not the case for $\cG(n,p,d,q)$.
\item On the other hand, while $\x_1, \ldots, \x_n \in \RR^d$ have
rotation symmetry in both models,
in $\dG(n,p,d,q)$ these vectors no longer have unit norm,
which must be accounted for.
\end{enumerate}
The first property simplifies the proof of Theorem~\ref{th:det_dp_lower};
however,
the second one adds complexity to the proof of Theorem~\ref{th:det_dp_upper}.
\end{remark}
Since the proofs are quite similar to those for Theorem~\ref{th:det},
we only sketch them, highlighting the important adaptations.

\subsection{Proof of part (a)}
The proof in the regime $nq \to 0$,
which is presented in Section~\ref{se:weak},
only uses the property that $\P(\inner{\x_i}{\x_j} \ge t_{p,d}) = p$.
Hence, it holds directly for $\dG(n,p,d,q)$ as well.

Now we consider the other regime.
If we define $\R \coloneqq (1 - q) s_{t_p}(\M) + q s_{u_{p,d}}(\Z\Z^\top)$,
following the same arguments, we have
\begin{equation*}
\tv{\cG(n, p)}{\dG(n, p, d, q)} \le \tv{\bP}{\R}.
\end{equation*}
Note that for $1 \le i < j \le n$,
\begin{equation*}
\inner{\z_i}{\z_j} = w_{i,j} = \sqrt{d} y_{i,j}.
\end{equation*}
Then we have $\R = (1 - q) s_{t_p}(\M) + q s_{u_{p,d}/\sqrt{d}}(\Y)$.
Let $ \H \coloneqq (\1\1^\top - \B) \circ \M
+ \B \circ (\Y + (t_p - u_{p,d}/\sqrt{d})\1\1^\top)$.
We can implement the same procedure and obtain the upper bound on the distance,
which gives
\begin{equation*}
\tv{\bP}{\R} \le \E_{\B} \tv{\M'}{\H}.
\end{equation*}
The challenge is that we have to bound $\lvert u_{p,d}/\sqrt{d} - t_p\rvert$
from above.
The following lemma can be derived as a corollary of Lemma~13.4 in
\cite{eldan2020information}.
\begin{lemma} \label{le:upd}
There exists a constant $C_p < \infty$, depending only on $p$, such that
\begin{equation*}
\lvert u_{p,d}/\sqrt{d} - t_p \rvert \le \frac{C_p}{\sqrt{d}}.
\end{equation*}
\end{lemma}
Plugging this estimate into the proof,
we conclude that there exist constants $C, C_p$, such that for $d \ge 2n$,
\begin{equation*}
\tv{\cG(n, p)}{\dG(n, p, d, q)} \le C_p \sqrt{\frac{n^2q}{d}}
+ C\biggl(\sqrt{\frac{n^2q}{d}} + \sqrt{\frac{n^3q^2}{d}}\biggr),
\end{equation*}
similarly to~\eqref{eq:tv_upper}. The conclusion follows.

\subsection{Proof of part (b)}
In the case $p = 1/2$, the random dot product graph is essentially the same
as the random geometric graph.
Hence, the results directly carry over.

For a fixed value of $p \in (0,1)$,
the proofs are similar but require the estimation of several quantities under a
different setting.

We first present technical lemmas for bounding the probabilities of
$E^\Lambda$ and $E^\Delta$ defined as counterparts of
\eqref{eq:cherry} and \eqref{eq:delta} respectively,
which are derived as corollaries from \cite{eldan2020information}.
\begin{lemma}[{Corollary of \cite[Lemma~13.10]{eldan2020information}}]
\label{le:Gdp_P_L}
For a fixed $p \in (0, 1)$, we have
\begin{equation*}
\P(E^\Lambda) - p^2 \le \frac{8}{d}.
\end{equation*}
\end{lemma}
\begin{lemma}[{Corollary of \cite[Theorem~13.5]{eldan2020information}}]
\label{le:Gdp_P_D}
For a fixed $p \in (0, 1)$, there exist some constants $C_p, C_p', C_p'' > 0$
depending only on $p$ such that for $d \ge C_p$,
\begin{equation*}
\frac{C_p'}{\sqrt{d}} \le \P(E^\Delta) - p^3
\le \frac{C_p''}{\sqrt{d}}.
\end{equation*}
\end{lemma}
Plugging Lemma~\ref{le:Gdp_P_L} and Lemma~\ref{le:Gdp_P_D} into the estimate for
$\E[\tau_{\{1,2,3\}}]$,
we obtain the same lower bounds as for $\cG(n,p,d,q)$.

For estimating the variance,
we still have
\begin{equation*}
V_{\{1,2,3\},\{4,5,6\}} = 0 \quad
\text{and} \quad V_{\{1,2,3\},\{1,2,3\}} \le 1.
\end{equation*}
However, bounding $V_{\{1,2,3\},\{1,2,4\}}$ and $V_{\{1,2,3\},\{1,4,5\}}$
requires the following estimate,
which follows from the proof of~\cite[Lemma~13.11]{eldan2020information}:
\begin{equation} \label{eq:Gdp_E_s23_s31}
\E_{\x_1,\x_2}[\E_{\x_3}[\brs_{2,3}\brs_{3,1}]^2]
\le \frac{80}{d}.
\end{equation}
Replacing \eqref{eq:E_s23_s31} with \eqref{eq:Gdp_E_s23_s31},
we obtain the same result for $V_{\{1,2,3\},\{1,2,4\}}$.

For two signed triangles sharing only one vertex,
we have
\begin{equation*}
\begin{split}
\E[\tau_{\{1,2,3\}}\tau_{\{1,4,5\}}]
&= \E[\E[\tau_{\{1,2,3\}}\tau_{\{1,4,5\}} \mid \x_1]]
= \E[\E[\tau_{\{1,2,3\}} \mid \x_1]^2]\\
&= \E_{\x_1}[\E_{\x_2,\x_3}[\E[\bra_{1,2}\bra_{2,3}\bra_{3,1}
   \mid \x_1, \x_2, \x_3]]^2]
= q^6 \E_{\x_1}[\E_{\x_2,\x_3}[\brs_{1,2}\brs_{2,3}\brs_{3,1}]^2].
\end{split}
\end{equation*}
By Jensen's inequality,
\begin{equation*}
\begin{split}
\E_{\x_1}[\E_{\x_2,\x_3}[\brs_{1,2}\brs_{2,3}\brs_{3,1}]^2]
&\le \E_{\x_1,\x_2}[\E_{\x_3}[\brs_{1,2}\brs_{2,3}\brs_{3,1}]^2]
= \E_{\x_1,\x_2}[\brs_{1,2}^2\E_{\x_3}[\brs_{2,3}\brs_{3,1}]^2]\\
&\le \E_{\x_1,\x_2}[\E_{\x_3}[\brs_{2,3}\brs_{3,1}]^2].
\end{split}
\end{equation*}
Using \eqref{eq:Gdp_E_s23_s31} again, we obtain
\begin{equation*}
\E[\tau_{\{1,2,3\}}\tau_{\{1,4,5\}}] \le \frac{80q^6}{d}.
\end{equation*}
Therefore, similarly
\begin{equation*}
V_{\{1,2,3\},\{1,4,5\}} \le \E[\tau_{\{1,2,3\}}\tau_{\{1,4,5\}}]
\le \frac{80q^6}{d}.
\end{equation*}

Inserting the above estimates into~\eqref{eq:Var_tau3},
we conclude that for an absolute constant $C$,
\begin{equation*}
\Var[\tau_3(\dG(n,p,d,q))]
\le C\biggl(n^3 + \frac{n^4q^4}{d} + \frac{n^5q^6}{d}\biggr).
\end{equation*}

From the previous analyses,
for a fixed $p \in (0, 1)$,
there exist constants $C_p > 0$ and $C < \infty$ such that
\begin{equation*}
\E_{\cG(n,p)}[\tau_3(G)] = 0, \quad
\E_{\dG(n,p,d,q)}[\tau_3(G)] \ge \frac{C_pn^3q^3}{\sqrt{d}}
\end{equation*}
and
\begin{equation*}
V_m \coloneqq \max \{\Var[\tau_3\bigl(\cG(n, p))],
\Var[\tau_3(\dG(n, p, d, q))]\}
\le C\biggl(n^3 + \frac{n^4 q^4}{d} + \frac{n^5q^6}{d}\biggr).
\end{equation*}
Repeating the same arguments as before,
we conclude that there exists a constant $C_p < \infty$
such that for $d \ge C_p$,
\begin{equation*}
\tv{\cG(n, p)}{\dG(n, p, d, q)}
\ge 1 - C_p\biggl(\frac{d}{n^3q^6} + \frac{1}{n^2q^2} + \frac{1}{n}\biggr).
\end{equation*}

\section*{Acknowledgement}
S.~L. thanks Cong Ma, Kaizheng Wang, and Jiacheng Zhang for numerous insightful
discussions throughout the years.
The authors acknowledge the generous support from NSF grant DMS 1811724.

\bibliographystyle{plainnat}

\end{document}